\documentclass[a4paper,10pt]{amsart}

\usepackage{amsfonts,amssymb}
\usepackage[all,2cell]{xy} \UseTwocells

\input{xy}
\xyoption{all}

\def\mapright#1{\smash{\mathop{\longrightarrow}\limits^{#1}}}

\newtheorem{proposition}{Proposition}[section]
\newtheorem{lemma}[proposition]{Lemma}
\newtheorem{corollary}[proposition]{Corollary}
\newtheorem{theorem}[proposition]{Theorem}
\theoremstyle{definition}
\newtheorem{definition}[proposition]{Definition}
\newtheorem{definitions}[proposition]{Definitions}

\theoremstyle{remark}
\newtheorem{remark}[proposition]{Remark}
\newtheorem{remarks}[proposition]{Remarks}
\newcommand{\thlabel}[1]{\label{th:#1}}
\newcommand{\thref}[1]{Theorem~\ref{th:#1}}
\newcommand{\selabel}[1]{\label{se:#1}}
\newcommand{\seref}[1]{Section~\ref{se:#1}}
\newcommand{\lelabel}[1]{\label{le:#1}}
\newcommand{\leref}[1]{Lemma~\ref{le:#1}}
\newcommand{\prlabel}[1]{\label{pr:#1}}
\newcommand{\prref}[1]{Proposition~\ref{pr:#1}}
\newcommand{\colabel}[1]{\label{co:#1}}
\newcommand{\coref}[1]{Corollary~\ref{co:#1}}

\newcommand{\delabel}[1]{\label{de:#1}}

\newcommand{\deslabel}[1]{\label{de:#1}}
\newcommand{\desref}[1]{Definitions~\ref{de:#1}}
\newcommand{\eqlabel}[1]{\label{eq:#1}}
\newcommand{\equref}[1]{(\ref{eq:#1})}

\def\ra{\rightarrow}
\def\cd{\cdot}

\def\Id{{\rm Id}}

\def\mfM{{\mf M}}
\def\mfN{\mf {N}}
\def\mfP{\mf {P}}
\def\mfC{{\mf C}}

\newcommand{\gbrbox}{\mbox{$\gbr\gvac{-1}\gnot{\hspace*{-4mm}\Box}$}}
\def\ot{\otimes}
\def\va{\varepsilon}

\def\un{\underline}

\def\mf{\mathfrak}

\def\mfA{\mf {A}}
\def\mfa{\mf {a}}
\def\mfm{\mf {m}}
\def\mfn{\mf {n}}
\def\le{\langle}
\def\ri{\rangle}
\def\l{\lambda}

\def\r{\rho}

\def\va{\varepsilon}
\def\v{\varphi}
\def\tl{\triangleleft}
\def\tr{\triangleright}
\def\btr{\blacktriangleright}
\def\btl{\blacktriangleleft}

\def\ra{\rightarrow}

\def\d{\delta}

\def\ov{\overline}
\def\cal{\mathcal}
\def\un{\underline}

\newcommand{\une}{\mbox{$\un {\va }$}}
\newcommand{\una}{\mbox{$c_{\un {1}}$}}
\newcommand{\unb}{\mbox{$c_{\un {2}}$}}

\newcommand{\Cc}{\cal C}
\newcommand{\Dc}{\cal D}
\newcommand{\Mm}{\cal M}

\def\equal#1{\smash{\mathop{=}\limits^{#1}}}

\def\equalupdown#1#2{\smash{\mathop{=}\limits^{#1}\limits_{#2}}}

 \newcount\grcalca
 \newcount\grcalcb
 \newcount\grcalcc
 \newcount\grcalcd
 \newcount\grcalce
 \newcount\grcalcf
 \newcount\grcalcg
 \newcount\grcalch
 \newcount\grcalci
 \newcount\grrow
 \newcount\grcolumn
 \newcount\grwidth
 \newcount\factor
 \newcount\hfactor
 \newcount\qfactor
 \newcount\tfactor
 \newcount\sfactor
 \newcount\dfactor
 \hfactor = \factor
 \divide    \hfactor by 2
 \qfactor = \hfactor
 \divide    \qfactor by 2
 \tfactor = \factor
 \divide    \tfactor by 3
 \sfactor = \factor
 \divide    \sfactor by 6
 \dfactor = \factor
 \divide    \dfactor by 12

 \newcommand{\gbeg}[2]{
   \unitlength=1pt
   \grrow = #2
   \grcolumn = 0
   \grcalca = #1
   \grcalcb = #2
   \multiply \grcalca by \factor
   \grwidth = \grcalca
   \multiply \grcalcb by \factor
   \begin{minipage}{\grcalca pt}
   \begin{picture}(\grcalca,\grcalcb)
   \advance \grcalcb by -\factor
   \put(0, \grcalcb){\line(1,0){\grwidth}} }

 \newcommand{\gend}{
   \put(0, \factor){\line(1,0){\grwidth}}
   \end{picture}
   {\vskip2.5ex}
   \end{minipage} }

 \newcommand{\gnl}{
   \advance \grrow by -1
   \grcolumn = 0}

 \newcommand{\gvac}[1]{       
   \advance \grcolumn by #1} 
 
 \newcommand{\gcl}[1]{
   \grcalca = \grcolumn
   \multiply \grcalca by \factor
   \advance \grcalca by \hfactor
   \grcalcb = \grrow
   \multiply \grcalcb by \factor
   \grcalcc = #1
   \multiply \grcalcc by \factor
   \put(\grcalca,\grcalcb) {\line(0,-1){\grcalcc}} 
   \advance \grcolumn by 1}

 \newcommand{\gcn}[4]{
   \grcalca = \grcolumn
   \multiply \grcalca by \factor
   \grcalci = #3
   \multiply \grcalci by \hfactor
   \advance \grcalca by \grcalci
   \grcalcb = \grcolumn
   \multiply \grcalcb by \factor 
   \grcalci = #3
   \advance \grcalci by #4
   \multiply \grcalci by \qfactor
   \advance \grcalcb by \grcalci
   \grcalcc = \grcolumn
   \multiply \grcalcc by \factor
   \grcalci = #4
   \multiply \grcalci by \hfactor
   \advance \grcalcc by \grcalci
   \grcalcd = \grrow
   \multiply \grcalcd by \factor 
   \grcalce = \grrow
   \multiply \grcalce by \factor 
   \grcalci = #2
   \multiply \grcalci by \tfactor
   \advance \grcalce by -\grcalci
   \grcalcf = \grrow
   \multiply \grcalcf by \factor 
   \grcalci = #2
   \multiply \grcalci by \hfactor
   \advance \grcalcf by -\grcalci
   \grcalcg = \grrow
   \multiply \grcalcg by \factor 
   \grcalci = #2
   \multiply \grcalci by \tfactor
   \multiply \grcalci by 2
   \advance \grcalcg by -\grcalci
   \grcalch = \grrow
   \advance \grcalch by -#2
   \multiply \grcalch by \factor 
   \qbezier(\grcalca,\grcalcd)(\grcalca,\grcalce)(\grcalcb,\grcalcf) 
   \qbezier(\grcalcb,\grcalcf)(\grcalcc,\grcalcg)(\grcalcc,\grcalch) 
   \advance \grcolumn by #1}

 \newcommand{\gnot}[1]{
   \grcalca = \grcolumn
   \multiply \grcalca by \factor
   \advance \grcalca by \hfactor
   \grcalcb = \grrow
   \multiply \grcalcb by \factor
   \advance \grcalcb by -\hfactor
   \put(\grcalca,\grcalcb) {\makebox(0,0){$\scriptstyle #1$}} }

 \newcommand{\got}[2]{
   \grcalca = \grcolumn
   \multiply \grcalca by \factor
   \grcalcc = #1
   \multiply \grcalcc by \hfactor
   \advance \grcalca by \grcalcc
   \grcalcb = \grrow
   \multiply \grcalcb by \factor
   \advance \grcalcb by -\tfactor
   \advance \grcalcb by -\tfactor
   \put(\grcalca,\grcalcb){\makebox(0,0)[b]{$#2$}}
   \advance \grcolumn by #1}

 \newcommand{\gob}[2]{
   \grcalca = \grcolumn
   \multiply \grcalca by \factor
   \grcalcc = #1
   \multiply \grcalcc by \hfactor
   \advance \grcalca by \grcalcc
   \put(\grcalca,0){\makebox(0,0)[b]{$#2$}}
   \advance \grcolumn by #1}

 \newcommand{\gmu}{  
   \grcalca = \grcolumn
   \advance \grcalca by 1
   \multiply \grcalca by \factor
   \grcalcb = \grrow
   \multiply \grcalcb by \factor
   \grcalcc = \factor
   \advance \grcalcc by \hfactor
   \put(\grcalca,\grcalcb){\oval(\factor,\grcalcc)[b]}
   \advance \grcalcb by -\hfactor
   \advance \grcalcb by -\qfactor
   \put(\grcalca,\grcalcb) {\line(0,-1){\qfactor}} 
   \advance \grcolumn by 2}

 \newcommand{\gcmu}{   
   \grcalca = \grcolumn
   \advance \grcalca by 1
   \multiply \grcalca by \factor
   \grcalcb = \grrow
   \advance \grcalcb by -1
   \multiply \grcalcb by \factor
   \grcalcc = \factor
   \advance \grcalcc by \hfactor
   \put(\grcalca,\grcalcb){\oval(\factor,\grcalcc)[t]}
   \advance \grcalcb by \factor
   \put(\grcalca,\grcalcb) {\line(0,-1){\qfactor}} 
   \advance \grcolumn by 2}

 \newcommand{\glm}{
   \grcalca = \grcolumn
   \multiply \grcalca by \factor
   \advance \grcalca by \hfactor
   \grcalcb = \grcalca
   \advance \grcalcb by \factor
   \grcalcc = \grrow
   \multiply \grcalcc by \factor
   \grcalcd = \grcalcc
   \advance \grcalcd by -\tfactor
   \grcalce = \grcalcd
   \advance \grcalce by -\tfactor
   \put(\grcalca, \grcalcc){\line(0,-1){\tfactor}}
   \put(\grcalca, \grcalcd){\line(1,0){\factor}}
   \put(\grcalca, \grcalcd){\line(3,-1){\factor}}
   \put(\grcalcb, \grcalcc){\line(0,-1){\factor}}
   \advance \grcolumn by 2}

 \newcommand{\grm}{
   \grcalcb = \grcolumn
   \multiply \grcalcb by \factor
   \advance \grcalcb by \hfactor
   \grcalca = \grcalcb
   \advance \grcalca by \factor
   \grcalcc = \grrow
   \multiply \grcalcc by \factor
   \grcalcd = \grcalcc
   \advance \grcalcd by -\tfactor
   \grcalce = \grcalcd
   \advance \grcalce by -\tfactor
   \put(\grcalca, \grcalcc){\line(0,-1){\tfactor}}
   \put(\grcalca, \grcalcd){\line(-1,0){\factor}}
   \put(\grcalca, \grcalcd){\line(-3,-1){\factor}}
   \put(\grcalcb, \grcalcc){\line(0,-1){\factor}}
   \advance \grcolumn by 2}

 \newcommand{\glcm}{
   \grcalca = \grcolumn
   \multiply \grcalca by \factor
   \advance \grcalca by \hfactor
   \grcalcb = \grcalca
   \advance \grcalcb by \factor
   \grcalcc = \grrow
   \advance \grcalcc by -1
   \multiply \grcalcc by \factor
   \grcalcd = \grcalcc
   \advance \grcalcd by \tfactor
   \grcalce = \grcalcd
   \advance \grcalce by \tfactor
   \put(\grcalca, \grcalcc){\line(0,1){\tfactor}}
   \put(\grcalca, \grcalcd){\line(1,0){\factor}}
   \put(\grcalca, \grcalcd){\line(3,1){\factor}}
   \put(\grcalcb, \grcalcc){\line(0,1){\factor}}
   \advance \grcolumn by 2}

 \newcommand{\grcm}{
   \grcalcb = \grcolumn
   \multiply \grcalcb by \factor
   \advance \grcalcb by \hfactor
   \grcalca = \grcalcb
   \advance \grcalca by \factor
   \grcalcc = \grrow
   \advance \grcalcc by -1
   \multiply \grcalcc by \factor
   \grcalcd = \grcalcc
   \advance \grcalcd by \tfactor
   \grcalce = \grcalcd
   \advance \grcalce by \tfactor
   \put(\grcalca, \grcalcc){\line(0,1){\tfactor}}
   \put(\grcalca, \grcalcd){\line(-1,0){\factor}}
   \put(\grcalca, \grcalcd){\line(-3,1){\factor}}
   \put(\grcalcb, \grcalcc){\line(0,1){\factor}}
   \advance \grcolumn by 2}

 \newcommand{\gwmu}[1]{    
   \grcalca = \grcolumn
   \multiply \grcalca by \factor
   \grcalcd = \hfactor
   \multiply \grcalcd by #1
   \advance \grcalca by \grcalcd
   \grcalcb = \grrow
   \multiply \grcalcb by \factor
   \grcalcc = \factor
   \advance \grcalcc by \hfactor
   \grcalcd = #1
   \advance \grcalcd by -1
   \multiply \grcalcd by \factor
   \put(\grcalca,\grcalcb){\oval(\grcalcd,\grcalcc)[b]}
   \advance \grcalcb by -\hfactor
   \advance \grcalcb by -\qfactor
   \put(\grcalca,\grcalcb) {\line(0,-1){\qfactor}} 
   \advance \grcolumn by #1}

 \newcommand{\gwcm}[1]{   
   \grcalca = \grcolumn
   \multiply \grcalca by \factor
   \grcalcd = \hfactor
   \multiply \grcalcd by #1
   \advance \grcalca by \grcalcd
   \grcalcb = \grrow
   \advance \grcalcb by -1
   \multiply \grcalcb by \factor
   \grcalcc = \factor
   \advance \grcalcc by \hfactor
   \grcalcd = #1
   \advance \grcalcd by -1
   \multiply \grcalcd by \factor
   \put(\grcalca,\grcalcb){\oval(\grcalcd,\grcalcc)[t]}
   \advance \grcalcb by \factor
   \put(\grcalca,\grcalcb) {\line(0,-1){\qfactor}} 
   \advance \grcolumn by #1}

 \newcommand{\gwmuc}[1]{    
   \grcalca = \grcolumn
   \multiply \grcalca by \factor
   \advance \grcalca by \hfactor
   \grcalcb = \grrow
   \multiply \grcalcb by \factor
   \grcalcc = #1
   \advance \grcalcc by -1
   \multiply \grcalcc by \factor
   \put(\grcalca,\grcalcb){\line(1,0){\grcalcc}}
   \advance \grcalca by -\hfactor
   \grcalcd = \hfactor
   \multiply \grcalcd by #1
   \advance \grcalca by \grcalcd
   \grcalcc = \factor
   \advance \grcalcc by \hfactor
   \grcalcd = #1
   \advance \grcalcd by -1
   \multiply \grcalcd by \factor
   \put(\grcalca,\grcalcb){\oval(\grcalcd,\grcalcc)[b]}
   \advance \grcalcb by -\hfactor
   \advance \grcalcb by -\qfactor
   \put(\grcalca,\grcalcb) {\line(0,-1){\qfactor}} 
   \advance \grcolumn by #1}

 \newcommand{\gwcmc}[1]{   
   \grcalca = \grcolumn
   \multiply \grcalca by \factor
   \advance \grcalca by \hfactor
   \grcalcb = \grrow
   \multiply \grcalcb by \factor
   \advance \grcalcb by -\factor
   \grcalcc = #1
   \advance \grcalcc by -1
   \multiply \grcalcc by \factor
   \put(\grcalca,\grcalcb){\line(1,0){\grcalcc}}
   \grcalcd = #1
   \advance \grcalcd by -1
   \multiply \grcalcd by \hfactor
   \advance \grcalca by \grcalcd
   \grcalcc = \factor
   \advance \grcalcc by \hfactor
   \grcalcd = #1
   \advance \grcalcd by -1
   \multiply \grcalcd by \factor
   \put(\grcalca,\grcalcb){\oval(\grcalcd,\grcalcc)[t]}
   \advance \grcalcb by \factor
   \put(\grcalca,\grcalcb) {\line(0,-1){\qfactor}} 
   \advance \grcolumn by #1}

 \newcommand{\gev}{  
   \grcalca = \grcolumn
   \advance \grcalca by 1
   \multiply \grcalca by \factor
   \grcalcb = \grrow
   \multiply \grcalcb by \factor
   \grcalcc = \factor
   \advance \grcalcc by \hfactor
   \put(\grcalca,\grcalcb){\oval(\factor,\grcalcc)[b]}
   \advance \grcolumn by 2}

 \newcommand{\gdb}{   
   \grcalca = \grcolumn
   \advance \grcalca by 1
   \multiply \grcalca by \factor
   \grcalcb = \grrow
   \advance \grcalcb by -1
   \multiply \grcalcb by \factor
   \grcalcc = \factor
   \advance \grcalcc by \hfactor
   \put(\grcalca,\grcalcb){\oval(\factor,\grcalcc)[t]}
   \advance \grcolumn by 2}

 \newcommand{\gwev}[1]{    
   \grcalca = \grcolumn
   \multiply \grcalca by \factor
   \grcalcd = \hfactor
   \multiply \grcalcd by #1
   \advance \grcalca by \grcalcd
   \grcalcb = \grrow
   \multiply \grcalcb by \factor
   \grcalcc = \factor
   \advance \grcalcc by \hfactor
   \grcalcd = #1
   \advance \grcalcd by -1
   \multiply \grcalcd by \factor
   \put(\grcalca,\grcalcb){\oval(\grcalcd,\grcalcc)[b]}
   \advance \grcolumn by #1}

 \newcommand{\gwdb}[1]{   
   \grcalca = \grcolumn
   \multiply \grcalca by \factor
   \grcalcd = \hfactor
   \multiply \grcalcd by #1
   \advance \grcalca by \grcalcd
   \grcalcb = \grrow
   \advance \grcalcb by -1
   \multiply \grcalcb by \factor
   \grcalcc = \factor
   \advance \grcalcc by \hfactor
   \grcalcd = #1
   \advance \grcalcd by -1
   \multiply \grcalcd by \factor
   \put(\grcalca,\grcalcb){\oval(\grcalcd,\grcalcc)[t]}
   \advance \grcolumn by #1}

 \newcommand{\gbr}{
   \grcalca = \grcolumn
   \multiply \grcalca by \factor
   \advance \grcalca by \hfactor
   \grcalcb = \grcalca
   \advance \grcalcb by \hfactor
   \grcalcc = \grcalca
   \advance \grcalcc by \factor
   \grcalcd = \grrow
   \multiply \grcalcd by \factor
   \grcalce = \grcalcd
   \advance \grcalce by -\tfactor
   \grcalcf = \grcalcd
   \advance \grcalcf by -\hfactor
   \grcalcg = \grcalce
   \advance \grcalcg by -\tfactor
   \grcalch = \grcalcd
   \advance \grcalch by -\factor
   \qbezier(\grcalca,\grcalcd)(\grcalca,\grcalce)(\grcalcb,\grcalcf) 
   \qbezier(\grcalcb,\grcalcf)(\grcalcc,\grcalcg)(\grcalcc,\grcalch) 
   \advance \grcalcf by -\dfactor
   \advance \grcalcb by -\sfactor
   \qbezier(\grcalca,\grcalch)(\grcalca,\grcalcg)(\grcalcb,\grcalcf) 
   \advance \grcalcf by \sfactor
   \advance \grcalcb by \tfactor
   \qbezier(\grcalcc,\grcalcd)(\grcalcc,\grcalce)(\grcalcb,\grcalcf) 
   \advance \grcolumn by 2}

 \newcommand{\gibr}{
   \grcalca = \grcolumn
   \multiply \grcalca by \factor
   \advance \grcalca by \hfactor
   \grcalcb = \grcalca
   \advance \grcalcb by \hfactor
   \grcalcc = \grcalca
   \advance \grcalcc by \factor
   \grcalcd = \grrow
   \multiply \grcalcd by \factor
   \grcalce = \grcalcd
   \advance \grcalce by -\tfactor
   \grcalcf = \grcalcd
   \advance \grcalcf by -\hfactor
   \grcalcg = \grcalce
   \advance \grcalcg by -\tfactor
   \grcalch = \grcalcd
   \advance \grcalch by -\factor
   \qbezier(\grcalcc,\grcalcd)(\grcalcc,\grcalce)(\grcalcb,\grcalcf) 
   \qbezier(\grcalcb,\grcalcf)(\grcalca,\grcalcg)(\grcalca,\grcalch) 
   \advance \grcalcf by -\dfactor
   \advance \grcalcb by \sfactor
   \qbezier(\grcalcc,\grcalch)(\grcalcc,\grcalcg)(\grcalcb,\grcalcf) 
   \advance \grcalcf by \sfactor
   \advance \grcalcb by -\tfactor
   \qbezier(\grcalca,\grcalcd)(\grcalca,\grcalce)(\grcalcb,\grcalcf) 
   \advance \grcolumn by 2}

 \newcommand{\gbrc}{
   \grcalca = \grcolumn
   \multiply \grcalca by \factor
   \advance \grcalca by \hfactor
   \grcalcb = \grcalca
   \advance \grcalcb by \hfactor
   \grcalcc = \grcalca
   \advance \grcalcc by \factor
   \grcalcd = \grrow
   \multiply \grcalcd by \factor
   \grcalce = \grcalcd
   \advance \grcalce by -\tfactor
   \grcalcf = \grcalcd
   \advance \grcalcf by -\hfactor
   \grcalcg = \grcalce
   \advance \grcalcg by -\tfactor
   \grcalch = \grcalcd
   \advance \grcalch by -\factor
   \put(\grcalcb,\grcalcf){\circle{\hfactor}}
   \qbezier(\grcalca,\grcalcd)(\grcalca,\grcalce)(\grcalcb,\grcalcf) 
   \qbezier(\grcalcb,\grcalcf)(\grcalcc,\grcalcg)(\grcalcc,\grcalch) 
   \advance \grcalcf by -\dfactor
   \advance \grcalcb by -\sfactor
   \qbezier(\grcalca,\grcalch)(\grcalca,\grcalcg)(\grcalcb,\grcalcf) 
   \advance \grcalcf by \sfactor
   \advance \grcalcb by \tfactor
   \qbezier(\grcalcc,\grcalcd)(\grcalcc,\grcalce)(\grcalcb,\grcalcf) 
   \advance \grcolumn by 2}

 \newcommand{\gibrc}{
   \grcalca = \grcolumn
   \multiply \grcalca by \factor
   \advance \grcalca by \hfactor
   \grcalcb = \grcalca
   \advance \grcalcb by \hfactor
   \grcalcc = \grcalca
   \advance \grcalcc by \factor
   \grcalcd = \grrow
   \multiply \grcalcd by \factor
   \grcalce = \grcalcd
   \advance \grcalce by -\tfactor
   \grcalcf = \grcalcd
   \advance \grcalcf by -\hfactor
   \grcalcg = \grcalce
   \advance \grcalcg by -\tfactor
   \grcalch = \grcalcd
   \advance \grcalch by -\factor
   \put(\grcalcb,\grcalcf){\circle{\hfactor}}
   \qbezier(\grcalcc,\grcalcd)(\grcalcc,\grcalce)(\grcalcb,\grcalcf) 
   \qbezier(\grcalcb,\grcalcf)(\grcalca,\grcalcg)(\grcalca,\grcalch) 
   \advance \grcalcf by -\dfactor
   \advance \grcalcb by \sfactor
   \qbezier(\grcalcc,\grcalch)(\grcalcc,\grcalcg)(\grcalcb,\grcalcf) 
   \advance \grcalcf by \sfactor
   \advance \grcalcb by -\tfactor
   \qbezier(\grcalca,\grcalcd)(\grcalca,\grcalce)(\grcalcb,\grcalcf) 
   \advance \grcolumn by 2} 

 \newcommand{\gu}[1]{
   \grcalca = \grcolumn
   \multiply \grcalca by \factor
   \grcalcd = \hfactor
   \multiply \grcalcd by #1
   \advance \grcalca by \grcalcd
   \grcalcb = \grrow
   \advance \grcalcb by -1
   \multiply \grcalcb by \factor
   \put(\grcalca,\grcalcb) {\line(0,1){\hfactor}} 
   \advance \grcalcb by \hfactor
   \put(\grcalca,\grcalcb) {\circle*{3}}
   \advance \grcolumn by #1}

 \newcommand{\gcu}[1]{
   \grcalca = \grcolumn
   \multiply \grcalca by \factor
   \grcalcd = \hfactor
   \multiply \grcalcd by #1
   \advance \grcalca by \grcalcd
   \grcalcb = \grrow
   \multiply \grcalcb by \factor
   \put(\grcalca,\grcalcb) {\line(0,-1){\hfactor}} 
   \advance \grcalcb by -\hfactor
   \put(\grcalca,\grcalcb) {\circle*{3}}
   \advance \grcolumn by #1}

 \newcommand{\gmp}[1]{
   \grcalca = \grcolumn
   \multiply \grcalca by \factor
   \advance \grcalca by \hfactor
   \grcalcb = \grrow
   \multiply \grcalcb by \factor
   \put(\grcalca,\grcalcb) {\line(0,-1){\dfactor}} 
   \advance \grcalcb by -\factor
   \put(\grcalca,\grcalcb) {\line(0,1){\dfactor}} 
   \advance \grcalcb by \hfactor
   \grcalcc = \factor
   \advance \grcalcc by -\qfactor
   \put(\grcalca,\grcalcb) {\circle{\grcalcc}}
   \put(\grcalca,\grcalcb) {\makebox(0,0){$\scriptstyle #1$}}
   \advance \grcolumn by 1}

 \newcommand{\gbmp}[1]{
   \grcalca = \grcolumn
   \multiply \grcalca by \factor
   \advance \grcalca by \hfactor
   \grcalcb = \grrow
   \multiply \grcalcb by \factor
   \put(\grcalca,\grcalcb) {\line(0,-1){\dfactor}} 
   \advance \grcalcb by -\factor
   \put(\grcalca,\grcalcb) {\line(0,1){\dfactor}} 
   \advance \grcalca by -\hfactor
   \advance \grcalca by \dfactor
   \advance \grcalcb by \dfactor
   \grcalcc = \factor
   \advance \grcalcc by -\sfactor
   \put(\grcalca,\grcalcb) {\framebox(\grcalcc,\grcalcc){$\scriptstyle #1$}}
   \advance \grcolumn by 1}

 \newcommand{\gbmpt}[1]{
   \grcalca = \grcolumn
   \multiply \grcalca by \factor
   \advance \grcalca by \hfactor
   \grcalcb = \grrow
   \multiply \grcalcb by \factor
   \put(\grcalca,\grcalcb) {\line(0,-1){\dfactor}} 
   \advance \grcalcb by -\factor
   \advance \grcalca by -\hfactor
   \advance \grcalca by \dfactor
   \advance \grcalcb by \dfactor
   \grcalcc = \factor
   \advance \grcalcc by -\sfactor
   \put(\grcalca,\grcalcb) {\framebox(\grcalcc,\grcalcc){$\scriptstyle #1$}}
   \advance \grcolumn by 1}

 \newcommand{\gbmpb}[1]{
   \grcalca = \grcolumn
   \multiply \grcalca by \factor
   \advance \grcalca by \hfactor
   \grcalcb = \grrow
   \multiply \grcalcb by \factor
   \advance \grcalcb by -\factor
   \put(\grcalca,\grcalcb) {\line(0,1){\dfactor}} 
   \advance \grcalca by -\hfactor
   \advance \grcalca by \dfactor
   \advance \grcalcb by \dfactor
   \grcalcc = \factor
   \advance \grcalcc by -\sfactor
   \put(\grcalca,\grcalcb) {\framebox(\grcalcc,\grcalcc){$\scriptstyle #1$}}
   \advance \grcolumn by 1}

 \newcommand{\gbmpn}[1]{
   \grcalca = \grcolumn
   \multiply \grcalca by \factor
   \advance \grcalca by \hfactor
   \grcalcb = \grrow
   \multiply \grcalcb by \factor
   \advance \grcalcb by -\factor
   \advance \grcalca by -\hfactor
   \advance \grcalca by \dfactor
   \advance \grcalcb by \dfactor
   \grcalcc = \factor
   \advance \grcalcc by -\sfactor
   \put(\grcalca,\grcalcb) {\framebox(\grcalcc,\grcalcc){$\scriptstyle #1$}}
   \advance \grcolumn by 1}

 \newcommand{\glmptb}{    
   \grcalca = \grcolumn
   \multiply \grcalca by \factor
   \advance \grcalca by \hfactor
   \grcalcb = \grrow
   \multiply \grcalcb by \factor
   \put(\grcalca,\grcalcb) {\line(0,-1){\dfactor}} 
   \advance \grcalcb by -\factor
   \put(\grcalca,\grcalcb) {\line(0,1){\dfactor}} 
   \advance \grcalca by -\hfactor
   \advance \grcalca by \dfactor
   \advance \grcalcb by \dfactor
   \put(\grcalca,\grcalcb) {\line(1,0){\factor}} 
   \advance \grcalcb by \factor
   \advance \grcalcb by -\sfactor
   \put(\grcalca,\grcalcb) {\line(1,0){\factor}} 
   \grcalcc = \factor
   \advance \grcalcc by -\sfactor
   \put(\grcalca,\grcalcb) {\line(0,-1){\grcalcc}} 
   \advance \grcolumn by 1}

 \newcommand{\glmpt}{    
   \grcalca = \grcolumn
   \multiply \grcalca by \factor
   \advance \grcalca by \hfactor
   \grcalcb = \grrow
   \multiply \grcalcb by \factor
   \put(\grcalca,\grcalcb) {\line(0,-1){\dfactor}} 
   \advance \grcalca by -\hfactor
   \advance \grcalca by \dfactor
   \advance \grcalcb by -\dfactor
   \put(\grcalca,\grcalcb) {\line(1,0){\factor}} 
   \advance \grcalcb by -\factor
   \advance \grcalcb by \sfactor
   \put(\grcalca,\grcalcb) {\line(1,0){\factor}} 
   \grcalcc = \factor
   \advance \grcalcc by -\sfactor
   \put(\grcalca,\grcalcb) {\line(0,1){\grcalcc}} 
   \advance \grcolumn by 1}

 \newcommand{\glmpb}{    
   \grcalca = \grcolumn
   \multiply \grcalca by \factor
   \advance \grcalca by \hfactor
   \grcalcb = \grrow
   \multiply \grcalcb by \factor
   \advance \grcalcb by -\factor
   \put(\grcalca,\grcalcb) {\line(0,1){\dfactor}} 
   \advance \grcalca by -\hfactor
   \advance \grcalca by \dfactor
   \advance \grcalcb by \dfactor
   \put(\grcalca,\grcalcb) {\line(1,0){\factor}} 
   \advance \grcalcb by \factor
   \advance \grcalcb by -\sfactor
   \put(\grcalca,\grcalcb) {\line(1,0){\factor}} 
   \grcalcc = \factor
   \advance \grcalcc by -\sfactor
   \put(\grcalca,\grcalcb) {\line(0,-1){\grcalcc}} 
   \advance \grcolumn by 1}

 \newcommand{\glmp}{    
   \grcalca = \grcolumn
   \multiply \grcalca by \factor
   \advance \grcalca by \dfactor
   \grcalcb = \grrow
   \multiply \grcalcb by \factor
   \advance \grcalcb by -\dfactor
   \put(\grcalca,\grcalcb) {\line(1,0){\factor}} 
   \advance \grcalcb by -\factor
   \advance \grcalcb by \sfactor
   \put(\grcalca,\grcalcb) {\line(1,0){\factor}} 
   \grcalcc = \factor
   \advance \grcalcc by -\sfactor
   \put(\grcalca,\grcalcb) {\line(0,1){\grcalcc}} 
   \advance \grcolumn by 1}

 \newcommand{\gcmptb}{    
   \grcalca = \grcolumn
   \multiply \grcalca by \factor
   \advance \grcalca by \hfactor
   \grcalcb = \grrow
   \multiply \grcalcb by \factor
   \put(\grcalca,\grcalcb) {\line(0,-1){\dfactor}} 
   \advance \grcalcb by -\factor
   \put(\grcalca,\grcalcb) {\line(0,1){\dfactor}} 
   \advance \grcalca by -\hfactor
   \advance \grcalcb by \dfactor
   \put(\grcalca,\grcalcb) {\line(1,0){\factor}} 
   \advance \grcalcb by \factor
   \advance \grcalcb by -\sfactor
   \put(\grcalca,\grcalcb) {\line(1,0){\factor}} 
   \advance \grcolumn by 1}

 \newcommand{\gcmpt}{    
   \grcalca = \grcolumn
   \multiply \grcalca by \factor
   \advance \grcalca by \hfactor
   \grcalcb = \grrow
   \multiply \grcalcb by \factor
   \put(\grcalca,\grcalcb) {\line(0,-1){\dfactor}} 
   \advance \grcalcb by -\factor
   \advance \grcalca by -\hfactor
   \advance \grcalcb by \dfactor
   \put(\grcalca,\grcalcb) {\line(1,0){\factor}} 
   \advance \grcalcb by \factor
   \advance \grcalcb by -\sfactor
   \put(\grcalca,\grcalcb) {\line(1,0){\factor}} 
   \advance \grcolumn by 1}

 \newcommand{\gcmpb}{    
   \grcalca = \grcolumn
   \multiply \grcalca by \factor
   \advance \grcalca by \hfactor
   \grcalcb = \grrow
   \multiply \grcalcb by \factor
   \advance \grcalcb by -\factor
   \put(\grcalca,\grcalcb) {\line(0,1){\dfactor}} 
   \advance \grcalca by -\hfactor
   \advance \grcalcb by \dfactor
   \put(\grcalca,\grcalcb) {\line(1,0){\factor}} 
   \advance \grcalcb by \factor
   \advance \grcalcb by -\sfactor
   \put(\grcalca,\grcalcb) {\line(1,0){\factor}} 
   \advance \grcolumn by 1}

 \newcommand{\gcmp}{    
   \grcalca = \grcolumn
   \multiply \grcalca by \factor
   \grcalcb = \grrow
   \multiply \grcalcb by \factor
   \advance \grcalcb by -\factor
   \advance \grcalcb by \dfactor
   \put(\grcalca,\grcalcb) {\line(1,0){\factor}} 
   \advance \grcalcb by \factor
   \advance \grcalcb by -\sfactor
   \put(\grcalca,\grcalcb) {\line(1,0){\factor}} 
   \advance \grcolumn by 1}

 \newcommand{\grmptb}{    
   \grcalca = \grcolumn
   \multiply \grcalca by \factor
   \advance \grcalca by \hfactor
   \grcalcb = \grrow
   \multiply \grcalcb by \factor
   \put(\grcalca,\grcalcb) {\line(0,-1){\dfactor}} 
   \advance \grcalcb by -\factor
   \put(\grcalca,\grcalcb) {\line(0,1){\dfactor}} 
   \advance \grcalca by \hfactor
   \advance \grcalca by -\dfactor
   \advance \grcalcb by \dfactor
   \put(\grcalca,\grcalcb) {\line(-1,0){\factor}} 
   \advance \grcalcb by \factor
   \advance \grcalcb by -\sfactor
   \put(\grcalca,\grcalcb) {\line(-1,0){\factor}} 
   \grcalcc = \factor
   \advance \grcalcc by -\sfactor
   \put(\grcalca,\grcalcb) {\line(0,-1){\grcalcc}} 
   \advance \grcolumn by 1}

 \newcommand{\grmpt}{    
   \grcalca = \grcolumn
   \multiply \grcalca by \factor
   \advance \grcalca by \hfactor
   \grcalcb = \grrow
   \multiply \grcalcb by \factor
   \put(\grcalca,\grcalcb) {\line(0,-1){\dfactor}} 
   \advance \grcalca by \hfactor
   \advance \grcalca by -\dfactor
   \advance \grcalcb by -\dfactor
   \put(\grcalca,\grcalcb) {\line(-1,0){\factor}} 
   \advance \grcalcb by -\factor
   \advance \grcalcb by \sfactor
   \put(\grcalca,\grcalcb) {\line(-1,0){\factor}} 
   \grcalcc = \factor
   \advance \grcalcc by -\sfactor
   \put(\grcalca,\grcalcb) {\line(0,1){\grcalcc}} 
   \advance \grcolumn by 1}

 \newcommand{\grmpb}{    
   \grcalca = \grcolumn
   \multiply \grcalca by \factor
   \advance \grcalca by \hfactor
   \grcalcb = \grrow
   \multiply \grcalcb by \factor
   \advance \grcalcb by -\factor
   \put(\grcalca,\grcalcb) {\line(0,1){\dfactor}} 
   \advance \grcalca by \hfactor
   \advance \grcalca by -\dfactor
   \advance \grcalcb by \dfactor
   \put(\grcalca,\grcalcb) {\line(-1,0){\factor}} 
   \advance \grcalcb by \factor
   \advance \grcalcb by -\sfactor
   \put(\grcalca,\grcalcb) {\line(-1,0){\factor}} 
   \grcalcc = \factor
   \advance \grcalcc by -\sfactor
   \put(\grcalca,\grcalcb) {\line(0,-1){\grcalcc}} 
   \advance \grcolumn by 1}

 \newcommand{\grmp}{    
   \grcalca = \grcolumn
   \multiply \grcalca by \factor
   \advance \grcalca by \factor
   \advance \grcalca by -\dfactor
   \grcalcb = \grrow
   \multiply \grcalcb by \factor
   \advance \grcalcb by -\dfactor
   \put(\grcalca,\grcalcb) {\line(-1,0){\factor}} 
   \advance \grcalcb by -\factor
   \advance \grcalcb by \sfactor
   \put(\grcalca,\grcalcb) {\line(-1,0){\factor}} 
   \grcalcc = \factor
   \advance \grcalcc by -\sfactor
   \put(\grcalca,\grcalcb) {\line(0,1){\grcalcc}} 
   \advance \grcolumn by 1}
\newcommand{\gsy}{
   \grcalca = \grcolumn
   \multiply \grcalca by \factor
   \advance \grcalca by \hfactor
   \grcalcb = \grcalca
   \advance \grcalcb by \hfactor
   \grcalcc = \grcalca
   \advance \grcalcc by \factor
   \grcalcd = \grrow
   \multiply \grcalcd by \factor
   \grcalce = \grcalcd
   \advance \grcalce by -\tfactor
   \grcalcf = \grcalcd
   \advance \grcalcf by -\hfactor
   \grcalcg = \grcalce
   \advance \grcalcg by -\tfactor
   \grcalch = \grcalcd
   \advance \grcalch by -\factor
   \qbezier(\grcalcc,\grcalcd)(\grcalcc,\grcalce)(\grcalcb,\grcalcf) 
   \qbezier(\grcalcb,\grcalcf)(\grcalca,\grcalcg)(\grcalca,\grcalch) 
   \advance \grcalcf by -\dfactor
   \advance \grcalcb by \sfactor
   \qbezier(\grcalcc,\grcalch)(\grcalcc,\grcalcg)(\grcalcb,\grcalcf) 
   \qbezier(\grcalca,\grcalcd)(\grcalca,\grcalce)(\grcalcb,\grcalcf) 
   \advance \grcolumn by 2}

 \newcommand{\gwmuh}[3]{    
   \grcalca = \grcolumn
   \multiply \grcalca by \factor
   \grcalcb = #2
   \advance \grcalcb by #3
   \multiply \grcalcb by \qfactor
   \advance \grcalca by \grcalcb
   \grcalcb = \grrow
   \multiply \grcalcb by \factor
   \grcalcc = #3
   \advance \grcalcc by -#2
   \multiply \grcalcc by \hfactor
   \grcalcd = \factor
   \advance \grcalcd by \hfactor
   \put(\grcalca,\grcalcb){\oval(\grcalcc,\grcalcd)[b]}
   \grcalca = \grcolumn
   \multiply \grcalca by \factor
   \grcalcc = #1
   \multiply \grcalcc by \hfactor
   \advance \grcalca by \grcalcc
   \advance \grcalcb by -\hfactor
   \advance \grcalcb by -\qfactor
   \put(\grcalca,\grcalcb) {\line(0,-1){\qfactor}} 
   \advance \grcolumn by #1}

 \newcommand{\gwcmh}[3]{   
   \grcalca = \grcolumn
   \multiply \grcalca by \factor
   \grcalcb = #2
   \advance \grcalcb by #3
   \multiply \grcalcb by \qfactor
   \advance \grcalca by \grcalcb
   \grcalcb = \grrow
   \advance \grcalcb by -1
   \multiply \grcalcb by \factor
   \grcalcc = #3
   \advance \grcalcc by -#2
   \multiply \grcalcc by \hfactor
   \grcalcd = \factor
   \advance \grcalcd by \hfactor
   \put(\grcalca,\grcalcb){\oval(\grcalcc,\grcalcd)[t]}
   \grcalca = \grcolumn
   \multiply \grcalca by \factor
   \grcalcc = #1
   \multiply \grcalcc by \hfactor
   \advance \grcalca by \grcalcc
   \advance \grcalcb by \factor
   \put(\grcalca,\grcalcb) {\line(0,-1){\qfactor}} 
   \advance \grcolumn by #1}

 \newcommand{\gsbox}[1]{
   \grcalca = \grcolumn
   \multiply \grcalca by \factor
   \grcalcb = \grrow
   \multiply \grcalcb by \factor
   \advance \grcalcb by -\factor
   \grcalcc = #1
   \multiply \grcalcc by \factor
   \grcalcd = \factor
   \put(\grcalca,\grcalcb){\framebox(\grcalcc,\grcalcd){}}}

\allowdisplaybreaks[4]
\newcommand{\nat}{\mbox{$\;\natural \;$}}

\begin{document}
\title[Wreath (co)products]
{Monoidal ring and coring structures obtained from wreaths and cowreaths}
\author{D. Bulacu}
\address{Faculty of Mathematics and Computer Science, University
of Bucharest, Str. Academiei 14, RO-010014 Bucharest 1, Romania}
\email{daniel.bulacu@fmi.unibuc.ro}
\author{S. Caenepeel}
\address{Faculty of Applied Sciences, 
Vrije Universiteit Brussel, VUB, B-1050 Brussels, Belgium}
\email{scaenepe@vub.ac.be}
\urladdr{http://homepages.vub.ac.be/\~{}scaenepe/}
\thanks{\rm 
The first author was supported by the strategic grant POSDRU/89/1.5/S/58852, Project 
``Postdoctoral program for training scientific researchers" cofinanced by the European Social Fund within the 
Sectorial Operational Program Human Resources Development 2007 - 2013.
The second author was supported by the research project G.0117.10  
``Equivariant Brauer groups and Galois deformations'' from
FWO-Vlaanderen.}
\begin{abstract}
Let $A$ be an algebra in a monoidal category $\Cc$, and let $X$ be an object in $\Cc$.
We study $A$-(co)ring structures on the left $A$-module $A\ot X$. These correspond
to (co)algebra structures in $EM(\Cc)(A)$, the Eilenberg-Moore category associated to $\Cc$ and $A$. The ring structures are in bijective correspondence to wreaths in $\Cc$, and
their category of representations is the category of representations over the induced wreath product. The coring structures are in bijective correspondence to cowreaths in $\Cc$, and
their category of corepresentations is the category of generalized entwined modules.
We present several examples coming from (co)actions of
Hopf algebras and their generalizations. Various notions of smash products that have
appeared in the literature appear as special cases of our construction.
\end{abstract}
\maketitle

\section*{Introduction}
Let $A$ be an algebra, and let $C$ be a coalgebra, and suppose that we have
an entwining map $\psi$. It is well-known that the vector space $A\ot C$ carries
a coring structure, such that the category of entwined modules is isomorphic to
the category of comodules over this coring. 
Examples of entwining structures come from Doi-Koppinen data over a bialgebra.
Doi-Koppinen data can be defined over quasi-bialgebras; we have similar results:
$A\ot C$ is still an $A$-coring; however, there is one major difference from the
classical theory: $C$ is no longer a $k$-coalgebra. 
Otherwise stated, we can build an $A$-coring structure on $A\ot C$ although $C$ is not an ordinary  $k$-coalgebra.

The aim of this paper is to describe all possible $A$-(co)ring structures of the form $A\ot X$.
We have chosen to present our results in the language of $\Cc$-categories, also known as
module categories. The motivation for this choice was twofold. On one hand,
the generality of this approach allows us to cover many constructions  that are known
for Hopf algebras and their generalizations. On the other hand, the naturality of the
involved categorical arguments allows us to simplify some of the computations.
We use the machinery developed in \cite{par, sch}, slightly improved in \cite{bc3}. 
Schauenburg \cite{sch} has observed that $A$-ring structures on $A\ot X$ (with left
$A$-module structure given by multiplication of $A$) depends on two morphisms, which we
will call $\zeta$ and $\sigma$. These morphisms have to satisfy certain conditions;
these are not given in \cite{sch}. We will work them out in \seref{strofcor}, and we will see 
that they are similar to conditions that appear in the Brzezi\'nski crossed product \cite{brz}.
We also discuss the dual question, namely we discuss $A$-coring structures on $A\ot X$,
and show that they are determined by two morphisms 
$\d: X\ra A\ot X\ot X$ and $f: X\ra A$ satisfying a list of compatibility conditions.
In \seref{monint}, we will restate the conditions on $\d$, $f$ (respectively $\zeta$, $\sigma$
in the ring case). Actually (co)ring structures on $A\ot X$ correspond to (co)algebra
structure on $X$ in a suitable monoidal category ${\cal T}_A^\#$.
Our notation is inspired by the well-known result that entwining structures of the form
$(A,C)_\psi$, with fixed $A$, correspond bijectively to coalgebras in Tamabara's
category ${\cal T}_A$, \cite{tambara}. During a visit of the first author to the
Wigner Research Centre for Physics in Budapest,
Gabriella B\"ohm pointed out that ${\cal T}_A^\#$ is the monoidal category ${\rm EM}(\Cc)(A)$ of
endomorphisms of $A$, viewed as a 0-cell in the Eilenberg-Moore category
${\rm EM}(\Cc)$, where $\Cc$ - being a monoidal category - is viewed as a 2-category
with one 0-cell. A similar result is that ${\cal T}_A={\rm Mnd}(\Cc)(A)$, where
${\rm Mnd}(\Cc)$ is the 2-category of monads in $\Cc$, as introduced in \cite{Ross}.
A second categorical interpretation is presented in \seref{(co)wreaths}:
$A$-(co)ring structures on $A\ot X$ are in bijective correspondence
to (co)wreath structures in $\Cc$, regarded again as a $2$-category with one object.
We also show that the category of representations of an $A$-ring of the form $A\ot X$
is isomorphic to the category of representations of the corresponding wreath product,
see \cite{LackRoss}. The category of corepresentations of an $A$-coring of the form
 $A\ot X$ is isomorphic to the category of generalized entwined modules,
 as introduced in \cite{bc4}.\\
 We present some applications in \seref{exps}. 
Using the theory of actions and coactions over a quasi-bialgebra we 
give examples of (co)wreaths $(A, X)$ with $X$ regarded as an object in ${\cal T}_A^\#$ 
rather than ${\cal T}_A$. 
As a consequence we obtain that the generalized-(quasi) smash product algebra defined in \cite{bpvo} 
is an example of a wreath product. Also the crossed product 
algebra built within the monoidal category of corepresentations over a dual quasi-Hopf algebra \cite{balan} is a wreath product.\\
Quasi-Hopf bimodules over a quasi-bialgebra $H$ can be applied to construct a 
cowreath $(\mfA, C)$ in ${}_H{\cal M}$, the monoidal category of left $H$-representations. We
remark that $C$ is 
viewed as an object in ${\cal T}_\mfA^\#$ and not in ${\cal T}_\mfA$, and that the category of corepresentations 
over the resulting  coring is isomorphic to ${}_H{\cal M}_\mfA^C$, the category of 
quasi-Hopf bimodules  associated to $(\mfA, C, H)$.\\
More examples can be obtained from actions and 
coactions of a bialgebroid. We propose an alternative way to define 
the crossed product algebra over a bialgebroid \cite{bb}, the underlying idea is to describe 
this algebra as a wreath product. We also construct a coring from a Doi-Koppinen datum over a 
bialgebroid, compatible with a module category structure, and recover the isomorphism between the category 
of Doi-Koppinen modules over a bialgebroid and the category of corepresentations over a suitable 
coring \cite{bcm}. These examples can be specified to bialgebroids coming 
from weak bialgebras.

Our theory can be applied to braided bialgebras; this will be the topic of the forthcoming
paper \cite{bc4}.

After an earlier version of this paper was finished, we were informed about the following possible alternative
approach, based on the description of the (co)wreath structures in a certain bicategory, leading to
\thref{EMAmain}. \thref{4.2} is actually a special case of \thref{EMAmain}. We have investigated this,
and we could give a proof of \thref{EMAmain}, but to this end we needed \thref{4.2}, so that
the two results are basically equivalent. Details are given in \seref{Appendix}. 
\section{Preliminaries}\selabel{prelim}
\setcounter{equation}{0}

\subsection{Module categories}
We assume that the reader is familiar with the basic theory 
of monoidal categories, and refer 
to \cite{bulacu,kas,maj} for more detail. Throughout this paper, $\Cc$ will be a monoidal category
with tensor product $\ot : \Cc\times \Cc\ra \Cc$ and unit object $\un{1}$. 
 We denote the identity morphism of an object $X\in \Cc$ by $\Id_X$. 
 We will assume implicitly that the monoidal 
category $\Cc$ is strict, that is, the associativity $a$ and unit constraints $l, r$ 
are all identity morphisms in $\Cc$. Our results will remain valid in arbitrary monoidal categories, since
every monoidal category is monoidal equivalent to a strict one, see for example \cite{bulacu, kas}. 

A right $\Cc$-category is a 
quadruple $(\Dc , \diamond, \Psi, {\bf r})$, where $\Dc$ is a category, 
$\diamond:\ \Dc\times \Cc\ra \Dc$ is a functor, and 
$\Psi:\ \diamond\circ (\diamond\times\Id)\ra \diamond\circ(\Id\times\ot)$ 
and ${\bf r}:\ \diamond\circ ({\rm Id}\times \un{1})\ra {\rm Id}$ are 
natural isomorphisms such that 
$({\Id_{\mfM}}\diamond l_X)\Psi_{{\mfM}, \un{1}, X}={\bf r}_{\mfM}\diamond {\Id_X}$ 
and the diagrams
\[
\xymatrix{
(({\mf M}\diamond X)\diamond Y)\diamond Z ~~\ar[d]_-{\Psi_{{\mf M}, X, Y}\diamond {\Id_Z}}
                                           \ar[r]^-{\Psi_{{\mf M}\diamond X, Y, Z}}
                                           & ~~({\mf M}\diamond X)\diamond (Y\ot Z) \\
({\mf M}\diamond(X\ot Y))\diamond Z ~~\ar[r]^-{\Psi_{{\mf M}, X\ot Y, Z}}& 
                                     ~~{\mf M}\diamond((X\ot Y)\ot Z)\ar[u]_-{{\Id_\mfM}\diamond a_{X, Y, Z}}
}
\]
commute, for all ${\mf M}\in \Dc$ and $X, Y, Z\in \Cc$. Obviously $\Cc$ is itself a right
$\Cc$-category, with $\diamond=\ot$, and $\Psi$ and ${\bf r}$ the natural identities
(recall that we assumed that $\Cc$ is strict). In fact, the above mentioned coherence
theorem can be extended to $\Cc$-categories, and this enables us to assume throughout
that $\Psi$ and ${\bf r}$ are natural identities, without loss of generality. In the literature,
$\Cc$-categories are also named module categories.

Let $\Dc$ be a right $\Cc$-category, and consider an algebra $A$ in $\Cc$.
A right module in $\Dc$ over $A$ is an object $\mfM\in \Dc$ together with a morphism 
$\nu_{\mfM}: {\mf M}\diamond A\ra {\mf M}$ such that 
$\nu_{\mfM}\circ (\Id_\mfM\diamond \un{\eta}_A)={\bf r}_{\mfM}$ and the diagram 
\[
\xymatrix{
(\mfM\diamond A)\diamond A\ar[rr]^-{\nu_{\mfM}\diamond\Id_A}
                           \ar[d]_-{\psi_{\mfM, A, A}}& & \mfM\diamond A\ar[d]^-{\nu_{\mfM}}\\
\mfM\diamond (A\ot A) \ar[r]_-{\Id_\mfM\diamond \un{m}_A}&\mfM\diamond A\ar[r]_-{\nu_{\mfM}}&\mfM                               
}
\] 
commutes.  Let $\Dc_A$ be the category of right modules and right 
linear maps in $\Dc$ over $A$. The right module structure on $\mfM\in \Dc_A$ will be
written symbolically as
$$\nu_\mfM=\gbeg{2}{3}
\got{1}{\mfM}\got{1}{A}\gnl
\grm\gnl
\gob{1}{\mfM}
\gend.$$
We can also define the dual notion of right comodule 
in a right $\Cc$-category $\Dc$ over a  coalgebra $C$ in $\Cc$.
The category of right comodules 
and right colinear maps in $\Dc$ over $C$ will be denoted as $\Dc^C$. 
The right comodule structure on $\mfM\in \Dc^C$ will be written as
$$\rho_\mfM=\gbeg{2}{3}
\got{1}{\mfM}\gnl
\grcm\gnl
\gob{1}{\mfM}\gob{1}{C}
\gend.$$

\subsection{Rings and corings in monoidal categories}\selabel{2.2}
The notion of ring and coring in a monoidal category is essentially due to Pareigis \cite{par} and 
Schauenburg \cite{sch}.  We present a brief survey on the topic, following terminology and
notation as in \cite{bc3}.\\
It is well-known that the category ${}_A\Mm_A$ of bimodules over a $k$-algebra $A$
is monoidal. A (co)algebra in ${}_A\Mm_A$ is called an $A$-(co)ring, 
see \cite{sw} for the original definition.\\
Let $\Dc$ be a right $\Cc$-category, and assume that both $\Cc$ and $\Dc$ have coequalizers.
Take an algebra $A$ in $\Cc$, $\mfM\in \Dc_A$ and $X\in {}_A\Cc$,
with structure morphism $\mu_X:\ A\ot X\to X$. We consider the
coequalizer $(M\diamond_AX, q^A_{{\mfM}, X})$ of the parallel morphisms  
$\nu_{\mfM}\diamond\Id_X$ and $(\Id_\mfM\diamond \mu_{X})\Psi_{\mfM, A, X}$ in $\Dc$:
\begin{displaymath}
\xymatrix{
(\mfM\diamond A)\diamond X\ar@<-1ex>[rr]_(.54){(\Id_\mfM\diamond \mu_{X})\Psi_{\mfM, A, X}} 
\ar@<1ex>[rr]^(.55){{\nu_{\mfM}\diamond\Id_X}}&&
M\diamond X
\ar[r]^-{q^A_{{\mfM}, X}}& M\diamond_AX .
}
\end{displaymath}
For a left $A$-linear morphism $f:\ X\ra Y$ in $\Cc$, 
let $\tilde{f}:\ \mfM\diamond_AX\ra \mfM\diamond_AY$ the unique morphism 
in $\Dc$ satisfying the equation 
\begin{equation}\eqlabel{coeqmorph}
\tilde{f}q^A_{\mfM, X}=q^A_{\mfM, Y}(\Id_\mfM\diamond f).
\end{equation}
Take $X\mapright{f}Y\mapright{f}Z$ in ${}_A\Cc$. It is easily verified that
$\widetilde{gf}=\tilde{g}\tilde{f}$.\\
Now let $g:\mfM\ra {\mf N}$ in $\Dc_A$ and $Y\in {}_A\Cc$. 
$\hat{g}: \mfM\diamond_AY\ra {\mf N}\diamond_AY$ is the unique morphism in 
$\Dc$ such that  
\begin{equation}\eqlabel{coeqmorph2}
\hat{g} q^A_{{\mfM}, Y}=q^A_{{\mf N}, Y}(g\diamond\Id_Y).
\end{equation}
For $\mfM \mapright{f} \mfN\mapright{g}\mfP$ in $\Dc_A$, we have that
$\widehat{gf}=\hat{g}\hat{f}$.\\
For $\mfM\in \Dc_A$, $X\in \Cc$ and $Y\in {}_A\Cc$, we have canonical
isomorphisms $\Upsilon_\mfM$, $\Upsilon_{\mfM, X}$ and $\Upsilon'_Y$:
\begin{itemize}
\item[-]$\Upsilon_\mfM: \mfM\diamond_AA\mapright{\cong} \mfM$, uniquely determined
by the property $\Upsilon_\mfM q^A_{\mfM, A}=\nu_\mfM$;
\item[-]$\Upsilon_{\mfM, X}: \mfM\diamond_A(A\ot X)\mapright{\cong} \mfM\diamond X$, 
uniquely determined by the property 
$\Upsilon_{\mfM, X}q^A_{\mfM, A\ot X}=(\nu_\mfM\diamond \Id_X)\Psi^{-1}_{\mfM, A, X}$; 
\item[-]$\Upsilon'_Y: A\ot _AY\mapright{\cong} Y$, uniquely determined by the property
$\Upsilon'_Yq^A_{A, Y}=\mu_Y$.
\end{itemize}
The following properties are now easily verified:
\begin{eqnarray}
\eqlabel{invcanisom1}
\Upsilon^{-1}_\mfM=q^A_{\mfM, A}(\Id_\mfM\diamond \un{\eta}_A)&;&
\Upsilon'^{-1}_Y=q^A_{A, Y}(\un{\eta}_A\ot \Id_Y);\\
\eqlabel{rel1}
&&\hspace*{-3cm}
\Upsilon^{-1}_{\mfM, X}=q^A_{\mfM, A\ot X}(\Id_\mfM\diamond \un{\eta}_A\ot \Id_X).
\end{eqnarray}
Before we are able to introduce the associativity constraint on ${}_A\Cc_A$, we need
the following concepts.

\begin{definitions}\deslabel{thetap}
Let $\Dc$ be a right $\Cc$-category and $A$ an algebra in $\Cc$.
\begin{itemize}
\item[(i)] An object $X\in \Cc$ is called (left) $\Dc$-coflat if the functor 
$-\diamond X:\ \Dc\ra \Dc$ preserves coequalizers. If $\Dc=\Cc$ then we simply say that $X$ 
is left coflat.
\item[(ii)] An object $Y\in {}_A\Cc$ is called (left) $\Dc$-robust if for any $\mfM\in \Dc$ and 
$X\in \Cc_A$, the morphism $\Theta'_{\mfM, X, Y}: (\mfM\diamond X)\diamond_AY\ra 
\mfM\diamond(X\ot_AY)$ defined by the commutativity of the diagram
\[
\xymatrix{
((\mfM\diamond X)\diamond A)\diamond Y\ar@<1ex>[rr]^(.55){\nu_{\mfM\diamond X}\diamond \Id_Y} 
\ar@<-1ex>[rr]_-{((\Id_\mfM\diamond \Id_X)\diamond \mu_Y)\Psi_{\mfM\diamond X, Y, A}}&&
(\mfM\diamond X)\diamond Y \ar[dr]_-{\xi'_{\mfM, X, Y}} 
\ar[r]^-{q^A_{\mfM\diamond X, Y}}&
(\mfM\diamond X)\diamond_AY\ar@{.>}[d]^-{\Theta'_{\mfM, X, Y}}\\
&&&\mfM\diamond (X\ot _AY)
}
\]
is an isomorphism. Here $\xi'_{\mfM, X, Y}:=(\Id_\mfM\diamond q^A_{X, Y})\Psi_{\mfM, X, Y}$. 
If $\Dc=\Cc$ then we simply say that $Y$ is left robust and write $\Theta'_{\mfM, X, Y}=
\theta'_{\mfM, X, Y}$.  
\end{itemize} 
We denote by ${}^{\Dc !}_A\Cc$ the category of left $A$-modules 
in $\Cc$ which are left $\Dc$-robust, left robust,  
left $\Dc$-coflat and left coflat. We write ${}^!_A\Cc={}^{\Cc !}_A\Cc$. 
\end{definitions}   

Note that left $\Dc$-coflatness of $A$ and left $\Dc$-robustness of $Y\in{}_A\Cc$
are needed in order to define:
\begin{itemize}
\item[-] a right module structure on $\mfM\diamond_AX$ in $\Dc$ over $A$, for all 
$\mfM\in \Dc_A$ and $X\in{}_A\Cc_A$;
\item[-] a left $A$-module structure on $X\ot_AY$ in $\Cc$, for any $X\in{}_A\Cc_A$;
\item[-] ``canonical" inverse isomorphisms 
\[
\xymatrix{
(\mfM\diamond_AX)\diamond_AY\ar@<1ex>[rr]^-{\Sigma'_{\mfM, X, Y}}&& 
\ar@<1ex>[ll]^-{\Gamma'_{\mfM, X, Y}}\mfM\diamond_A(X\ot_AY),
}
\] 
for $\mfM\in \Dc_A$ and $X\in {}_A\Cc_A$.   
\end{itemize}

Assume that $A$ is left coflat and left $\Dc$-coflat. Then the category 
${}^{\Dc !}_A\Cc_A$ is monoidal with tensor product 
$\ot_A$, associativity constraint $\Sigma'_{-, -, -}$, unit object $A$,
and left and right unit constraints $\Upsilon'_{-}$ and $\Upsilon_{-}$.
Here $\Sigma'$, $\Upsilon'_{-}$ and $\Upsilon_{-}$ have to be specialized
to the case $\Dc=\Cc$. Moreover, the functor $\diamond_A$ defines a right 
${}^{\Dc !}_A\Cc_A$-category structure on $\Dc_A$. More details about all these concepts 
and results can be found in \cite{par, sch, bc3}.

We are now able to define the notions of ring and coring in a monoidal category. 

\begin{definition}\delabel{coring}
Let $\Dc$ be a right $\Cc$-category and $A$ an algebra in $\Cc$ which is left $\Dc$-coflat and 
left coflat. An $A$-(co)ring in $\Cc$ compatible with $\Dc$ is a (co)algebra in the 
monoidal category ${}^{\Dc !}_A\Cc_A$. For such a (co)ring a right (co)module in $\Dc$ 
is a (co)module over it in the right ${}^{\Dc !}_A\Cc_A$-category $\Dc_A$.\\
A (co)ring in $\Cc$ compatible with $\Cc$ is simply termed a (co)ring in $\Cc$. 
Then a right (co)module over a (co)ring in $\Cc$ is exactly a (co)module over it 
in the right ${}_A^!\Cc_A$-category $\Cc_A$.  
\end{definition}

\section{$A$-(co)rings of the form $A\ot -$}\selabel{strofcor}
\setcounter{equation}{0}

Throughout this section, $A$ will be a left coflat algebra 
in a monoidal category $\Cc$ and $X$ a left coflat object of $\Cc$, so that 
$\mfC=A\ot X\in {}^!_A\Cc$, with left $A$-module structure morphism 
$\mu_{\mfC}:=\un{m}_A\ot\Id_X$.\\
As we have already explained we are interested in finding the $A$-(co)ring structures on $A\ot X$. 
Our starting point is the following result due to 
Tambara \cite{tambara} and Schauenburg \cite{sch}. 

\begin{lemma}\lelabel{TaSchstr}
Let $A$ be an algebra in a monoidal category $\Cc$ and $X$ an object of 
$\Cc$. Consider $\mfC=A\ot X$ as a left $A$-module in $\Cc$ via $\mu_{\mfC}:=\un{m}_A\ot\Id_X$. 
Then there is a bijective correspondence between right $A$-module structures on $\mfC=A\ot X$ 
that make $\mfC=A\ot X$ into an object of ${}_A\Cc_A$ and morphisms 
$\psi: X\ot A\ra A\ot X$ satisfying 
\begin{eqnarray}
\psi(\Id_X\ot \un{m}_A)&=&
(\un{m}_A\ot \Id_X)(\Id_A\ot \psi)(\psi\ot\Id_X)\eqlabel{e1};\\
\psi(\Id_X\ot \un{\eta}_{A})&=&\un{\eta}_A\ot\Id_X. \eqlabel{e2} 
\end{eqnarray}
\end{lemma}

\begin{proof}
We outline the proof. If $\nu_{\mfC}$ defines a right 
$A$-module structure on $A\ot X$ then $\psi =\nu_{\mfC}(\un{\eta}_A\ot\Id_X\ot\Id_A)$ 
satisfies (\ref{eq:e1}-\ref{eq:e2}).
 Conversely, if $\psi:\ X\ot A\ra A\ot X$ satisfies (\ref{eq:e1}-\ref{eq:e2}), 
then $\nu_{\mfC}=(\un{m}_A\ot \Id_X)(\Id_A\ot\psi)$ defines a right 
$A$-module structure on $\mfC$. 
\end{proof}

Let $\mfC\in {}_A\Cc_A$, as in \leref{TaSchstr}. 
Since $A$ is left coflat, $\mfC\ot_A\mfC$ is a right $A$-module via $\nu_{\mfC\ot_A\mfC}$, 
the unique morphism in $\Cc$ making the triangle in the diagram
\[
\xymatrix{
\mfC\ot A\ot \mfC\ot A~~\ar@<1ex>[rr]^-{\nu_{\mfC}\ot \Id_\mfC\ot \Id_A}\ar@<-1ex>[rr]_-{\Id_\mfC\ot\mu_{\mfC}\ot \Id_A}~~&&
\mfC\ot \mfC\ot A~~\ar[r]^{q^A_{\mfC, \mfC}\ot\Id_A}\ar[dr]_{q^A_{\mfC, \mfC}(\Id_\mfC\ot\nu_{\mfC})}~~&
(\mfC\ot_A\mfC)\ot A\ar@{.>}[d]^-{\nu_{\mfC\ot_A\mfC}}\\
& && \mfC\ot_A\mfC  
}
\]
commutative. 
Since $\mfC$ is left robust, $\mfC\ot_A\mfC$ is a left 
$A$-module via $\mu_{\mfC\ot_A\mfC}=\ov{\mu}_{\mfC\ot_A\mfC}\theta'^{-1}_{A, \mfC, \mfC}$, 
where $\theta'$ is defined in \desref{thetap} and $\ov{\mu}_{\mfC\ot_A\mfC}$ is the unique 
morphism in $\Cc$ making the triangle below commutative:
\[
\xymatrix{
A\ot \mfC\ot A\ot \mfC~~\ar@<1ex>[rr]^-{\Id_A\ot \nu_{\mfC}\ot \Id_\mfC}\ar@<-1ex>[rr]_-{\Id_A\ot\Id_\mfC\ot \mu_{\mfC}}~~&&
A\ot \mfC\ot \mfC~~\ar[r]^-{q^A_{A\ot \mfC, \mfC}}\ar[rd]_-{q^A_{\mfC, \mfC}(\mu_{\mfC}\ot\Id_\mfC)}~~&
(A\ot\mfC)\ot_A\mfC\ar@{.>}[d]^-{\ov{\mu}_{\mfC\ot_A\mfC}}\\
& && \mfC\ot_A\mfC  
}.
\]

Our next aim is to characterize left $A$-linear morphisms that define a (co)multiplication on $\mfC$.

\begin{lemma}\lelabel{(co)multstr}
Assume that $\mfC=A\ot X\in {}_A\Cc_A$, as in \leref{TaSchstr}.
\begin{itemize}
\item[(i)] There is a bijective correspondence between left $A$-linear morphisms 
$\un{m}_\mfC:\mfC\ot_A\mfC\ra \mfC$ and morphisms $\zeta: X\ot X\ra A\ot X$ in $\Cc$. 
\item[(ii)] There is a bijective correspondence between left $A$-linear morphisms 
$\un{\Delta}_{\mf C}:\ \mfC\ra \mfC\ot_A\mfC$ in $\Cc$ and 
 morphisms $\d: X\ra A\ot X\ot X$ in $\Cc$. 
\end{itemize}
\end{lemma}

\begin{proof}
(i)  The morphism corresponding to a left $A$-linear
$\un{m}_\mfC: \mfC\ot_A\mfC\ra \mfC$ is $\zeta:=\un{m}_\mfC\Upsilon ^{-1}_{\mfC, X}(\un{\eta}_A\ot \Id_X\ot \Id_X)$. 
The left $A$-linear morphism corresponding to $\zeta$ is
$\un{m}_\mfC:=(\un{m}_A\ot \Id_X)(\Id_A\ot \zeta)\gamma_{\mfC, X}$.
For the sake of completeness, observe that the left $A$-linearity of $\un{m}_\mfC$
is equivalent to
$\un{m}_\mfC q^A_{\mfC, \mfC}(\mu_\mfC\ot \Id_\mfC)=\mu_\mfC(\Id_A\ot \un{m}_\mfC q^A_{\mfC, \mfC})$.\\
(ii) A left $A$-linear morphism $\un{\Delta}_{\mf C}:\ \mfC\ra \mfC\ot_A\mfC$ in $\Cc$ has the form 
$
\un{\Delta}_{\mfC}=\Upsilon^{-1}_{\mfC, X}(\un{m}_A\ot\Id^{\ot 2}_X)(\Id_A\ot \delta_X)$,
for some $\delta_X:\ X\ra A\ot X\ot X$. 
$\delta_X$ is obtained from $\un{\Delta}_{\mfC}$ as 
$\delta_X=\Upsilon_{\mfC, X}\un{\Delta}_{\mfC}(\un{\eta}_A\ot\Id_X)$.
\end{proof}

\begin{lemma}
A left $A$-linear $\un{m}_\mfC: \mfC\ot_A\mfC\ra \mfC$ is right $A$-linear if and only if 
\begin{equation}\eqlabel{righAlinmult}
(\un{m}_A\ot \Id_X)(\Id_A\ot \psi)(\zeta\ot \Id_A)=(\un{m}_A\ot \Id_X)(\Id_A\ot \zeta)
(\psi\ot \Id_X)(\Id_X\ot \psi). 
\end{equation}
A left $A$-linear  $\un{\Delta}_{\mfC}: \mfC\ra \mfC\ot_A\mfC$ is right $A$-linear if and only if 
\begin{equation}\eqlabel{rightAlin}
(\un{m}_A\ot\Id_X^{\ot 2})(\Id_A\ot\d_X)\psi = 
(\un{m}_A\ot\Id_X^{\ot 2})(\Id_A\ot \psi\ot\Id_X)(\Id_\mfC\ot\psi)(\d_X\ot\Id_A).
\end{equation}
\end{lemma}

\begin{proof}
We only prove the second assertion, the first one is left to the reader.
It is not hard to see that 
\begin{eqnarray*}
&&\hspace*{-2cm}
(\Id_A\ot \Upsilon_{\mfC, X})\theta'_{A, \mfC, \mfC}q^A_{A\ot \mfC, \mfC}
=\Id_A\ot \Upsilon_{\mfC, X}q^A_{\mfC, \mfC}\\
&=&
\Id_A\ot \nu_{\mfC}\ot\Id_X=
\nu_{A\ot \mfC}\ot\Id_X=\Upsilon_{A\ot \mfC, X}q^A_{A\ot \mfC, \mfC},
\end{eqnarray*}
and since $q^A_{A\ot \mfC, \mfC}$ is a coequalizer it follows that 
\begin{equation}\eqlabel{rel2}
(\Id_A\ot \Upsilon_{\mfC, X})\theta'_{A, \mfC, \mfC}=
\Upsilon_{A\ot \mfC, X}.
\end{equation}
In order to investigate when $\un{\Delta}_{\mfC}$ is right $A$-linear, we compute
\begin{eqnarray*}
&&\hspace*{-1.5cm}
\nu_{\mfC\ot_A\mfC}(\un{\Delta}_{\mfC}\ot\Id_A)\\
&=&
\nu_{\mfC\ot_A\mfC}(\Upsilon^{-1}_{\mfC, X}\ot\Id_A)(\un{m}_A\ot\Id_X^{\ot 2}\ot\Id_A)
(\Id_A\ot \d_X\ot\Id_A)\\
&=&\nu_{\mfC\ot_A\mfC}(q^A_{\mfC, \mfC}\ot\Id_A)(\Id_{\mfC}\ot\un{\eta}_A\ot\Id_X\ot\Id_A)\\
&&(\un{m}_A\ot\Id_X^{\ot 2}\ot\Id_A)(\Id_A\ot \d_X\ot\Id_A)\\
&=&q^A_{\mfC, \mfC}(\Id_{\mfC}\ot \nu_{\mfC}(\un{\eta}_A\ot\Id_X\ot\Id_A))
(\un{m}_A\ot\Id_X^{\ot 2}\ot\Id_A)(\Id_A\ot \d_X\ot\Id_A)\\
&=&q^A_{\mfC, \mfC}(\Id_{\mfC}\ot\psi)(\un{m}_A\ot\Id_X^{\ot 2}\ot\Id_A)
(\Id_A\ot \d_X\ot\Id_A)\\
&=&q^A_{\mfC, \mfC}(\un{m}_A\ot\Id_X\ot\Id_\mfC)(\Id_A^{\ot 2}\ot\Id_X\ot\psi)
(\Id_A\ot \d_X\ot\Id_A).
\end{eqnarray*} 
Since 
\begin{eqnarray*}
\un{\Delta}_{\mfC}\nu_{\mfC}&=&\Upsilon^{-1}_{\mfC, X}(\un{m}_A\ot\Id_X^{\ot 2})(\Id_A\ot\d_X)
(\un{m}_A\ot\Id_X)(\Id_A\ot \psi)\\
&=&
\Upsilon^{-1}_{\mfC, X}(\un{m}_A\ot \Id^{\ot 2}_X)
(\un{m}_A\ot\Id_A\ot\Id_X^{\ot 2})(\Id^{\ot 2}_A\ot\d_X)(\Id_A\ot\psi) 
\end{eqnarray*}
and 
\begin{eqnarray*}
&&\hspace*{-1.5cm}
\Upsilon_{\mfC, X}q^A_{\mfC, \mfC}(\un{m}_A\ot\Id_X\ot\Id_\mfC)=
(\nu_{\mfC}\ot\Id_X)(\un{m}_A\ot\Id_X\ot\Id_\mfC)\\
&=&(\un{m}_A\ot\Id^{\ot 2}_X)(\Id_A\ot \psi\ot \Id_X)(\un{m}_A\ot\Id_X\ot\Id_\mfC)\\
&=&(\un{m}_A^2\ot\Id_X^{\ot 2})(\Id_A\ot (\Id_A\ot\psi\ot\Id_X)
(\Id_\mfC\ot \psi)(\d_X\ot\Id_A)), 
\end{eqnarray*}
where $\un{m}_A^2:=\un{m}_A(\un{m}_A\ot\Id_A)=\un{m}_A(\Id_A\ot\un{m}_A)$, we obtain that 
$\un{\Delta}_{\mfC}$ is right $A$-linear if and only if 
\begin{eqnarray}
&&\hspace*{-1cm}
(\un{m}_A^2\ot\Id^{\ot 2}_X)(\Id_A\ot (\Id_A\ot\d_X)\psi)\nonumber\\
&&=(\un{m}_A^2\ot\Id^{\ot 2}_X)(\Id_A\ot (\Id_A\ot\psi\ot\Id_X)
(\Id_\mfC\ot\psi)(\d_X\ot\Id_A)).\eqlabel{rightAlinb}
\end{eqnarray}
Composing both sides of \equref{rightAlinb} with 
$\un{\eta}_A\ot\Id_X\ot\Id_A$, we find \equref{rightAlin}. Consequently
\equref{rightAlin} is equivalent to \equref{rightAlinb}, and the result follows.
\end{proof}

Now we have to investigate when $\un{m}_\mfC$ is associative and 
when $\Delta_\mfC$ is coassociative.

\begin{lemma}\lelabel{3.15}
(i) Assume that $\un{m}_\mfC: \mfC\ot_A\mfC\ra \mfC$ is left and right $A$-linear,
that is, \equref{righAlinmult} is satisfied.
Then $\un{m}_\mfC$ is associative if and only if 
\begin{equation}\eqlabel{Multass}
(\un{m}_A\ot \Id_X)(\Id_A\ot \zeta)(\psi\ot\Id_X)(\Id_X\ot \zeta)
=(\un{m}_A\ot\Id_X)(\Id_A\ot\zeta)(\zeta\ot\Id_X).
\end{equation} 
(ii) Assume that $\un{\Delta}_\mfC: \mfC\ra \mfC\ot_A \mfC$ is left and right $A$-linear,
that is, \equref{rightAlin} is satisfied.
Then $\un{\Delta}_\mfC$ is coassociative if and only if 
\begin{equation}\eqlabel{Deltacoass}
(\un{m}_A\ot \Id_X^{\ot 3})(\Id_A\ot\psi\ot\Id_X^{\ot 2})(\Id_\mfC\ot\d_X)\d_X=
(\un{m}_A\ot\Id_X^{\ot 3})(\Id_A\ot\d_X\ot\Id_X)\delta_X.
\end{equation} 
\end{lemma}

\begin{proof}
We will only prove assertion (ii). Recall from \cite{par, sch, bc3}
that the morphism
$\Gamma'_{\mfM, X, Y}:\ \mfM\diamond_A(X\ot_AY)\ra (\mfM\diamond_AX)\diamond_AY$ 
is determined by the commutativity of the following diagram: 
\[
\xymatrix{
(\mfM\diamond A)\diamond (X\ot_AY)~~\ar@<1ex>[rr]^-{\nu_{\mfM}\diamond \Id_{X\ot_AY}}
\ar@<-1ex>[rr]_-{(\Id_\mfM\diamond\mu_{X\ot_AY})\Psi_{\mfM, A, X\ot_AY}}~~& &
\mfM\diamond(X\ot_AY)\ar[dr]_-{\widehat{q^A_{\mfM, X}}\Theta'^{-1}_{\mfM, X, Y}}\ar[r]^-{q^A_{\mfM, X\ot_AY}}&\mfM\diamond_A(X\ot_AY)
\ar@{.>}[d]^{\Gamma'_{\mfM, X, Y}}\\
&&&(\mfM\diamond_AX)\diamond_AY
}.
\]
We also have to make the observation that 
\begin{eqnarray*}
\Lambda_{\mfM, X}:\hspace*{-2mm}
&=&q^A_{\mfM\diamond_A\mfC, \mfC}
(\Upsilon^{-1}_{\mfM, X}\diamond \Id_{\mfC})(\Id_{\mfM}\diamond\Id_X\diamond\un{\eta}_A\ot\Id_X)\\
\hspace*{-2mm}
&=&q^A_{\mfM\diamond_A\mfC, \mfC}(\Id_{\mfM\diamond_A\mfC}\diamond\un{\eta}_A\ot \Id_X)
(\Upsilon^{-1}_{\mfM, X}\diamond\Id_X)
=\Upsilon^{-1}_{\mfM\diamond_A\mfC, X}(\Upsilon^{-1}_{\mfM, X}\diamond\Id_X),
\end{eqnarray*} 
is an isomorphism in $\Dc$. Now we can compute that
\begin{eqnarray*}
&&\hspace*{-1.5cm}
\Gamma'_{\mfC, \mfC, \mfC}\widetilde{\un{\Delta}_{\mfC}}\un{\Delta}_{\mfC}\\
&\equal{\equref{rel1}}&\Gamma'_{\mfC, \mfC, \mfC}\widetilde{\un{\Delta}_{\mfC}}
q^A_{\mfC, \mfC}(\Id_{\mfC}\ot\un{\eta}_A\ot\Id_X)(\un{m}_A\ot\Id_X^{\ot 2})(\Id_A\ot\delta_X)\\
&\equal{\equref{coeqmorph}}&
\Gamma'_{\mfC, \mfC, \mfC}q^A_{\mfC, \mfC\ot_A\mfC}(\Id_{\mfC}\ot\un{\Delta}_{\mfC})
(\Id_{\mfC}\ot\un{\eta}_A\ot\Id_X)(\un{m}_A\ot\Id_X^{\ot 2})(\Id_A\ot\delta_X)\\
&=&\widehat{q^A_{\mfC, \mfC}}\theta'^{-1}_{\mfC, \mfC, \mfC}(\Id_{\mfC}\ot \Upsilon^{-1}_{\mfC, X})
(\Id_{\mfC}\ot\un{m}_A\ot\Id_X^{\ot 2})(\Id_\mfC\ot\Id_A\ot\d_X)\\
&&(\Id_{\mfC}\ot\un{\eta}_A\ot\Id_X)(\un{m}_A\ot\Id_X^{\ot 2})(\Id_A\ot\d_X)\\
&\equal{\equref{rel1}}&\widehat{q^A_{\mfC, \mfC}}\theta'^{-1}_{\mfC, \mfC, \mfC}
(\Id_{\mfC}\ot q^A_{\mfC, \mfC})
(\Id_{\mfC}^{\ot 2}\ot\un{\eta}_A\ot\Id_X)(\Id_\mfC\ot\un{m}_A\ot\Id_X^{\ot 2})\\
&&(\Id_\mfC\ot\un{\eta}_A\ot\Id_A\ot\Id_X^{\ot 2})(\Id_\mfC\ot\d_X)
(\un{m}_A\ot\Id_X^{\ot 2})(\Id_A\ot\d_X)\\
&=&\widehat{q^A_{\mfC, \mfC}}q^A_{\mfC\ot\mfC, \mfC}
(\Id_{\mfC}^{\ot 2}\ot\un{\eta}_A\ot\Id_X)(\Id_{\mfC}\ot\d_X)(\un{m}_A\ot\Id_X^{\ot 2})
(\Id_A\ot \d_X)\\
&\equal{\equref{coeqmorph2}}&
q^A_{\mfC\ot_A\mfC, \mfC}(q^A_{\mfC, \mfC}\ot\Id_\mfC)(\Id_\mfC^{\ot 2}\ot\un{\eta}_A\ot\Id_X)
(\un{m}_A\ot\Id_X\ot\Id_\mfC\ot\Id_X)\\
&&(\Id_A^{\ot 2}\ot\Id_X\ot\d_X)(\Id_A\ot\d_X).
\end{eqnarray*}
Similarly, we have that 
\begin{eqnarray*}
\widehat{\un{\Delta}_{\mfC}}\un{\Delta}_{\mfC}
&\equal{\equref{rel1}}&
\widehat{\un{\Delta}_{\mfC}}q^A_{\mfC, \mfC}(\Id_{\mfC}\ot\un{\eta}_A\ot\Id_X)(\un{m}_A\ot\Id_X^{\ot 2})
(\Id_A\ot\d_X)\\
&\equal{\equref{coeqmorph2}}&
q^A_{\mfC\ot_A\mfC, \mfC}(\un{\Delta}_{\mfC}\ot\Id_{\mfC})
(\Id_{\mfC}\ot\un{\eta}_A\ot\Id_X)(\un{m}_A\ot\Id_X^{\ot 2})(\Id_A\ot\d_X)\\
&=&q^A_{\mfC\ot_A\mfC, \mfC}(\Upsilon^{-1}_{\mfC, X}\ot \Id_{\mfC})
(\un{m}_A\ot\Id_X^{\ot 2}\ot\Id_\mfC)(\Id_A\ot\d_X\ot\Id_\mfC)\\
&&(\Id_\mfC\ot\un{\eta}_A\ot\Id_X)(\un{m}_A\ot\Id_X^{\ot 2})(\Id_A\ot\d_X)\\
&=&q^A_{\mfC\ot_A\mfC, \mfC}(\Upsilon^{-1}_{\mfC, X}\ot \Id_{\mfC})
(\un{m}_A\ot\Id_X^{\ot 2}\ot\Id_\mfC)(\Id_A^{\ot 2}\ot\Id_X^{\ot 2}\ot\un{\eta}_A\ot\Id_X)\\
&&(\Id_A\ot\d_X\ot\Id_X)
(\un{m}_A\ot\Id_X^{\ot 2})(\Id_A\ot\d_X)\\
&=&\Lambda_{\mfC, X}(\un{m}_A\ot\Id_X^{\ot 3})(\Id_A\ot\d_X\ot\Id_X)
(\un{m}_A\ot\Id_X^{\ot 2})(\Id_A\ot\d_X)\\
&=&\Lambda_{\mfC, X}(\un{m}_A^2\ot\Id_X^{\ot 3})
(\Id_A\ot(\Id_A\ot\d_X\ot\Id_X)\d_X), 
\end{eqnarray*}
where $\un{m}^2_A:=\un{m}_A(\un{m}_A\ot \Id_A)=\un{m}_A(\Id_A\ot\un{m}_A)$. 
Now, since 
\begin{eqnarray*}
&&\hspace*{-1.5cm}
\Lambda^{-1}_{\mfC, X}q^A_{\mfC\ot_A\mfC, \mfC}(q^A_{\mfC, \mfC}\ot\Id_\mfC)
(\Id_\mfC^{\ot 2}\ot\un{\eta}_A\ot\Id_X)(\un{m}_A\ot\Id_X\ot\Id_\mfC\ot\Id_X)\\
&=&(\Upsilon_{\mfC, X}\ot\Id_X)(\nu_{\mfC\ot_A\mfC}\ot\Id_X)
(q^A_{\mfC, \mfC}\ot\Id_{\mfC})\\
&&(\Id_\mfC^{\ot 2}\ot\un{\eta}_A\ot\Id_X)(\un{m}_A\ot\Id_X\ot\Id_\mfC\ot\Id_X)\\
&=&\Upsilon_{\mfC, X}q^A_{\mfC, \mfC}(\Id_\mfC\ot\nu_\mfC)(\Id_\mfC^{\ot 2}\ot\un{\eta}_A)(\un{m}_A\ot\Id_X\ot\Id_\mfC)\ot\Id_X\\
&=&(\nu_\mfC\ot\Id_X)(\un{m}_A\ot\Id_X\ot\Id_\mfC)\ot\Id_X\\
&=&(\un{m}_A^2\ot\Id_X^{\ot 3})
(\Id_A\ot (\Id_A\ot \psi\ot\Id_X^{\ot 2})),
\end{eqnarray*}
it follows that $\un{\Delta}_\mfC$ is coassociative if and only if 
\begin{eqnarray*}
&&\hspace*{-2cm}
(\un{m}_A\ot\Id_X^{\ot 3})
(\Id_A\ot (\Id_A\ot \psi\ot\Id_X^{\ot 2})(\Id_\mfC\ot\d_X)\d_X)\\
&&=(\un{m}_A^2\ot\Id_X^{\ot 3})
(\Id_A\ot(\Id_A\ot\d_X\ot\Id_X)\d_X),
\end{eqnarray*}
and this is equivalent to \equref{Deltacoass}. 
\end{proof}

Finally, we have to discuss when $\un{m}_\mfC$ has a unit and
when $\un{\Delta}_{\mfC}$ has a counit.\\
To this end, we first observe that left $A$-linear morphisms $\un{\eta}_\mfC:A\ra \mfC$ 
corresponds bijectively to morphisms $\sigma: \un{1}\ra A\ot X$.
Actually, $\sigma$ can be obtained 
from $\un{\eta}_\mfC$ as $\un{\eta}_\mfC\un{\eta}_A$, while $\un{\eta}_\mfC$ can
be reconstructed from $\sigma$ using the formula
$(\un{m}_A\ot\Id_X)(\Id_A\ot\sigma): A\ra A\ot X$.\\
In a similar way, left $A$-linear morphisms $\un{\va}_{\mfC}: \mfC\ra A$ are in bijective
correspondence to morphisms $f: X\ra A$. 
Indeed, $f$ can be obtained from $\un{\va}_{\mfC}$ as 
$f:=\un{\va}_{\mfC}(\un{\eta}_A\ot \Id_X)$. Conversely, $\un{\va}_{\mfC}$ can be obtained from
$f$ using the formula $\un{\va}_{\mfC}:=\un{m}_A(\Id_A\ot f)$ from $\mfC$ to $A$. 

\begin{lemma}\lelabel{3.5}
(i) Let $\un{m}_\mfC$ be a left and right $A$-linear associative map, as in 
\leref{3.15}, and $\un{\eta}_\mfC:A\ra \mfC$ left $A$-linear.\\
$\un{\eta}_\mfC: A\ra \mfC$ is right $A$-linear if and only if 
\begin{equation}\eqlabel{unitAbilin}
(\un{m}_A\ot\Id_X)(\Id_A\ot\sigma)=(\un{m}_A\ot\Id_X)(\Id_A\ot \psi)(\sigma\ot\Id_A).
\end{equation}
In this situation, $\un{\eta}_\mfC$ is a unit for $\un{m}_\mfC$ 
if and only 
\begin{eqnarray}
&&(\un{m}_A\ot\Id_X)(\Id_A\ot \zeta)(\psi\ot\Id_X)(\Id_X\ot\sigma)=\un{\eta}_A\ot \Id_X,\eqlabel{unit1}\\
&&(\un{m}_A\ot \Id_X)(\Id_A\ot \zeta)(\sigma\ot \Id_X)=\un{\eta}_A\ot\Id_X.\eqlabel{unit2}
\end{eqnarray}

(ii) Let $\un{\Delta}_\mfC$ be a left and right $A$-linear coassociative map, as in 
\leref{3.15}, and $\un{\va}_{\mfC}: \mfC\ra A$ left $A$-linear.\\
$\un{\va}_{\mfC}$ is right $A$-linear if and only if 
\begin{equation}\eqlabel{counitAbilin}
\un{m}_A(\Id_A\ot f)\psi=\un{m}_A(f\ot \Id_A).
\end{equation}
In this situation, $\un{\va}_\mfC$ is a counit for $\un{\Delta}_\mfC$ 
if and only if
\begin{eqnarray}
&&(\un{m}_A\ot\Id_X)(\Id_A\ot f\ot\Id_X)\d_X=\un{\eta}_A\ot\Id_X,\eqlabel{counit1}\\
&&(\un{m}_A\ot\Id_X)(\Id_A\ot\psi)(\Id_\mfC\ot f)\d_X=\un{\eta}_A\ot\Id_X.\eqlabel{counit2}
\end{eqnarray}
\end{lemma}   

\begin{proof}
We will prove the second statement, and leave the first one to the reader.
$\un{\va}_{\mfC}$ is right $A$-linear if and only if
\begin{equation}\eqlabel{counitAbilina}
\un{m}_A^2(\Id_A\ot(\Id_A\ot f)\psi)=
\un{m}_A^2(\Id_A\ot f\ot\Id_A).
\end{equation}
Clearly \equref{counitAbilina} is equivalent to \equref{counitAbilin}. Note that 
\equref{counitAbilin} follows from \equref{counitAbilina}, after we compose both
sides with $\un{\eta}_A\ot \Id_X\ot\Id_A$. Now
\begin{eqnarray*}
\Upsilon_{A, X}\widehat{\un{\va}_{\mfC}}\un{\Delta}_\mfC&=&
\Upsilon_{A, X}\widehat{\un{\va}_{\mfC}}\Upsilon^{-1}_{\mfC, X}(\un{m}_A\ot\Id_X^{\ot 2})
(\Id_A\ot\d_X)\\
&=&\Upsilon_{A, X}\widehat{\un{\va}_{\mfC}}q^A_{\mfC, \mfC}(\Id_{\mfC}\ot\un{\eta}_A\ot\Id_X)
(\un{m}_A\ot\Id_X^{\ot 2})(\Id_A\ot\d_X)\\
&=&\Upsilon_{A, X}q^A_{A, \mfC}(\un{\va}_{\mfC}\ot\Id_{\mfC})
(\Id_{\mfC}\ot\un{\eta}_A\ot\Id_X)(\un{m}_A\ot \Id_X^{\ot 2})(\Id_A\ot\d_X)\\
&=&(\un{m}_A\ot\Id_X)(\un{m}_A\ot\Id_A\ot\Id_X)(\Id_A^{\ot 2}\ot\un{\eta}_A\ot\Id_X)\\
&&(\Id_A\ot f\ot\Id_X)(\un{m}_A\ot\Id_X^{\ot 2})(\Id_A\ot\d_X)\\
&=&(\un{m}_A\ot\Id_X)(\Id_A\ot f\ot\Id_X)(\un{m}_A\ot\Id_X^{\ot 2})(\Id_A\ot\d_X)\\
&=&(\un{m}_A^2\ot\Id_X)(\Id_A\ot (\Id_A\ot f\ot\Id_X)\d_X),
\end{eqnarray*}
and we obtain that $\Upsilon_{A, X}\widehat{\un{\va}_{\mfC}}\un{\Delta}_{\mfC}=\Id_\mfC$ 
if and only if \equref{counit1} holds. In a similar way, we have that
\[
\Upsilon_{\mfC}\widetilde{\un{\va}_{\mfC}}\un{\Delta}_{\mfC}=
(\un{m}_A^2\ot\Id_X)(\Id_A\ot(\Id_A\ot \psi)(\Id_\mfC\ot f)\d_X),
\]
and so $\Upsilon_{\mfC}\widetilde{\un{\va}_{\mfC}}\un{\Delta}_{\mfC}=\Id_\mfC$ if and only 
if \equref{counit2} holds. 
\end{proof}

Collecting our results, we obtain \prref{3.18}.

\begin{proposition}\prlabel{3.18}
Let $A$ be an algebra in a monoidal category $\Cc$ which is left coflat and $X$ a left 
coflat object of $\Cc$. $A\ot X$ is a left $A$-module in $\Cc$ via the multiplication 
$\un{m}_A$ of $A$. Then 

(i) $A$-ring structures on $A\ot X$ in $\Cc$ correspond bijectively to 
morphisms $\psi: X\ot A\ra A\ot X$, $\zeta_X: X\ot X\ra A\ot X$ and 
$\sigma:\un{1}\ra A\ot X$ in $\Cc$ such that \equref{e1}, \equref{e2}, \equref{righAlinmult}, 
\equref{Multass}, \equref{unitAbilin}, \equref{unit1} and \equref{unit2} are satisfied.

(ii) $A$-coring structures
on $\mfC=A\ot X$ correspond bijectively to morphisms $\psi:X\ot A\ra A\ot X$, 
$\d_X: X\ra A\ot X\ot X$ and $f: X\ra A$ in $\Cc$ such that \equref{e1}, \equref{e2}, 
\equref{rightAlin}, \equref{Deltacoass}, \equref{counitAbilin}, \equref{counit1} and 
\equref{counit2} hold.    
\end{proposition}

For further use, record that $\mfC$ is a (co)ring in $\Cc$ compatible with a right $\Cc$-category $\Dc$ if, 
in addition, $A$ and $X$ are also (left) $\Dc$-coflat objects of $\Cc$. In both cases we refer to the 
$A$-(co)ring $\mfC$ as a quadruple 
$\mfC=(A\ot X, \psi, (\d_X)\zeta_X, (f)\sigma)$ with $\psi, (\d_X)\zeta_X, (f)\sigma$ 
satisfying all the conditions in \prref{3.18}.

\section{The categories ${\rm EM}(\Cc)(A)$ and ${\rm Mnd}(\Cc)(A)$}\selabel{monint}
\setcounter{equation}{0}
The main goal of this section is to restate  the necessary and sufficient 
conditions for $A\ot -$ to be an $A$-(co)ring in terms of monoidal categories.
More precisely, we will show that $A\ot X$ admits 
an $A$-(co)ring structure with the given left $A$-module structure if and only if $X$ 
is a (co)algebra in a certain monoidal category. In a earlier version of this paper we have 
constructed this category by hand, inspired by the structures of $X$ that endow $A\ot X$ with an $A$-(co)ring 
structure. Afterwards Gabriella B\"ohm pointed us that our monoidal category should be related to 
the Eilenberg-Moore category associated to $\Cc$, viewed as a $2$-category in the canonical way. 
After some more investigations we obtained the results of this Section. 

Let ${\cal K}$ be a $2$-category; its objects (or $0$-cells) will be denoted by capital letters. $1$-cell between
two 0-cells $U$ and $V$ will be denoted as 
$
\xymatrix{
U\ar[r]^{f}&V
}
$, 
the identity morphism of a $1$-cell $f$ by $1_f$ and, more generally, a $2$-cell by 
$
\xymatrix{
f\ar@2{->}[r]^\rho&f'
}
$. We also denote by $\circ$ the vertical composition of $2$-cells 
$
\xymatrix{
f\ar@2{->}[r]^\rho&f'\ar@2{->}[r]^-{\tau}&f{''}
}
$ in ${\cal K}(U, V)$, 
by $\odot$ the horizontal composition of $2$-cells 
\[
\xymatrix{
U\rtwocell^f_{f'}{\rho} &V\rtwocell^g_{g'}{\rho'}&W 
},~~
\xymatrix{
gf\ar@2{->}[r]^-{\rho'\odot\rho}&g'f', 
}
\] 
and by 
$
(
\xymatrix{
U\ar[r]^{1_U}&U
}, 
\xymatrix{
1_U\ar@2{->}[r]^{i_U}&1_U
}
)
$ 
the pair defined by the image of the unit functor from ${\bf 1}$ to ${\cal K}(U, U)$, where 
${\bf 1}$ is the terminal object of the category of small categories.  
For more detail on 2-categories, we refer the reader to \cite[Ch. 7]{Borceux} or \cite[Ch. XII]{mc}. 

To a 2-category ${\cal K}$, we can associate a new 2-category $EM({\cal K})$, called the Eilenberg-Moore 
category associated to ${\cal K}$. We sketch the definition, following \cite{LackRoss}.

$\bullet$ $0$-cells are monads in ${\cal K}$, that is quadruples $(A, t, \mu, \eta)$ consisting 
in an object $A$ of ${\cal K}$, a $1$-cell 
$
\xymatrix{
A\ar[r]^t&A
}
$  
in ${\cal K}$ and $2$-cells 
$
\xymatrix{
t\circ t\ar@2{->}[r]^{\mu}&t
}
$ 
and 
$
\xymatrix{
1_A\ar@2{->}[r]^\eta&t
}
$ 
in ${\cal K}$ such that 
\[
\mu\circ(\mu \odot 1_t)=\mu\circ (1_t\odot \mu)~,~\mu\circ (1_t\odot \eta)=1_t=\mu\circ (\eta\odot 1_t)~; 
\]

$\bullet$ $1$-cells are the monad morphisms, i.e., if $\mathbb{A}=(A, t, \mu_t, \eta_t)$ and 
$\mathbb{B}=(B, s, \mu_s, \eta_s)$ are monads in ${\cal K}$ then a monad morphism between $\mathbb{A}$ and $\mathbb{B}$ 
is a pair $(f, \psi)$ with 
$
\xymatrix{
A\ar[r]^f&B
}
$ 
a $1$-cell in ${\cal K}$ and 
$\xymatrix{
sf\ar@2{->}[r]^{\psi}&ft
}
$   
a $2$-cell in ${\mathcal K}$ such that the following equalities hold:
\[
(1_f\odot \mu_t)\circ (\psi\odot 1_t)\circ (1_s\odot \psi)=\psi\circ (\mu_s\odot 1_f)~,~
\psi\circ(\eta_s\odot 1_f)=1_f\odot \eta_t;
\]

$\bullet$ $2$-cells 
$
\xymatrix{
(f, \psi)\ar@2{->}[r]^{\rho}&(g, \phi)
}
$ 
are $2$-cells 
$
\xymatrix{
f\ar@2{->}[r]^{\rho}&gt
}
$ 
in ${\cal K}$ obeying the equality
\begin{equation}\eqlabel{2cellEM}
(1_g\odot\mu_t)\circ (\rho\odot 1_t)\circ \psi=(1_g\odot \mu_t)\circ (\phi\odot 1_t)\circ (1_s\odot \rho);
\end{equation}

$\bullet$ the vertical composition of two $2$-cells 
$
\xymatrix{
(f, \psi)\ar@2{->}[r]^{\rho}&(g, \phi)\ar@2{->}[r]^{\rho'}&(h, \gamma)
}
$ 
is given by 
\[
\xymatrix{
(f, \psi)\ar@2{->}[r]^-{\rho'\circ \rho}&(h, \gamma)
}~,~ 
\rho'\circ \rho:=(1_h\odot \mu_t)\circ (\rho'\odot 1_t)\circ \rho;
\]

$\bullet$ the horizontal composition of two cells 
\[
\xymatrix{
\mathbb{A}\rtwocell^{(f, \psi)}_{(f', \psi')}{\rho}&\mathbb{B}
\rtwocell^{(g, \phi)}_{(g', \phi')}{\rho'}&\mathbb{C}
}
\]
is defined by $(g, \phi)(f, \psi)=(gf, (1_g\odot \psi)\circ (\phi\odot 1_f))$, etc. and 
$
\xymatrix{
gf\ar@2{->}[r]^{\rho'\oslash \rho}&g'f't
}
$ 
given by 
\[
\rho'\oslash \rho:=(1_{g'}\odot 1_{f'}\odot \mu_t)\circ (1_{g'}\odot \rho\odot 1_t)\circ 
(1_{g'}\odot \psi)\circ (\rho'\odot 1_f);
\]

$\bullet$ The identity morphism of the $1$-cell $(f, \psi)$ is $1_f\odot \eta_t$, 
and for any monad $\mathbb{A}=(A, t, \mu_t, \eta_t)$ in ${\cal K}$ we have  
$(1_\mathbb{A}, i_\mathbb{A})=((1_A, 1_t), \eta_t)$. 

It is well-known that strict monoidal categories can be viewed as 2-categories with
one 0-cell $*$. The 1-cells are the objects of the monoidal category, and the 2-cells
are its morphisms. So we can consider the Eilenberg-Moore category associated
to a strict monoidal category. This will be described in \prref{TA2}.

\begin{proposition}\prlabel{TA2}
The Eilenberg-Moore category $EM(\Cc)$ of a strict monoidal category $\Cc$ can be
described as follows.

$\bullet$ $0$-cells: algebras in $\Cc$;

$\bullet$ $1$-cells: 
$
\xymatrix{
A\ar[r]^{(X, \psi)}&B
}
$ 
with $X$ an object of $\Cc$ and $\psi: X\ot B\ra A\ot X$ morphism in $\Cc$ compatible with 
the algebra structure of $A$ and $B$, in the sense that
\begin{equation}\eqlabel{ocellspc}
\gbeg{3}{5}
\got{1}{X}\got{1}{B}\got{1}{B}\gnl
\gcl{1}\gmu\gnl
\gcl{1}\gcn{1}{1}{2}{1}\gnl
\gbrc\gnl
\gob{1}{A}\gob{1}{X}
\gend =
\gbeg{3}{5}
\got{1}{X}\got{1}{B}\got{1}{B}\gnl
\gbrc\gcl{1}\gnl
\gcl{1}\gbrc\gnl
\gmu\gcl{1}\gnl
\gob{2}{A}\gob{1}{X}
\gend 
~~,~~
\gbeg{2}{4}
\got{1}{X}\gnl
\gcl{1}\gu{1}\gnl
\gbrc\gnl
\gob{1}{A}\gob{1}{X}
\gend
=
\gbeg{2}{3}
\got{3}{X}\gnl
\gu{1}\gcl{1}\gnl
\gob{1}{A}\gob{1}{X}
\gend~,
\end{equation}
where 
$
\psi=
\gbeg{2}{3}
\got{1}{X}\got{1}{B}\gnl
\gbrc\gnl
\gob{1}{A}\gob{1}{X}
\gend
$
is our diagrammatic notation for a morphism $\psi : X\ot B\ra A\ot X$; 

$\bullet$ $2$-cells:
$
\xymatrix{
(X, \psi)\ar@2{->}[r]^{\rho}&(Y, \phi)
}
$
is a morphism $\rho: X\ra A\ot Y$ in $\Cc$ such that 
\begin{equation}\eqlabel{mtaext}
(\un{m}_A\ot\Id_Y)(\Id_A\ot \rho)\psi=
(\un{m}_A\ot \Id_Y)(\Id_A\ot \phi)(\rho\ot \Id_B);
\end{equation} 

$\bullet$ the vertical composition of two $2$-cells  
$
\xymatrix{
(X, \psi)\ar@2{->}[r]^{\rho}&(Y, \phi)\ar@2{->}[r]^{\rho'}&(Z, \gamma)
}
$  
is the following composition in $\Cc$:
\begin{equation}\eqlabel{compte}
\rho'\circ \rho:=(\un{m}_A\ot \Id_Z)(\Id_A\ot \rho')\rho : X\ra A\ot Z; 
\end{equation}

$\bullet$ the horizontal composition of $2$-cells 
\[
\xymatrix{
A\rtwocell^{(X, \psi)}_{(X', \psi')}{\rho}&B\rtwocell^{(Y, \phi)}_{(Y', \phi')}{\rho'}&C
}
\]
is defined by 
$
(Y, \phi)(X, \psi)=\left(X\ot Y, 
\gbeg{3}{4}
\got{1}{X}\got{1}{Y}\got{1}{C}\gnl
\gcl{1}\gbrbox\gnl
\gbrc\gcl{1}\gnl
\gob{1}{A}\gob{1}{X}\gob{1}{Y}
\gend
\right), 
$
where 
$
\phi=
\gbeg{2}{3}
\got{1}{X}\got{1}{B}\gnl
\gbrbox\gnl
\gob{1}{A}\gob{1}{X}
\gend
$, etc. and $\rho'\oslash \rho: X\ot Y\ra A\ot X'\ot Y'$ given by  
\begin{equation}\eqlabel{tensorprodmorph}
\rho'\oslash \rho:=
\gbeg{5}{8}
\got{2}{X}\gvac{1}\got{2}{Y}\gnl
\gcn{1}{1}{2}{2}\gvac{2}\gcn{1}{1}{2}{2}\gnl
\gsbox{2}\gnot{\hspace{5mm}\rho}\gvac{3}\gsbox{2}\gnot{\hspace{5mm}\rho'}\gnl
\gcl{1}\gcn{1}{1}{1}{3}\gvac{1}\gcl{1}\gcl{4}\gnl
\gcl{1}\gvac{1}\gbrc\gvac{2}\gnl
\gcn{1}{1}{1}{3}\gvac{1}\gcl{1}\gcl{1}\gnl
\gvac{1}\gmu\gcl{1}\gnl
\gvac{1}\gob{2}{A}\gob{1}{X'}\gob{1}{~~Y'}
\gend\hspace{2mm}, 
\mbox{~~where~~}
\rho=
\gbeg{2}{5}
\got{2}{X}\gnl
\gcn{1}{1}{2}{2}\gnl
\gsbox{2}\gnot{\hspace{5mm}\rho}\gnl
\gcl{1}\gcl{1}\gnl
\gob{1}{A}\gob{1}{~~X'}
\gend
~, ~etc.;
\end{equation}

$\bullet$ For any algebra $A$ in $\Cc$ we have 
$1_A=(\un{1}, \Id_A)$ and $i_A=\un{\eta}_A$, and for 
any $1$-cell 
$
\xymatrix{
A\ar[r]^{(X, \psi)}&A
}
$ 
we have $1_{(X, \psi)}=\un{\eta}_A\ot\Id_X$. 
\end{proposition}

\begin{proof}
The starting point is the identification of a monad in $EM(\Cc)$ with an 
algebra in $\Cc$. It can be easily checked that a monad $\mathbb{A}=(*, 
\xymatrix{
{*}\ar[r]^A&{*}
}, 
\un{m}_A: AA=A\ot A\ra A, \eta_{A}: 1_*=\un{1}\ra A)$ in $EM(\Cc)$ 
is a triple $(A, \un{m}_A, \un{\eta}_A)$ consisting of an 
object $A\in \Cc$ together with two morphisms $\un{m}_A:A\ot A\ra A$ and 
$\un{\eta}_A:\un{1}\ra A$ in $\Cc$ such that 
\[
\un{m}_A(\un{m}_A\ot\Id_A)=\un{m}_A(\Id_A\ot \un{m}_A)~~{\rm and}~~
\un{m}_A(\un{\eta}_A\ot\Id_A)=\Id_A=\un{m}_A(\Id_A\ot\un{\eta}_A).
\]   
Thus the $0$-cells of $EM(\Cc)$ are the algebras in $\Cc$. Then a monad morphism 
between two algebras $A$ and $B$ in $\Cc$ is a pair 
$(
\xymatrix{
{*}\ar[r]^X&{*}
}, \psi)
$ 
with $X\in \Cc$ and $\psi: BX=X\ot B\ra XA=A\ot X$ morphism in $\Cc$ such that 
\[
(\un{m}_A\ot\Id_X)(\Id_A\ot\psi)(\psi\ot \Id_B)=\psi(\Id_X\ot\un{m}_B)~~{\rm and}~~
\psi(\Id_X\ot \un{\eta}_B)=\un{\eta}_A\ot\Id_X.
\]
In diagrammatic notation these equalities read as \equref{ocellspc}, so they are the required conditions 
for $(X, \psi)$ to be a $1$-cell in $EM(\Cc)$.\\
Now, a $2$-cell 
$
\xymatrix{
(X, \psi)\ar@2{->}[r]^\rho&(Y, \phi)
}
$  
is a morphism $\rho: X\ra YA=A\ot Y$ satisfying the condition 
\begin{equation}\eqlabel{2cellpc}
\gbeg{3}{5}
\got{1}{X}\got{1}{B}\gnl
\gbrc\gnl
\gcl{1}\gsbox{2}\gnot{\hspace{5mm}\rho}\gnl
\gmu\gcl{1}\gnl
\gob{2}{A}\gob{1}{Y}
\gend
=
\gbeg{3}{6}
\got{2}{X}\got{1}{B}\gnl
\gcn{1}{1}{2}{2}\gvac{1}\gcl{1}\gnl
\gsbox{2}\gnot{\hspace{5mm}\rho}\gvac{2}\gcl{1}\gnl
\gcl{1}\gbrbox\gnl
\gmu\gcl{1}\gnl
\gob{2}{A}\gob{1}{Y}
\gend .
\end{equation}
If we rewrite this formula as a composition of maps, then we obtain 
\equref{mtaext}. The proof of all the remaining assertions is similar.
We point out that the proof of  \equref{tensorprodmorph} is based on \equref{2cellpc}. 
\end{proof}  

Let $U$ be 0-cell in a $2$-category ${\cal K}$. Then ${\cal K}(U):={\cal K}(U, U)$ is a monoidal
category. The objects are 1-cells $K\to K$, morphisms are $2$-cells, and the tensor product
is given by vertical composition of $2$-cells. The unit is $1_U$, the unit 1-cell on $U$.
We can apply this construction to $EM(\Cc)$. In this way, we obtain a monoidal category
$EM(\Cc)(A)$, for any algebra $A$ in $\Cc$. In \coref{4.2}, we provide an explicit description of this category.

\begin{corollary}\colabel{4.2}
Let $A$ be an algebra in a monoidal category  $\Cc$. Then 
$EM(\Cc)(A)$ is the following monoidal category.\\
$\bullet$ Objects are pairs $(X, \psi)$ with $X$ object in $\Cc$ and 
$\psi: X\ot A\ra A\ot X$ morphism in $\Cc$ such that       
\begin{equation}\eqlabel{ta}
\gbeg{3}{5}
\got{1}{X}\got{1}{A}\got{1}{A}\gnl
\gcl{1}\gmu\gnl
\gcl{1}\gcn{1}{1}{2}{1}\gnl
\gbrc\gnl
\gob{1}{A}\gob{1}{X}
\gend =
\gbeg{3}{5}
\got{1}{X}\got{1}{A}\got{1}{A}\gnl
\gbrc\gcl{1}\gnl
\gcl{1}\gbrc\gnl
\gmu\gcl{1}\gnl
\gob{2}{A}\gob{1}{X}
\gend 
~~,~~
\gbeg{2}{4}
\got{1}{X}\gnl
\gcl{1}\gu{1}\gnl
\gbrc\gnl
\gob{1}{A}\gob{1}{X}
\gend
=
\gbeg{2}{3}
\got{3}{X}\gnl
\gu{1}\gcl{1}\gnl
\gob{1}{A}\gob{1}{X}
\gend~.
\end{equation}

$\bullet$ Morphisms $\rho: (X, \psi)\ra (Y, \phi)$ are morphisms 
$\rho:X\ra A\ot Y$ satisfying \equref{2cellpc}, and the 
composition of two morphisms $\rho$ and $\rho'$ is as in \equref{compte}. 
The identity morphism $\Id_{(X, \psi)}$ is $\un{\eta}_A\ot\Id_X$.

$\bullet$ The tensor product is defined by 
\begin{equation}\eqlabel{mstrtamb}
(X, \psi)\un{\ot}(Y, \phi)=(X\ot Y, \psi_{X\ot Y}),~~{\rm with}~~ 
\psi_{X\ot Y}:=
\gbeg{3}{4}
\got{1}{X}\got{1}{Y}\got{1}{A}\gnl
\gcl{1}\gbrbox\gnl
\gbrc\gcl{1}\gnl
\gob{1}{A}\gob{1}{X}\gob{1}{Y}
\gend ,
\end{equation}
and the unit object is $(\un{1}, \Id_A)$. The tensor product of two morphisms 
$\rho$ and $\rho'$ is $\rho\un{\ot}\rho':=\rho'\oslash \rho$ as it is 
defined by \equref{tensorprodmorph}.   
\end{corollary}

We are now ready to state and prove the main result of this Section.

\begin{theorem}\thlabel{4.2}
Let $\Cc$ be a monoidal category, $A$ an algebra in $\Cc$, $\Dc$ a right 
$\Cc$-category and $X$ an object of $\Cc$. Suppose that $A, X$ are (left) $\Dc$-coflat 
and left coflat objects of $\Cc$. Then $\mfC=A\ot X$ has an $A$-(co)ring structure in 
$\Cc$ compatible with $\Dc$ if and only if $X$ has a (co)algebra structure in $EM(\Cc)(A)$. 
\end{theorem} 

\begin{proof}
We show that the conditions in \prref{3.18} are equivalent to $X$ being a 
(co)algebra in $EM(\Cc)(A)$. Denoting 
$\psi=
\gbeg{2}{3}
\got{1}{X}\got{1}{A}\gnl
\gbrc\gnl
\gob{1}{A}\gob{1}{X}
\gend\hspace{2mm},$ 
in diagrammatic notation, (\ref{eq:e1}-\ref{eq:e2})
reduce to \equref{ta}. Hence (\ref{eq:e1}-\ref{eq:e2}) are equivalent to 
the fact that $(X, \psi)$ is an object of $EM(\Cc)(A)$. Now we write
$$\d_X=
\gbeg{3}{5}
\got{3}{X}\gnl
\gvac{1}\gcl{1}\gnl
\gsbox{3}\gnl
\gcl{1}\gcl{1}\gcl{1}\gnl
\gob{1}{A}\gob{1}{X}\gob{1}{X}
\gend\hspace*{2mm};\hspace*{2mm}
f=
\gbeg{1}{3}
\got{1}{X}\gnl
\gmp{f}\gnl
\gob{1}{A}
\gend\hspace*{2mm};\hspace*{2mm} 
\zeta_X=
\gbeg{2}{5}
\got{1}{X}\got{1}{X}\gnl
\gcl{1}\gcl{1}\gnl
\gsbox{2}\gnl
\gcl{1}\gcl{1}\gnl
\gob{1}{A}\gob{1}{X}
\gend\hspace*{2mm};\hspace*{2mm}
\sigma=
\gbeg{2}{6}
\got{2}{\un{1}}\gnl
\gu{1}\gnl
\gcn{1}{1}{1}{2}\gnl
\gsbox{2}\gnl
\gcl{1}\gcl{1}\gnl
\gob{1}{A}\gob{1}{X}
\gend\hspace*{2mm}.$$
\equref{righAlinmult} and \equref{Multass} take the form 
\begin{equation}\eqlabel{apdf1}
\gbeg{3}{6}
\got{1}{X}\got{1}{X}\got{1}{A}\gnl
\gcl{1}\gcl{1}\gcl{1}\gnl
\gsbox{2}\gvac{2}\gcl{1}\gnl
\gcl{1}\gbrc\gnl\gmu\gcl{1}\gnl
\gob{2}{A}\gob{1}{X}
\gend
=
\gbeg{3}{6}
\got{1}{X}\got{1}{X}\got{1}{A}\gnl
\gcl{1}\gbrc\gnl
\gbrc\gcl{1}\gnl
\gcl{1}\gsbox{2}\gnl
\gmu\gcl{1}\gnl
\gob{2}{A}\gob{1}{X}
\gend 
\hspace{3mm}\mbox{~~and~~}\hspace{3mm}
\gbeg{3}{7}
\got{1}{X}\got{1}{X}\got{1}{X}\gnl
\gcl{1}\gcl{1}\gcl{1}\gnl
\gcl{1}\gsbox{2}\gnl
\gbrc\gcl{1}\gnl
\gcl{1}\gsbox{2}\gnl
\gmu\gcl{1}\gnl
\gob{2}{A}\gob{1}{X}
\gend
=
\gbeg{3}{7}
\got{1}{X}\got{1}{X}\got{1}{X}\gnl
\gcl{1}\gcl{1}\gcl{1}\gnl
\gsbox{2}\gvac{2}\gcl{1}\gnl
\gcl{1}\gcl{1}\gcl{1}\gnl
\gcl{1}\gsbox{2}\gnl
\gmu\gcl{1}\gnl
\gob{2}{A}\gob{1}{X}
\gend .
\end{equation}
\equref{rightAlin} and 
\equref{Deltacoass} reduce to  
\begin{equation}\eqlabel{pdf1}
\gbeg{4}{8}
\got{3}{X}\got{1}{A}\gnl
\gvac{1}\gcl{1}\gvac{1}\gcl{2}\gnl
\gsbox{3}\gnl
\gcl{1}\gcl{1}\gcl{1}\gcl{1}\gnl
\gcl{1}\gcl{1}\gbrc\gnl
\gcl{1}\gbrc\gcl{1}\gnl
\gmu\gcl{1}\gcl{1}\gnl
\gob{2}{A}\gob{1}{X}\gob{1}{X}
\gend =
\gbeg{4}{7}
\got{1}{X}\got{1}{A}\gnl
\gbrc\gnl
\gcl{1}\gcn{1}{1}{1}{3}\gnl
\gcl{1}\gsbox{3}\gnl
\gcl{1}\gcl{1}\gcl{1}\gcl{1}\gnl
\gmu\gcl{1}\gcl{1}\gnl
\gob{2}{A}\gob{1}{X}\gob{1}{X}
\gend
\hspace{3mm}\mbox{~~and~~}\hspace{3mm}
\gbeg{5}{8}
\got{5}{X}\gnl
\gvac{2}\gcl{1}\gnl
\gvac{1}\gsbox{3}\gnl
\gcn{1}{1}{3}{1}\gvac{1}\gcl{1}\gcn{1}{1}{1}{3}\gnl
\gcl{1}\gsbox{3}\gvac{3}\gcl{1}\gnl
\gcl{1}\gcl{1}\gcl{1}\gcl{1}\gcl{1}\gnl
\gmu\gcl{1}\gcl{1}\gcl{1}\gnl
\gob{2}{A}\gob{1}{X}\gob{1}{X}\gob{1}{X}
\gend =
\gbeg{5}{9}
\got{3}{X}\gnl
\gvac{1}\gcl{1}\gnl
\gsbox{3}\gnl
\gcl{1}\gcl{1}\gcn{1}{1}{1}{3}\gnl
\gcl{1}\gcl{1}\gsbox{3}\gnl
\gcl{1}\gcl{1}\gcl{1}\gcl{1}\gcl{1}\gnl
\gcl{1}\gbrc\gcl{1}\gcl{1}\gnl
\gmu\gcl{1}\gcl{1}\gcl{1}\gnl
\gob{2}{A}\gob{1}{X}\gob{1}{X}\gob{1}{X}
\gend \hspace{2mm}.
\end{equation}  
In a similar way, \equref{unitAbilin}, \equref{unit1} and \equref{unit2} read as 
\begin{equation}\eqlabel{apdf2}
\gbeg{3}{6}
\got{1}{A}\gnl
\gcl{1}\gu{1}\gnl
\gcl{1}\gcn{1}{1}{1}{2}\gnl
\gcl{1}\gsbox{2}\gnl
\gmu\gcl{1}\gnl
\gob{2}{A}\gob{1}{X}
\gend
=
\gbeg{3}{7}
\gvac{2}\got{1}{A}\gnl
\gu{1}\gvac{1}\gcl{1}\gnl
\gcn{1}{1}{1}{2}\gvac{1}\gcl{1}\gnl
\gsbox{2}\gvac{2}\gcl{1}\gnl
\gcl{1}\gbrc\gnl
\gmu\gcl{1}\gnl
\gob{2}{A}\gob{1}{X}
\gend
\hspace{2mm},\hspace{2mm}
\gbeg{3}{8}
\got{1}{X}\gnl
\gcl{1}\gu{1}\gnl
\gcl{1}\gcn{1}{1}{1}{2}\gnl
\gcl{1}\gsbox{2}\gnl
\gbrc\gcl{1}\gnl
\gcl{1}\gsbox{2}\gnl
\gmu\gcl{1}\gnl
\gob{2}{A}\gob{1}{X}
\gend
=
\gbeg{2}{3}
\gvac{1}\got{1}{X}\gnl
\gu{1}\gcl{1}\gnl
\gob{1}{A}\gob{1}{X}
\gend
\hspace{2mm},\hspace{2mm}
\gbeg{3}{8}
\gvac{2}\got{1}{X}\gnl
\gu{1}\gvac{1}\gcl{1}\gnl
\gcn{1}{1}{1}{2}\gvac{1}\gcl{1}\gnl
\gsbox{2}\gvac{2}\gcl{1}\gnl
\gcl{1}\gcl{1}\gcl{1}\gnl
\gcl{1}\gsbox{2}\gnl
\gmu\gcl{1}\gnl
\gob{2}{A}\gob{1}{X}
\gend
=
\gbeg{2}{3}
\gvac{1}\got{1}{X}\gnl
\gu{1}\gcl{1}\gnl
\gob{1}{A}\gob{1}{X}
\gend , 
\end{equation}
while \equref{counitAbilin}, \equref{counit1} and \equref{counit2}
can be written as
\begin{equation}\eqlabel{pdf2}
\gbeg{2}{5}
\got{1}{X}\got{1}{A}\gnl
\gbrc\gnl
\gcl{1}\gmp{f}\gnl
\gmu\gnl
\gob{2}{A}
\gend = 
\gbeg{2}{4}
\got{1}{X}\got{1}{A}\gnl
\gmp{f}\gcl{1}\gnl
\gmu\gnl
\gob{2}{A}
\gend
\hspace{2mm},\hspace{2mm}
\gbeg{3}{6}
\got{3}{X}\gnl
\gvac{1}\gcl{1}\gnl
\gsbox{3}\gnl
\gcl{1}\gmp{f}\gcl{1}\gnl
\gmu\gcl{1}\gnl
\gob{2}{A}\gob{1}{X}
\gend =
\gbeg{2}{3}
\got{3}{X}\gnl
\gu{1}\gcl{1}\gnl
\gob{1}{A}\gob{1}{X}
\gend
\hspace{3mm}\mbox{~~and~~}\hspace{3mm}
\gbeg{3}{7}
\got{3}{X}\gnl
\gvac{1}\gcl{1}\gnl
\gsbox{3}\gnl
\gcl{1}\gcl{1}\gmp{f}\gnl
\gcl{1}\gbrc\gnl
\gmu\gcl{1}\gnl
\gob{2}{A}\gob{1}{X}
\gend = 
\gbeg{2}{3}
\got{3}{X}\gnl
\gu{1}\gcl{1}\gnl
\gob{1}{A}\gob{1}{X}
\gend
\hspace{2mm}.
\end{equation}
It can be easily checked that the first equality 
in \equref{apdf1} is equivalent to the fact that the map 
$\un{m}_X: (X, \psi)\un{\ot}(X, \psi)\ra (X, \psi)$ defined by 
$\un{m}_X=\zeta_X: X\ot X\ra A\ot X$ is a morphism in $EM(\Cc)(A)$, 
and that \equref{pdf1} is equivalent to the fact that 
$\un{\Delta}_X: (X, \psi)\ra (X, \psi)\un{\ot}(X, \psi)$ defined by 
$\un{\Delta}_X=\d_X: X\ra A\ot X\ot X$ is a morphism in $EM(\Cc)(A)$.\\
We now show that $\un{\Delta}_X$ is coassociative if and only if the second equality in 
\equref{pdf1} holds:
$$
\Id_{(X, \psi)}\un{\ot}\un{\Delta}_X=
\gbeg{5}{6}
\gvac{1}\got{1}{X}\gvac{1}\got{1}{X}\gnl
\gvac{1}\gcl{1}\gvac{1}\gcl{1}\gnl
\gu{1}\gcl{1}\gsbox{3}\gnl
\gcl{1}\gbrc\gcl{1}\gcl{1}\gnl
\gmu\gcl{1}\gcl{1}\gcl{1}\gnl
\gob{2}{A}\gob{1}{X}\gob{1}{X}\gob{1}{X}
\gend =
\gbeg{4}{5}
\got{1}{X}\gvac{1}\got{1}{X}\gnl
\gcl{1}\gvac{1}\gcl{1}\gnl
\gcl{1}\gsbox{3}\gnl
\gbrc\gcl{1}\gcl{1}\gnl
\gob{1}{A}\gob{1}{X}\gob{1}{X}\gob{1}{X}
\gend\hspace*{2mm},
$$
so
$$(\Id_{(X, \psi)}\un{\ot}\un{\Delta}_X)\circ \un{\Delta}_X=
\gbeg{5}{9}
\got{3}{X}\gnl
\gvac{1}\gcl{1}\gnl
\gsbox{3}\gnl
\gcl{1}\gcl{1}\gcn{1}{1}{1}{3}\gnl
\gcl{1}\gcl{1}\gsbox{3}\gnl
\gcl{1}\gcl{1}\gcl{1}\gcl{1}\gcl{1}\gnl
\gcl{1}\gbrc\gcl{1}\gcl{1}\gnl
\gmu\gcl{1}\gcl{1}\gcl{1}\gnl
\gob{2}{A}\gob{1}{X}\gob{1}{X}\gob{1}{X}
\gend \hspace{2mm}.
$$
In a similar way, we have that
\[
\un{\Delta}_X\un{\ot}\Id_{(X, \psi)}=
\gbeg{4}{5}
\gvac{1}\got{1}{X}\gvac{1}\got{1}{X}\gnl
\gvac{1}\gcl{1}\gvac{1}\gcl{1}\gnl
\gsbox{3}\gvac{3}\gcl{1}\gnl
\gcl{1}\gcl{1}\gcl{1}\gcl{1}\gnl
\gob{1}{A}\gob{1}{X}\gob{1}{X}\gob{1}{X}
\gend
,{~\rm so~}\hspace{2mm}
(\un{\Delta}_X\un{\ot}\Id_{(X, \psi)})\circ \un{\Delta}_X=
\gbeg{5}{8}
\got{5}{X}\gnl
\gvac{2}\gcl{1}\gnl
\gvac{1}\gsbox{3}\gnl
\gcn{1}{1}{3}{1}\gvac{1}\gcl{1}\gcn{1}{1}{1}{3}\gnl
\gcl{1}\gsbox{3}\gvac{3}\gcl{1}\gnl
\gcl{1}\gcl{1}\gcl{1}\gcl{1}\gcl{1}\gnl
\gmu\gcl{1}\gcl{1}\gcl{1}\gnl
\gob{2}{A}\gob{1}{X}\gob{1}{X}\gob{1}{X}
\gend\hspace{2mm}.
\]
Comparing these two equalities, we obtain that coassociativity of $\un{\Delta}_X$
is equivalent to the second equality in \equref{pdf1}. In a similar way,
associativity of $\un{m}_X$ is equivalent to the second equality in \equref{apdf1}.\\
Finally, assume that $\un{\eta}_X: (\un{1}, \Id_A)\ra (X, \psi)$ and
$\un{\va}_X: (X, \psi)\ra (\un{1}, \Id_A)$ correspond respectively to 
$\sigma: \un{1}\ra A\ot X$ and $f: X\ra A$ in $\Cc$, as in \leref{3.5}. 
Then $\un{\eta}_X$ is a morphism in $EM(\Cc)(A)$ if and only if 
the first equality in \equref{apdf2} is satisfied, and $\un{\va}_X$ is a morphism in 
$EM(\Cc)(A)$ if and only if the first equality in \equref{pdf2} holds. 
Composing the equalities
\[
\un{\va}_X\un{\ot}\Id_{(X, \psi)}=
\gbeg{3}{4}
\got{1}{X}\gvac{1}\got{1}{X}\gnl
\gmp{f}\gu{1}\gcl{1}\gnl
\gmu\gcl{1}\gnl
\gob{2}{A}\gob{1}{X}
\gend =
\gbeg{2}{3}
\got{1}{X}\got{1}{X}\gnl
\gmp{f}\gcl{1}\gnl
\gob{1}{A}\gob{1}{X}
\gend
\hspace{2mm}{\rm and}\hspace{2mm}
\Id_{(X, \psi)}\un{\ot}\un{\va}_X=
\gbeg{3}{5}
\gvac{1}\got{1}{X}\got{1}{X}\gnl
\gu{1}\gcl{1}\gmp{f}\gnl
\gcl{1}\gbrc\gnl
\gmu\gcl{1}\gnl
\gob{2}{A}\gob{1}{X}
\gend =
\gbeg{2}{4}
\got{1}{X}\got{1}{X}\gnl
\gcl{1}\gmp{f}\gnl
\gbrc\gnl
\gob{1}{A}\gob{1}{X}\gnl
\gend
\]
with $\un{\Delta}_X$, we obtain that
\[
(\un{\va}_X\un{\ot}\Id_{(X, \psi)})\circ \un{\Delta}_X=
\gbeg{3}{6}
\got{3}{X}\gnl
\gvac{1}\gcl{1}\gnl
\gsbox{3}\gnl
\gcl{1}\gmp{f}\gcl{1}\gnl
\gmu\gcl{1}\gnl
\gob{2}{A}\gob{1}{X}
\gend
\hspace{3mm}{\rm and}\hspace{3mm}
(\Id_{(X, \psi)}\un{\ot}\un{\va}_X)\circ \un{\Delta}_X=
\gbeg{3}{7}
\got{3}{X}\gnl
\gvac{1}\gcl{1}\gnl
\gsbox{3}\gnl
\gcl{1}\gcl{1}\gmp{f}\gnl
\gcl{1}\gbrc\gnl
\gmu\gcl{1}\gnl
\gob{2}{A}\gob{1}{X}
\gend\hspace{2mm}.
\]
Therefore, $(\un{\va}_X\un{\ot}\Id_{(X, \psi)})\circ \un{\Delta}_X=\Id_{(X, \psi)}$ 
if and only if the second equality in \equref{pdf2} holds, and 
$(\Id_{(X, \psi)}\un{\ot}\un{\va}_X)\circ \un{\Delta}_X=\Id_{(X, \psi)}$ 
if and only if the third equality in \equref{pdf2} holds. In a similar manner 
we can show that the unit property for $\un{\eta}_X=\sigma$ is equivalent to 
the second and the third equality in \equref{apdf2}. 
\end{proof} 

It is well-known fact in classical Hopf algebra theory that particular examples of 
$A$-(co)ring structures of the form $A\ot X$ can be obtained in the situation when $X$ is a 
(co)algebra entwined with $A$. We end this section by showing that 
this situation occurs in the particular case when $\zeta_X=\un{\eta}_A\ot \un{m}_X$, 
$\sigma=\un{\eta}_A\ot\un{\eta}_X$ for some $\un{m}_X:X\ot X\ra X$, $\un{\eta}_X:\un{1}\ra X$ in $\Cc$, 
respectively when $\delta_X=\un{\eta}_A\ot \un{\Delta}_X$, $f=\un{\eta}_A\un{\va}_X$ 
for some morphism $\un{\Delta}_X:\ X\ra X\ot X$, $\un{\va}_X: X\ra \un{1}$ in $\Cc$. Actually, 
if this is the case then     
\begin{equation}\eqlabel{aass}
\un{m}_\mfC=(\Id_A\ot \un{m}_X)\gamma_{\mfC, X}~,~\un{\eta}_\mfC=\Id_A\ot \un{\eta}_X,
\end{equation}
and respectively
\begin{equation}\eqlabel{ass} 
\un{\Delta}_{\mfC}=\Upsilon^{-1}_{\mfC, X}(\Id_A\ot \un{\Delta}_X)~,~\un{\va}_\mfC=\Id_A\ot\un{\va}_X.  
\end{equation} 

\begin{proposition}\prlabel{fact(co)ring}
Let $A, X$ be left coflat objects in $\Cc$ such that $A$ carries an algebra structure in $\Cc$ and 
$\un{\eta}_A\ot \Id_Y$ is a monomorphism, for any $Y\in \Cc$. 
\begin{itemize}
\item[(i)] $(\mfC, \un{m}_\mfC, \un{\eta}_\mfC)$, with $\un{m}_\mfC$ and $\un{\eta}_\mfC$ defined 
by \equref{aass} is an $A$-ring if 
and only if $(X, \un{m}_X, \un{\eta}_X)$ is an algebra in $\Cc$ and there exists $\psi: X\ot A\ra A\ot X$ 
such that 
\begin{eqnarray*}
&&\psi(\Id_X\ot\un{m}_A)=(\psi\ot \Id_A)(\Id_A\ot\psi)(\un{m}_A\ot\Id_X)~,~
\psi(\Id_X\ot\un{\eta}_A)=\un{\eta}_A\ot \Id_X,\\
&&\psi(\un{m}_X\ot \Id_A)=(\Id_X\ot \psi)(\psi\ot \Id_X)(\Id_A\ot \un{m}_X)~,~
\psi(\un{\eta}_X\ot \Id_A)=\Id_A\ot \un{\eta}_X~.
\end{eqnarray*}   
\item[(ii)] $(\mfC,\un{\Delta}_\mfC, \un{\va}_\mfC)$, with $\un{\Delta}_\mfC$ and $\un{\va}_\mfC$ 
as in \equref{ass}, is an $A$-coring if and 
only if $(X, \un{\Delta}_X, \un{\va}_X)$ is a coalgebra in $\Cc$ and there exists a morphism 
$\psi: X\ot A\ra A\ot X$ in $\Cc$ such that 
\begin{eqnarray*}
&&\psi(\Id_X\ot\un{m}_A)=(\psi\ot \Id_A)(\Id_A\ot\psi)(\un{m}_A\ot\Id_X)~,~
\psi(\Id_X\ot\un{\eta}_A)=\un{\eta}_A\ot \Id_X~,\\
&&(\psi \ot\Id_X)(\Id_X\ot\psi)(\un{\Delta}_X\ot\Id_A)=
(\Id_A\ot\un{\Delta}_X)\psi~,~
(\Id_A\ot \un{\va}_X)\psi=\un{\va}_X\ot \Id_A~.
\end{eqnarray*}
\end{itemize}
\end{proposition}

\begin{proof}
This is a consequence of \prref{3.18}. 
The first two equalities in (i) and (ii) coincide, and are imposed by \leref{TaSchstr}. 
Then \equref{righAlinmult}, the condition that $\un{m}_\mfC$ is right $A$-linear, simplifies  
to the third condition in (i). Right $A$-linearity of $\Delta_\mfC$ is equivalent to
\equref{rightAlin}, which reduces to the third equality in (ii).
The fourth equality 
in (i) is equivalent to \equref{unitAbilin}, which is equivalent to right $A$-linearity of $\un{\eta}_\mfC$.
Similarly, the fourth equality in (ii) is equivalent to \equref{counitAbilin}, the condition that is needed
to make $\un{\va}_\mfC$ right $A$-linear. \\
If $\un{m}_\mfC$ is given by \equref{aass}, then $\un{m}_\mfC$ is associative if and only if
\[
(\un{\eta}_A\ot \Id_X)\un{m}_X(\un{m}_X\ot \Id_X)=
(\un{\eta}_A\ot \Id_X)\un{m}_X(\Id_X\ot \un{m}_X).
\]
Using the assumption that $\un{\eta}_A\ot \Id_X$ is monic, we obtain that
$\un{m}_X$ is associative if and only if $\un{m}_\mfC$ is associative.\\
Similarly, if $\un{\Delta}_X$ is defined by \equref{ass}, then \equref{Deltacoass} reduces to 
\[
\Id_A\ot(\Id_X\ot\un{\Delta}_X)\un{\Delta}_X=
\Id_A\ot(\un{\Delta}_X\ot\Id_X)\un{\Delta}_X,
\]
and therefore the coassociativity of $\un{\Delta}_X$ implies the coassociativity of $\un{\Delta}_{\mfC}$. 
Conversely, if $\un{\Delta}_{\mfC}$ is coassociative then the above equality and the naturality 
of $\ot$ imply that 
\[
(\un{\eta}_A\ot\Id_{X^{\ot 3}})(\Id_X\ot\un{\Delta}_X)\un{\Delta}_X
=(\un{\eta}_A\ot\Id_{X^{\ot 3}})(\un{\Delta}_X\ot\Id_X)\un{\Delta}_X.
\] 
By assumption, $\un{\eta}_A\ot\Id_{X^{\ot 3}}$
is monic, so $\un{\Delta}_X$ is coassociative.\\
Finally, it is easy to see that the counit property of $\un{\va}_X$ (with respect to $\un{\Delta}_X$) implies  
the counit property for $\un{\va}_{\mfC}$ (with respect to $\un{\Delta}_{\mfC}$). Conversely, 
the assumption that $\un{\eta}_A\ot\Id_X$ is monic implies that
the counit property of $\un{\Delta}_X$ follows from the counit property of $\un{\va}_{\mfC}$.\\
The equivalence of the unit property of $\un{\eta}_\mfC$ with respect to $\un{m}_\mfC$
and the unit property of $\un{\eta}_X$ with respect to $\un{m}_X$ follows using
similar arguments. 
\end{proof}

\begin{remark}
The assumption that $\un{\eta}_A\ot\Id_X$ is monic is not needed in the two converse
implications in \prref{fact(co)ring}: if $X$ is a (co)algebra in $\Cc$ satisfying the
four conditions in (i) and (ii), then $\mfC$ is always an $A$-(co)ring. For the direct
implications, it suffices to assume that 
$\un{\eta}_A\ot \Id_X$ and $\un{\eta}_A\ot\Id_{X^{\ot 3}}$ are monic. 
\end{remark}

We finish this Section with a monoidal interpretation of \prref{fact(co)ring}. First we
recall from \cite{Ross, LackRoss} that we can associate a $2$-category ${\rm Mnd}({\cal K})$
to any 2-category ${\cal K}$. 
The $0$-cells and $1$-cells of ${\rm Mnd}({\cal K})$ are the same as those of $EM({\cal K})$, that is, 
0-cells are monads in ${\cal K}$ and 1-cells are monad morphisms. The 2-cells are defined in a
different way: if $(f, \psi)$ and $(g, \phi)$ are 1-cells 
$\mathbb{A}=(A, t, \mu_t, \eta_t)\to\mathbb{B}=(B, s, \mu_s, \eta_s)$, then a $2$-cell 
$\xymatrix{
(f, \psi)\ar@2{->}[r]^{\rho}&(g, \phi)
}
$ is a so-called monad transformation, that
is a $2$-cell 
$\xymatrix{
f\ar@2{->}[r]^{\rho}&g
}
$ 
in ${\cal K}$ such that 
\[
(\rho\odot 1_t)\circ \psi=\phi\circ (1_s\odot \rho).
\]
The vertical composition of  $2$-cells is given by 
vertical composition in ${\cal K}$ and the horizontal composition of $2$-cells 
\[
\xymatrix{
\mathbb{A}\rtwocell^{(f, \psi)}_{(f', \psi')}{\rho}&\mathbb{B}
\rtwocell^{(g, \phi)}_{(g', \phi')}{\rho'}&\mathbb{C}
}
~{\rm is}~
\xymatrix{
gf\ar@2{->}[r]^{\rho'\oslash \rho}&g'f'
},~
\rho'\oslash \rho:=(1_{g'}\odot \rho)\circ (\rho'\odot 1_f), 
\]
where $(g, \phi)(f, \psi)=(gf, (1_g\odot \psi)\circ (\phi\odot 1_f))$, etc. 
We also have that $1_{(f, \psi)}=1_f$ and 
$(1_\mathbb{A}, i_\mathbb{A})=((1_A, 1_t), i_A)$, for any monad $\mathbb{A}=(A, t, \mu_t, \eta_t)$ 
in ${\cal K}$. In \prref{4.6}, we describe the 2-category ${\rm Mnd}(\Cc)$ corresponding to a
2-category with one 0-cell, that is a monoidal category $\Cc$. We omit the proof, as it is similar
to the proof of \prref{TA2}.

\begin{proposition}\prlabel{4.6}
Let $\Cc$ be a monoidal category. The 0-cells and 1-cells of ${\rm Mnd}(\Cc)$ coincide with
those of $EM(\Cc)$, see \prref{TA2}. A $2$-cell 
$
\xymatrix{
(X, \psi)\ar@2{->}[r]^\rho&(Y, \phi)
}
$ 
is a morphism $\rho : X\ra Y$ such that $\phi(\rho\ot \Id_B)=(\Id_A\ot\rho)\psi$.
Vertical composition of 2-cells is given by composition of morphisms in $\Cc$.
The horizontal composition of two 2-cells
\[
\xymatrix{
A\rtwocell^{(X, \psi)}_{(X', \psi')}{\rho}&B\rtwocell^{(Y, \phi)}_{(Y', \phi')}{\rho'}&C
}
\mbox{~is defined by~} 
(Y, \phi)(X, \psi)=\left(X\ot Y, 
\gbeg{3}{4}
\got{1}{X}\got{1}{Y}\got{1}{C}\gnl
\gcl{1}\gbrbox\gnl
\gbrc\gcl{1}\gnl
\gob{1}{A}\gob{1}{X}\gob{1}{Y} 
\gend
\right),~{\rm etc.}  
\]
and $\rho'\oslash \rho:=\rho\ot \rho':X\ot Y\ra X'\ot Y'$. Furthermore, for any algebra $A$ in 
$\Cc$ we have $1_A=(\un{1}, \Id_A)$ and $i_A=\Id_{\un{1}}$, and for 
any $1$-cell 
$
\xymatrix{
A\ar[r]^-{(X, \psi)}&A
}
$ 
we have $1_{(X, \psi)}=\Id_X$. 
\end{proposition}

As the reader might expect, the $A$-(co)ring structures in \prref{fact(co)ring} 
are related to the monoidal structure of ${\rm Mnd}(\Cc)(A)$. These structures can be also viewed 
in $EM(\Cc)(A)$ via the monoidal functor $F: {\rm Mnd}(\Cc)(A)\ra EM(\Cc)(A)$ that acts as the identity 
on objects and sends a morphism $f$ in ${\rm Mnd}(\Cc)(A)$ to $F(f)=\un{\eta}_A\ot f$. $F$ comes
from the $2$-functor $E: {\rm Mnd}({\cal K})\ra EM({\cal K})$ which is the identity 
on objects and $1$-cells and sends a $2$-cell 
$\xymatrix{
(f, \psi)\ar@2{->}[r]^\rho&(g, \phi)
}
$ 
in ${\rm Mnd}(\Cc)$ to 
$\xymatrix{
(f, \psi)\ar@2{->}[r]^{(1_g\odot \eta)\circ \rho}&(g, \phi)
}
$, 
a $2$-cell in $EM({\cal K})$. Now the proof of \prref{4.7} is left to the reader as a
straightforward exercise.

\begin{proposition}\prlabel{4.7}
Let $A, X$ be left coflat objects in $\Cc$, and assume that $A$ carries an algebra structure 
in $\Cc$ such that $\un{\eta}_A\ot \Id_Y$ is monic, for any $Y\in \Cc$.  
\begin{itemize}
\item[(i)] $(\mfC,\un{m}_\mfC, \un{\eta}_\mfC)$, with $\un{m}_\mfC$ and $\un{\eta}_\mfC$ defined 
by \equref{aass}, is an $A$-ring if 
and only if $(X, \un{m}_X, \un{\eta}_X)$ is an algebra in ${\rm Mnd}(\Cc)(A)$.
\item[(ii)] $(\mfC,\un{\Delta}_\mfC,\un{\va}_\mfC)$, with $\un{\Delta}_\mfC$ and $\un{\va}_\mfC$ as in \equref{ass}, is an $A$-coring if and 
only if $(X, \un{\Delta}_X, \un{\va}_X)$ is a coalgebra in ${\rm Mnd}(\Cc)(A)$. 
\end{itemize}  
\end{proposition}

\section{Wreath products and categories of (co)representations}\selabel{(co)wreaths}
\setcounter{equation}{0}
The monoidal category ${\rm Mnd}(\Cc)(A)$ appears already in the work of Tambara \cite{tambara},
where it is termed the category of transfer morphisms. Inspired by this terminology, and by the paper
\cite{sch}, we introduce the notation ${\cal T}_A={\rm Mnd}(\Cc)(A)$, and, by analogy,
${\cal T}_A^\#=EM(\Cc)(A)$.
It is known that algebras in ${\cal T}_A$ are in bijective correspondence with 
cross algebra product structures,  see for example \cite[Proposition 2.1]{bcCPHA}.
The aim of this Section is to show that algebras in ${\cal T}_A^\#$ are in bijective correspondence
with the so-called wreath products. We will also discuss coalgebras in ${\cal T}_A^\#$.
Recall that a wreath in a  $2$-category ${\cal K}$ is a monad in $EM({\cal K})$.
According to \cite{LackRoss}, 
a wreath is a monad $\mathbb{A}=(A, t, \mu, \eta)$ in ${\cal K}$ together with a $1$-cell 
$
\xymatrix{
A\ar[r]^s&A
}
$  
and $2$-cells 
$
\xymatrix{
ts\ar@2{->}[r]^\psi&st
}, 
$
$
\xymatrix{
1_A\ar@2{->}[r]^\sigma &st
} 
$ 
and 
$
\xymatrix{
ss\ar@2[r]^\zeta&st
}
$ 
satisfying the following conditions:
\begin{eqnarray}
&&(1_s\odot \mu)\circ (\psi\odot 1_t)\circ (1_t\odot \psi)=\psi\circ (\mu\odot 1_s)~,~
\psi\circ (\eta\odot 1_s)=1_s\odot \eta~;\eqlabel{wr1}\\
&&(1_s\odot \mu)\circ (\psi\odot 1_t)\circ (1_t\odot \sigma)=(1_s\odot \mu)\circ (\sigma\odot 1_t)~;\eqlabel{wr2}\\
&&(1_s\odot \mu)\circ (\psi\odot 1_t)\circ (1_t\odot \zeta)=(1_s\odot \mu)\circ (\zeta\odot 1_t)\circ 
(1_s\odot \psi)\circ (\psi\odot 1_s)~;\eqlabel{wr3}\\
&&(1_s\odot \mu)\circ (\zeta\odot 1_t)\circ (1_s\odot\zeta)=(1_s\odot \mu)\circ (\zeta\odot 1_t)\circ (1_s\odot \psi)
\circ (\zeta\odot 1_s)~;\eqlabel{wr4}\\
&&(1_s\odot \mu)\circ (\zeta\odot 1_t)\circ (1_s\odot \sigma)=1_s\odot \eta~;\eqlabel{wr5}\\
&&(1_s\odot \mu)\circ (\zeta\odot 1_t)\circ (1_s\odot \psi)\circ (\sigma\odot 1_s)=1_s\odot \eta~.\eqlabel{wr6} 
\end{eqnarray}
Let us introduce the dual notion. A triple 
$
(C, 
\xymatrix{
C\ar[r]^{t}&C
}, 
\xymatrix{
tt\ar@2{->}[r]^\delta&t
},  
\xymatrix{
C\ar@2{->}[r]^f&1_C
}
)
$ 
is called a comonad if
\[
(1_t\odot \delta)\circ \delta=(\delta\odot 1_t)\circ \delta~\mbox{\rm and}~
(f\odot 1_t)\circ \delta=1_t=(1_t\odot f)\circ \delta.
\]
A comonad in $EM({\cal K})$ is called a cowreath. 
It can be described as a monad 
$\mathbb{A}=(A, t, \mu, \eta)$ in ${\cal K}$ equipped with a $1$-cell 
$
\xymatrix{
A\ar[r]^s&A
}
$ 
and 2-cells 
$
\xymatrix{
ts\ar@2{->}[r]^\psi&st
}
$, 
$
\xymatrix{
s\ar@2{->}[r]^\delta &sst
}
$ 
and 
$\xymatrix{
s\ar@2{->}[r]^f&t
}
$
such that (\ref{eq:wr1}-\ref{eq:wr2}) and (\ref{eq:cwr2}-\ref{eq:cwr6}) are satisfied.
\begin{eqnarray}
&&(1_s\odot 1_s\odot \mu)\circ (\delta\odot 1_t)\circ \psi\nonumber\\
&&\hspace*{1cm}
=(1_s\odot 1_s\odot \mu)\circ (1_s\odot \psi \odot 1_t)\circ (\psi\odot 1_s\odot 1_t)\circ (1_t\odot \delta)~;\eqlabel{cwr2}\\
&&\mu\circ (1_t\odot f)=\mu\circ (f\odot 1_t)\circ \psi~;\eqlabel{cwr3}\\
&&(1_s\odot 1_s\odot 1_s\odot \mu)\circ (1_s\odot 1_s\odot \psi\odot 1_t)\circ (\delta\odot 1_s\odot 1_t)\circ \delta\nonumber\\
&&\hspace*{1cm}
=(1_s\odot 1_s\odot 1_s\odot \mu)\circ (1_s\odot \delta\odot 1_t)\circ \delta~;\eqlabel{cwr4}\\
&&(1_s\odot \mu)\circ (\psi\odot 1_t)\circ (f\odot 1_s\odot 1_t)\circ \delta =1_s\odot \eta~;\eqlabel{cwr5}\\  
&&(1_s\odot \mu)\circ (1_s\odot f\odot 1_t)=1_s\odot \eta~.\eqlabel{cwr6}
\end{eqnarray}

We now look at wreats and cowreaths in a monoidal category.

\begin{proposition}\prlabel{5.1}
A  (co)wreath in a monoidal category $\Cc$ is a pair $(A, X)$, where $A$ is an algebra in $\Cc$,
an $X$ is a (co)algebra in ${\cal T}_A^\#$. Consequently, (co)ring structures of the form
$A\ot X$ studied in Sections \ref{se:strofcor} and \ref{se:monint}
are in bijective correspondence with (co)wreath structures in $\Cc$.   
\end{proposition}

\begin{proof}
As we have seen in the proof of \prref{TA2}, giving a monad in $\Cc$ is equivalent to giving
 an algebra $A$ in $\Cc$. 
Moreover, a $1$-cell 
$
\xymatrix{
{*}\ar[r]^s&{*}
} 
$ 
is an object $X$ of $\Cc$ and the required $2$-cells $\psi$, $\sigma$, $\zeta$ in the definition of a wreath 
are morphisms $\psi: X\ot A\ra A\ot X$, $\sigma: \un{1}\ra A\ot X$, $\zeta: X\ot X\ra A\ot X$ in $\Cc$. 
It then comes out that \equref{wr1} is \equref{ta}, \equref{wr2} is the first equality in \equref{apdf2}, 
\equref{wr3} is the first equality in \equref{apdf1}, \equref{wr4} is the second equality in \equref{apdf1}, 
\equref{wr5} is the last equality in \equref{apdf2} and \equref{wr6} is the second equality in \equref{apdf2}.
The dual statement about cowreaths can be proved in a similar way, we leave the details to the reader. 
\end{proof}

Now we discuss (co)representations of  (co)rings defined by (co)wreaths. Since the computations are
rather lenghty, we decided to divide them into several lemma.\\
If $\Dc$ is a right $\Cc$-category then a right module over an $A$-ring $\mfC$ is an algebra for the 
monad $-\diamond_A\mfC$ on the category $\Dc_A$. Explicitly, it is a right module $\mfM$ in $\Dc$ 
over $A$ together with a morphism $\nu_\mfM^\mfC: \mfM\diamond_A\mfC\ra \mfM$ in $\Dc_A$ which 
is associative and unital modulo the $\Cc$-category structure of $\Dc$ and the ring structure of $\mfC$. 

\begin{lemma}\lelabel{5.2}
Let $\Dc$ be a right $\Cc$-category and $(A, X)$ a wreath in $\Cc$ such that the 
$A$-coring $\mfC=A\ot X$ is compatible with the $\Cc$-category structure of $\Dc$. Then 
a right $\mfC$-module is an object $\mfM$ in $\Dc_A$ (with structure morphism $\nu_\mfM^A: \mfM\diamond A\ra \mfM$)
 equipped with a right 
$X$-action $\mu_\mfM^X: \mfM\diamond X\ra \mfM$ such that 
\begin{eqnarray}
&&\nu_\mfM^X(\nu_\mfM^A\diamond \Id_X)(\Id_\mfM\diamond \psi)=\nu_\mfM^A(\nu_\mfM^X\diamond\Id_A)~;\eqlabel{rmr1}\\
&&\nu_\mfM^X(\nu_\mfM^X\diamond \Id_X)=\nu_\mfM^X(\nu_\mfM^A\diamond \Id_X)(\Id_\mfM\diamond \zeta)~;\eqlabel{rmr2}\\
&&\nu_\mfM^X(\nu_\mfM^A\diamond \Id_X)(\Id_\mfM\diamond \sigma)=\Id_\mfM~.\eqlabel{rmr3}
\end{eqnarray}
\end{lemma} 

\begin{proof} 
Since 
$\Upsilon_{\mfM, X}: \mfM\diamond_A(A\ot X)\ra \mfM\diamond X$ is an isomorphism, 
any morphism $\nu_\mfM^\mfC: \mfM\diamond_A(A\ot X)\ra \mfM$ in $\Dc$ is completely determined 
by a morphism $\nu_\mfM^X: \mfM\diamond X\ra \mfM$ in $\Dc$. Then $\nu_\mfM^\mfC$ is right $A$-linear 
if and only if \equref{rmr1} holds. Moreover $\nu_\mfM^\mfC$ is associative if and only if \equref{rmr2} 
is satisfied, and $\nu_\mfM^\mfC$ is unital if and only if \equref{rmr3} is fulfilled. Verification of 
the details is left to the reader.
\end{proof}

This description of right representations of the ring $A\ot X$ will allow us to show that there
is a bijective correspondence between right representations of $A\ot X$ and representations
of the wreath product $A\#_{\psi, \zeta, \sigma}X$, see \thref{5.4}.
Recall from \cite{LackRoss} that if $(\mathbb{A}, s, \psi, \sigma, \zeta)$ is a wreath in a 
$2$-category $\Cc$ then $st$ with 
\[
\xymatrix{
stst~~\ar@2{->}[r]^{1_s\odot \psi\odot 1_t}~~&sstt~~\ar@2{->}[r]^{1_s\odot 1_s\odot \mu}~~&
~~sst\ar@2{->}[r]^{\zeta}~~&~~stt
\ar@2{->}[r]^{1_s\odot \mu}~~&~~st
}
\]
and $
\xymatrix{
1_A\ar@2{->}[r]^\sigma&st
}
$ 
is a monad in ${\cal K}$, called the wreath product of $A$ and $X$.
If ${\cal K}=\Cc$ is a monoidal 
category, and $(A, X)$ is a wreath in $\Cc$ then $A\#_{\psi, \zeta, \sigma}X$, the wreath product of $A$ and $X$, 
is $A\ot X$ with multiplication   
\begin{equation}\eqlabel{multwpr}
\gbeg{4}{6}
\got{1}{A}\got{1}{X}\got{1}{A}\got{1}{X}\gnl
\gcl{1}\gbrc\gcl{1}\gnl
\gmu\gsbox{2}\gnl
\gcn{1}{1}{2}{3}\gvac{1}\gcl{1}\gcl{1}\gnl
\gvac{1}\gmu\gcl{1}\gnl
\gvac{1}\gob{2}{A}\gob{1}{X}
\gend
~~,
\end{equation}
and unit $\sigma: \un{1}\ra A\ot X$.

\begin{remarks}
(i) Observe that Brzezi\'nski products \cite{brz} are particular examples of wreath products. Namely, 
they are wreath products for which the unit morphism $\sigma$ has the form $\un{\eta}_A\ot \iota$, 
for some $\iota: \un{1}\ra X$ morphism in $\Cc$.\\
(ii) Let $A\#_{\psi, \zeta, \sigma}X$ be a wreath product in a monoidal category $\Cc$. It is easy to show that
$\psi$ and $\zeta$ can be recovered from the multiplication
$$
\gbeg{4}{5}
\got{1}{A}\got{1}{X}\got{1}{A}\got{1}{X}\gnl
\gcl{1}\gcl{1}\gcl{1}\gcl{1}\gnl
\gsbox{4}\gnl
\gvac{1}\gcl{1}\gcl{1}\gnl
\gvac{1}\gob{1}{A}\gob{1}{X}
\gend
$$
in $A\#_{\psi, \zeta, \sigma}X$, namely
\begin{equation}\eqlabel{psideltwreath}
\psi=
\gbeg{5}{8}
\gvac{1}\got{1}{X}\got{1}{A}\gnl
\gvac{1}\gcl{1}\gcl{1}\gu{1}\gnl
\gvac{1}\gcl{1}\gcl{1}\gcn{1}{1}{1}{2}\gnl
\gvac{1}\gcl{1}\gcl{1}\gsbox{2}\gnl
\gu{1}\gcl{1}\gmu\gcl{1}\gnl
\gsbox{5}\gnl
\gvac{1}\gcl{1}\gvac{1}\gcl{1}\gnl
\gvac{1}\gob{1}{A}\gvac{1}\gob{1}{X}
\gend
\hspace*{5mm}{\rm and}\hspace*{5mm}
\zeta=
\gbeg{4}{5}
\gvac{1}\got{1}{X}\gvac{1}\got{1}{X}\gnl
\gu{1}\gcl{1}\gu{1}\gcl{1}\gnl
\gsbox{4}\gnl
\gvac{1}\gcl{1}\gcl{1}\gnl
\gvac{1}\gob{1}{A}\gob{1}{X}
\gend~~.
\end{equation}
Furthermore, in $A\#_{\psi, \zeta, \sigma}X$ we have the identities
\begin{equation}\eqlabel{psideltwreathcond}
\gbeg{5}{8}
\got{1}{A}\gvac{2}\got{1}{A}\got{1}{X}\gnl
\gcl{1}\gu{1}\gvac{1}\gcl{1}\gcl{1}\gnl
\gcl{1}\gcn{1}{1}{1}{2}\gvac{1}\gcl{1}\gcl{1}\gnl
\gcl{1}\gsbox{2}\gvac{2}\gcl{1}\gcl{1}\gnl
\gmu\gcl{1}\gcl{1}\gcl{1}\gnl
\gsbox{5}\gnl
\gvac{1}\gcl{1}\gvac{1}\gcl{1}\gnl
\gvac{1}\gob{1}{A}\gvac{1}\gob{1}{X}
\gend
=
\gbeg{3}{3}
\got{1}{A}\got{1}{A}\got{1}{X}\gnl
\gmu\gcl{1}\gnl
\gob{2}{A}\gob{1}{X}
\gend
\hspace*{5mm}{\rm and}\hspace*{5mm}
\gbeg{5}{5}
\got{1}{A}\gvac{1}\got{1}{X}\got{1}{A}\got{1}{X}\gnl
\gcl{1}\gu{1}\gcl{1}\gcl{1}\gcl{1}\gnl
\gcl{1}\gsbox{4}\gnl
\gmu\gvac{1}\gcl{1}\gnl
\gob{2}{A}\gvac{1}\gob{1}{X}
\gend
=
\gbeg{4}{5}
\got{1}{A}\got{1}{X}\got{1}{A}\got{1}{X}\gnl
\gcl{1}\gcl{1}\gcl{1}\gcl{1}\gnl
\gsbox{4}\gnl
\gvac{1}\gcl{1}\gcl{1}\gnl
\gvac{1}\gob{1}{A}\gob{1}{X}
\gend~.
\end{equation}
A long but straightforward computation shows that 
if $A\ot X$ carries an algebra structure in $\Cc$ with unit $\sigma$ then $A\ot X$ is a wreath product 
with the same unit as $A\ot X$ if and only if the two conditions above are satisfied. To this end, we 
define $\psi, \zeta$ as in \equref{psideltwreath} and then show that 
\equref{ta}, \equref{apdf1} and \equref{apdf2} are fulfilled, and that the original 
multiplication of $A\ot X$ coincides with the one on $A\#_{\psi, \zeta, \sigma}X$. 
We leave all these details to the reader.\\
(iii) This characterization of a wreath product algebra allows us to show that 
there exist algebras of the form $A\ot X$ which are not wreath product algebras.\\
Let $H$ be a $k$-bialgebra, let $\mathbb{A}$ be a $H$-bicomodule algebra and 
let ${\cal A}$ be a $H$-bimodule algebra. This means that
$\mathbb{A}$ is an algebra in the monoidal category of bicomodules over $H$,
and that ${\cal A}$ is an algebra in the monoidal category of $H$-bimodules.
Now we consider the right version of the $L$-$R$ smash product 
$\mathbb{A}\nat {\cal A}$, as introduced in \cite[Proposition 2.1]{pvo}.
As a vector space, $\mathbb{A}\nat {\cal A}=\mathbb{A}\ot {\cal A}$,
with multiplication
\[
(u\nat \v)(u'\nat \v')=u_{[0]}u'_{\le 0\ri}\nat (\v\cdot u'_{\le 1\ri})(u_{[-1]}\cdot \v'),
\]
for $u,~u'\in \mathbb{A}$ and $\v,~\v'\in {\cal A}$, and unit $1_\mathbb{A}\nat 1_{\cal A}$.
Then $\mathbb{A}\nat {\cal A}$ is an associative unital $k$-algebra,
and a simple inspection shows that \equref{psideltwreathcond} is not satisfied.
We can conclude that 
$\mathbb{A}\ot {\cal A}$ it is not a wreath product in the category of 
$k$-vector spaces, ${}_k{\cal M}$.
\end{remarks}

Now we can prove the main result of this Section.

\begin{theorem}\thlabel{5.4}
Let $\Dc$ be a right $\Cc$-category and $(A, X)$ a wreath in $\Cc$ such that the $A$-ring $\mfC=A\ot X$ 
is compatible with the $\Cc$-category structure of $\Dc$. Then the category of right representations in $\Dc$ 
over the $A$-coring $\mfC$ is isomorphic to the category of right modules in $\Dc$ over 
the wreath product $A\#_{\psi, \zeta, \sigma}X$.
\end{theorem}

\begin{proof}
We use the characterization of a right representation over $\mfC$ presented in
\leref{5.2}. First observe that a right $\mfC$-module $\mfM$ is
a right module in $\Dc$ over $A\#_{\psi, \zeta, \sigma}X$ with
\[
\gbeg{3}{5}
\got{1}{\mfM}\got{1}{A}\got{1}{X}\gnl
\gcl{1}\gcl{1}\gcl{1}\gnl
\gcl{1}\gcn{1}{1}{2}{1} \gcn{0}{0}{-1}{1}\gnl
\grm\gnl
\gob{1}{\mfM}
\gend
:=
\gbeg{3}{5}
\got{1}{\mfM}\got{1}{A}\got{1}{X}\gnl
\grm\gcl{1}\gnl
\gcl{1}\gvac{1}\gcn{1}{1}{1}{-1}\gnl
\grm\gnl
\gob{1}{\mfM}
\gend~~.
\]
Indeed, we have that
\[
\gbeg{5}{6}
\got{1}{\mfM}\got{1}{A}\got{1}{X}\got{1}{A}\got{1}{X}\gnl
\gcl{1}\gcl{1}\gcl{1}\gcl{2}\gcl{2}\gnl
\gcl{1}\gcn{1}{1}{2}{1} \gcn{0}{0}{-1}{1}\gnl
\grm\gcn{0}{0}{3}{5}\gvac{1}\gcn{1}{1}{2}{-3}\gnl
\grm\gnl
\gob{1}{\mfM}
\gend
=
\gbeg{5}{6}
\got{1}{\mfM}\got{1}{A}\got{1}{X}\got{1}{A}\got{1}{X}\gnl
\grm\gcn{1}{1}{1}{-1}\gcl{1}\gcl{1}\gnl
\grm\gvac{1}\gcn{1}{1}{1}{-3}\gcl{1}\gnl
\grm\gvac{2}\gcn{1}{1}{1}{-5}\gnl
\grm\gnl
\gob{1}{\mfM}
\gend
\equal{\equref{rmr1}}
\gbeg{5}{7}
\got{1}{\mfM}\got{1}{A}\got{1}{X}\got{1}{A}\got{1}{X}\gnl
\grm\gbrc\gcl{1}\gnl
\gcl{1}\gvac{1}\gcn{1}{1}{1}{-1}\gcl{1}\gcl{1}\gnl
\grm\gvac{1}\gcn{1}{1}{1}{-3}\gcl{1}\gnl
\grm\gvac{2}\gcn{1}{1}{1}{-5}\gnl
\grm\gnl
\gob{1}{\mfM}
\gend
\equalupdown{\equref{rmr2}}{\mfM\in \Dc_A}
\gbeg{5}{8}
\got{1}{\mfM}\got{1}{A}\got{1}{X}\got{1}{A}\got{1}{X}\gnl
\gcl{1}\gcl{1}\gbrc\gcl{1}\gnl
\gcl{1}\gmu\gcl{1}\gcl{1}\gnl
\gcl{1}\gvac{1}\gcn{1}{1}{0}{1}\gsbox{2}\gnl
\gcl{1}\gvac{1}\gmu\gcl{1}\gnl
\gcl{1}\gvac{2}\gcn{0}{0}{0}{3}\gcn{1}{1}{1}{-3}\gnl
\grm\gnl
\gob{1}{\mfM}
\gend ~~, 
\]
as needed, cf. \equref{multwpr}. Moreover, it follows from \equref{rmr3}
that this action is unital.\\
Conversely, if $\mfM$ is a right module in $\Dc$ over $A\#_{\psi, \zeta, \sigma}X$ via 
$\nu_\mfM$ then the actions on $\mfM$ defined by 
\[
\nu_\mfM^A:=\nu_\mfM(\Id_\mfM\diamond (\un{m}_A\ot\Id_X)(\Id_A\ot \sigma)): \mfM\diamond A\ra \mfM
\]  
and
\[
\nu_\mfM^X:=\nu_\mfM(\Id_M\diamond \un{\eta}_A\ot\Id_X): \mfM\diamond X\ra \mfM
\]
satisfy (\ref{eq:rmr1}-\ref{eq:rmr3}), so $\mfM$ is a right $\mfC$-module. We leave it
to the reader to verify that these constructions are inverse to each other.
\end{proof}

Now we will focus on the dual situation. We need some preliminary results first.

\begin{lemma}\lelabel{3.20}
Let $\mfM\in \Dc_A$ and let $\rho_\mfM^X: \mfM\ra \mfM\diamond X$ and 
$\rho_\mfM^\mfC: \mfM\ra \mfM\diamond_A\mfC$ be morphisms in $\Dc$ such that 
$\rho_\mfM^X=\Upsilon_{\mfM, X}\rho_\mfM^\mfC$. Then $\rho_\mfM^\mfC$ is right 
$A$-linear if and only if 
\begin{equation}\eqlabel{rAlin}
\rho_\mfM^X\nu_\mfM=(\nu_\mfM\diamond\Id_X)(\Id_{\mfM}\diamond \psi)
(\rho_{\mfM}^X\diamond \Id_A).
\end{equation} 
\end{lemma} 

\begin{proof}
Recall from \cite{par, sch, bc3} that $\mfM\diamond_A\mfC$ has a right module 
structure in $\Dc$ over $A$ given by the unique morphism 
$\nu_{\mfM\diamond_A\mfC}: (\mfM\diamond_A\mfC)\diamond A\ra \mfM\diamond_A\mfC$ 
in $\Dc$ satisfying 
$\nu_{\mfM\diamond_A\mfC}(q^A_{\mfM, \mfC}\diamond\Id_A)=
q^A_{\mfM, \mfC}(\Id_\mfM\diamond \nu_\mfC)$. 
We then have 
\begin{eqnarray*}
&&\hspace*{-1.5cm}
\nu_{\mfM\diamond_A\mfC}(\rho_\mfM^\mfC\diamond\Id_A)
=\nu_{\mfM\diamond_A\mfC}(\Upsilon_{\mfM, X}^{-1}\diamond\Id_A)(\rho_\mfM^X\diamond\Id_A)\\
&=&\nu_{\mfM\diamond_A\mfC}(q^A_{\mfM, \mfC}\diamond\Id_A)
((\Id_\mfM\diamond\un{\eta}_A\ot\Id_X)\diamond\Id_A)(\rho_\mfM^X\diamond\Id_A)\\
&=&q^A_{\mfM, \mfC}(\Id_\mfM\diamond\nu_\mfC)(\Id_\mfM\diamond(\un{\eta}_A\ot\Id_X\ot\Id_A))
(\rho_\mfM^X\diamond\Id_A)\\
&=&q^A_{\mfM, \mfC}(\Id_\mfM\diamond\psi)(\rho_\mfM^X\diamond\Id_A).
\end{eqnarray*}  
Since $\rho_\mfM^\mfC\nu_\mfM=\Upsilon_{\mfM, X}^{-1}\rho_\mfM^X\nu_\mfM$ it 
follows that $\nu_{\mfM\diamond_A\mfC}$ is right $A$-linear if and only if 
\[
\rho_\mfM^X\nu_\mfM=\Upsilon_{\mfM, X}q^A_{\mfM, \mfC}(\Id_\mfM\diamond\psi)
(\rho_\mfM^X\diamond\Id_A).
\]
It is clear that this condition is equivalent to \equref{rAlin}. 
\end{proof} 

\begin{lemma}\lelabel{3.21}
Under the assumptions of \leref{3.20}, assume that $\rho_\mfM^\mfC$ is right 
$A$-linear, so that $\widehat{\rho_\mfM^\mfC}$ is defined, see \seref{2.2}.
Then $\rho_\mfM^\mfC$ is coassociative if and 
only if 
\begin{equation}\eqlabel{c1}
(\rho_\mfM^X\diamond\Id_X)\rho_\mfM^X=
(\nu_\mfM\diamond\Id_X^{\ot 2})(\Id_\mfM\diamond\d_X)\rho_\mfM^X.
\end{equation}  
\end{lemma}

\begin{proof}
Consider $\Gamma'_{\mfM, X, Y}$ and $\Lambda_{\mfM, X}$ as defined in the proof
of \leref{3.15}.

On one hand, we compute
\begin{eqnarray*}
&&\hspace*{-8mm}
\widehat{\rho_\mfM^\mfC}\rho_\mfM^\mfC=
\widehat{\rho_\mfM^\mfC}q^A_{\mfM, \mfC}(\Id_\mfM\diamond\un{\eta}_A\ot\Id_X)\rho_\mfM^X\\
&=&q^A_{\mfM\diamond_A\mfC, \mfC}(\rho_\mfM^\mfC\diamond\Id_\mfC)
(\Id_\mfM\diamond\un{\eta}_A\ot\Id_X)\rho_\mfM^X\\
&=&q^A_{\mfM\diamond_A\mfC, \mfC}(q^A_{\mfM, \mfC}\diamond\Id_\mfC)
((\Id_\mfM\diamond\un{\eta}_A\ot\Id_X)\diamond\Id_\mfC)
(\rho_\mfM^X\diamond\Id_\mfC)(\Id_\mfM\diamond\un{\eta}_A\ot\Id_X)\rho_\mfM^X\\
&=&q^A_{\mfM\diamond_A\mfC, \mfC}(q^A_{\mfM, \mfC}\diamond\Id_\mfC)
((\Id_\mfM\diamond\un{\eta}_A\ot\Id_X)\diamond\Id_\mfC)\\
&&(\Id_\mfM\diamond\Id_X\diamond\un{\eta}_A\ot\Id_X)(\rho_\mfM^X\diamond\Id_X)\rho_\mfM^X\\
&=&\Lambda_{\mfM, X}(\rho_\mfM^X\diamond\Id_X)\rho_\mfM^X
=\Upsilon_{\mfM\diamond_A\mfC, X}^{-1}(\Upsilon_{\mfM, X}^{-1}\diamond\Id_X)
(\rho_\mfM^X\diamond\Id_X)\rho_\mfM^X.
\end{eqnarray*}
On the other hand, we have
\begin{eqnarray*}
&&\hspace*{-2cm}
\Gamma'_{\mfM, \mfC, \mfC}\widetilde{\un{\Delta}_\mfC}\rho_\mfM^\mfC
=\Gamma'_{\mfM, \mfC, \mfC}\widetilde{\un{\Delta}_\mfC}q^A_{\mfM, \mfC}
(\Id_\mfM\diamond\un{\eta}_A\ot\Id_X)\rho_\mfM^X\\
&\equal{\equref{coeqmorph}}&
\Gamma'_{\mfM, \mfC, \mfC}q^A_{\mfM, \mfC\ot_A\mfC}
(\Id_\mfM\diamond\un{\Delta}_\mfC)(\Id_\mfM\diamond\un{\eta}_A\ot\Id_X)\rho_\mfM^X\\
&=&\widehat{q^A_{\mfM, \mfC}}\Theta'^{-1}_{\mfM, \mfC, \mfC}
(\Id_\mfM\diamond q^A_{\mfC, \mfC})(\Id_{\mfM}\diamond\Id_\mfC\ot\un{\eta}_A\ot\Id_X)\\
&&(\Id_\mfM\diamond (\un{m}_A\ot\Id_X^{\ot 2})(\Id_A\ot \d_X)(\un{\eta}_A\ot \Id_X))
\rho_\mfM^X\\
&=&\widehat{q^A_{\mfM, \mfC}}q^A_{\mfM\diamond\mfC, \mfC}
(\Id_\mfM\diamond\Id_\mfC\ot\un{\eta}_A\ot\Id_X)(\Id_\mfM\diamond \d_X)\rho_\mfM^X\\
&=&q^A_{\mfM\diamond_A\mfC, \mfC}(q^A_{\mfM, \mfC}\diamond\Id_\mfC)
(\Id_\mfM\diamond\Id_\mfC\ot\un{\eta}_A\ot\Id_X)(\Id_\mfM\diamond\d_X)\rho_\mfM^X\\
&=&q^A_{\mfM\diamond_A\mfC, \mfC}(\Id_{\mfM\diamond_A\mfC}\diamond\un{\eta}_A\ot\Id_X)
(q^A_{\mfM, \mfC}\diamond\Id_X)(\Id_\mfM\diamond\d_X)\rho_\mfM^X\\
&=&\Upsilon^{-1}_{\mfM\diamond_A\mfC, X}(q^A_{\mfM, \mfC}\diamond\Id_X)
(\Id_\mfM\diamond\d_X)\rho_\mfM^X.
\end{eqnarray*}
It follows that $\rho_\mfM^\mfC$ is coassociative if and only if  
\[
(\rho_\mfM^X\diamond\Id_X)\rho_\mfM^X =
(\Upsilon_{\mfM, X}\diamond\Id_X)(q^A_{\mfM, \mfC}\diamond\Id_X)
(\Id_\mfM\diamond\d_X)\rho_\mfM^X.
\]
Now
\[
(\Upsilon_{\mfM, X}\diamond\Id_X)(q^A_{\mfM, \mfC}\diamond\Id_X)
=\nu_\mfM\diamond\Id_X\ot\Id_X,
\]
so we can conclude that $\rho^\mfC_\mfM$ is coassociative if and only if \equref{c1} holds. 
\end{proof}

The above results lead us to our next definition.

\begin{definition}
Let $(A, X)$ be a cowreath in a monoidal category $\Cc$ and let $\Dc$ be a right $\Cc$-category. A 
right generalized entwined module in $\Dc$ over $(A, X)$ is a right module $\mfM$ 
in $\Dc$ over $A$ together with a morphism 
$\rho_\mfM^X: \mfM\ra \mfM\diamond X$ in $\Dc$ satisfying \equref{rAlin}, \equref{c1} and 
\begin{equation}\eqlabel{c2}
\nu_\mfM(\Id_\mfM\diamond f)\rho_\mfM^X=\Id_\mfM.
\end{equation}
$\Dc(\psi, \d, f)_A^X$ is the category of right generalized entwined modules in $\Dc$ 
over $(A, X)$ with morphisms $\vartheta: \mfM\ra \mfN$ in $\Dc_A$ satisying 
\begin{equation}\eqlabel{cm}
\rho^X_\mfN \vartheta=(\nu_\mfN\diamond\Id_X)(\vartheta\diamond \un{\eta}_A\ot \Id_X).
\end{equation}
\end{definition}

Lemmas \ref{le:3.20} and \ref{le:3.21} provide the following description of
the category corepresentations 
over a coring of the form $A\ot X$.

\begin{theorem}\thlabel{genentcoring}
Let $\Dc$ be a right $\Cc$-category and $(A, X)$ a cowreath in $\Cc$ such that $A$ and $X$ 
are (left) $\Dc$-coflat and left coflat objects of $\Cc$. Then the category $\Dc^{\mfC}$ of right 
corepresentations in $\Dc$ over the $A$-coring $\mfC=A\ot X$, is isomorphic 
to $\Dc(\psi, \delta, f)_A^X$.
\end{theorem}

\begin{proof}
The isomorphism is given the functor ${\cal F}: \Dc(\psi, \d_X, f)_A^X\ra \Dc^{\mfC}$ defined 
as follows. For $\mfM\in \Dc(\psi, \d_X, f)_A^X$, ${\cal F}(\mfM)=\mfM$ as modules in $\Dc$ over $A$, with right $\mfC$-comodule structure 
$\rho_\mfM^\mfC:=\Upsilon_{\mfM, X}^{-1}\rho_\mfM^X$. Note that 
$\Upsilon_\mfM\widetilde{\va_\mfC}\rho_\mfM^\mfC=\Id_\mfM$ if and only if  \equref{c2} is satisfied. 
\end{proof}

\section{Examples of wreaths and cowreaths arising from actions and coactions of Hopf algebras 
and their generalizations}\selabel{exps}
\setcounter{equation}{0}
In this final Section, we discuss a series of examples, coming from
quasi-bialgebras, dual quasi-bialgebras, bialgebroids and weak bialgebras. 
\subsection{Quasi-bialgebras}\selabel{qba}
Our first aim is to provide
examples of (co)wreaths of the form $(A, X)$ with $X$ as a 0-cell in ${\cal T}_A^\#$ 
rather than in ${\cal T}_A$. Such examples can be produced using actions and
coactions of quasi-bialgebras and their duals.
For the definition of a quasi-bialgebra $H$ we invite the reader to consult 
\cite{kas, maj}. We note that we adopt the following convention:
 the tensor components of the reassociator $\Phi$ of $H$ are denoted 
by capital letters,
\[
\Phi=X^1\ot X^2\ot X^3\in H\ot H\ot H
\]
and the components of the inverse $\Phi^{-1}$ are denoted by small letters,
\[
\Phi^{-1}=x^1\ot x^2\ot x^3\in H\ot H\ot H.
\]
Let $H$ be a quasi-bialgebra, and let $(\mfA, \rho, \Phi_\rho)$ be a right $H$-comodule algebra, 
as defined in \cite{hn}. For the right $H$-coaction $\rho$ on $\mfA$, we use the
notation $\rho(\mfa)=\mfa_{\le 0\ri}\ot \mfa_{\le 1\ri}\in \mfA\ot H$.
Let ${\cal A}$ be an $H$-bimodule algebra, that is an algebra in the monoidal category of 
$H$-bimodules. According to \cite{bpvo}, $\mfA\ov{\#} {\cal A}$ is a left $H$-module algebra,
that is an algebra in the monoidal category of left $H$-modules ${}_H{\cal M}$.
The multiplication is given by the formula
\[
(\mfa\ov{\#} \v)(\mfa'\ov{\#} \v')=\mfa \mfa'_{\le 0\ri}\tilde{x}^1_\rho \ov{\#} (\v\cdot \mfa'_{\le 1\ri}\tilde{x}^2_\rho)
(\v'\cd \tilde{x}^3_\rho),
\]
where $\Phi^{-1}_\rho:=\tilde{x}^1_\rho\ot \tilde{x}^2_\rho\ot \tilde{x}^3_\rho$ is the inverse of $\Phi_\rho$ 
in the tensor product algebra $\mfA\ot H\ot H$ and $\cdot$ denotes the right action of $H$ on ${\cal A}$, 
the unit is $1_\mfA\ov{\#} 1_{\cal A}$, and the action of $H$ on $\mfA\ov{\#} {\cal A}$ is given by 
$h\cdot (\mfa\ov{\#} \v)=\mfa \ov{\#}h\cdot \v$.
$\mfA\ov{\#} {\cal A}$ is called the generalized quasi-smash product of $\mfA$ and ${\cal A}$. 
In the case where ${\cal A}=B$ is a right 
$H$-module algebra, considered as an $H$-bimodule algebra with trivial left $H$-action
obtained by restriction of scalars via the counit $\varepsilon$,
$\mfA\ov{\#} B$ reduces to the right generalized smash product of $\mfA$ and $B$, as 
introduced in \cite{bpvo}.

\begin{proposition}\prlabel{wpqHa}
Let $H$ be a quasi-bialgebra, $(\mfA, \rho, \Phi_\rho)$ a right $H$-comodule algebra and 
${\cal A}$ an $H$-bimodule algebra. Then $({\cal A}, \mfA)$ is a wreath in ${}_H{\cal M}$ with
\begin{itemize}
\item[$\bullet$] $\mfA$ considered as a left $H$-module in a trivial way;
\item[$\bullet$] $\psi: {\cal A}\ot \mfA\ra \mfA\ot {\cal A}$, $\psi(\v \ot \mfa)=\mfa_{\le 0\ri}\ot \v\cdot \mfa_{\le 1\ri}$;
\item[$\bullet$] $\zeta: {\cal A}\ot {\cal A}\ra \mfA\ot {\cal A}$, $\zeta(\v \ot \v')=\tilde{x}^1_\rho\ot 
(\v \cdot \tilde{x}^2_\rho)(\v'\cdot \tilde{x}^3_\rho)$; 
\item[$\bullet$] $\sigma: k\ra \mfA\ot {\cal A}$, $\sigma(1)=1_\mfA\ot 1_{\cal A}$.  
\end{itemize}
Furthermore, the resulting wreath product is the generalized quasi-smash product
$\mfA\ov{\#} {\cal A}$. 
\end{proposition}

\begin{proof}
The associativity constraint on ${}_H{\cal M}$ is given by left action by the tensor components of $\Phi$. In several situations below, we have left $H$-modules with trivial action,
and then we can freely omit the parentheses.
With this observation in mind, we have the following.\\
1) \equref{ta} follows from the fact that $\rho$ is an algebra morphism and 
${\cal A}$ is a right $H$-module;\\
2) the first equality in \equref{apdf1} comes out as 
\[
\tilde{x}^1_\rho\mfa_{\le 0\ri}\ot (\v\cdot \tilde{x}^2_\rho\mfa_{\le 1\ri_1})(\v'\cdot \tilde{x}^3_\rho\mfa_{\le 1\ri_2})
=\mfa_{\le 0\ri_{\le 0\ri}}\tilde{x}^1_\rho \ot (\v\cdot a_{\le 0\ri_{\le 1\ri}}\tilde{x}^2_\rho)
(\v'\cdot \mfa_{\le 1\ri}\tilde{x}^3_\rho), 
\]
and follows from the coassociativity of the $H$-coaction $\rho$ on $\mfA$;\\
3) the second equality in \equref{apdf1} reduces to  
\begin{eqnarray*}
&&\hspace*{-1.5cm}
(\tilde{x}^1_\rho)_{\le 0\ri}\tilde{y}^1_\rho\ot (\v \cdot (\tilde{x}^1_\rho)_{\le 1\ri}\tilde{y}^2_\rho)
[(\v'\cdot \tilde{x}^2_\rho(\tilde{y}^3_\rho)_1)(\v{''}\cdot \tilde{x}^3_\rho(\tilde{y}^3_\rho)_2)]\\
&=&\tilde{x}^1_\rho\tilde{y}^1_\rho\ot [(x^1\cdot \v\cdot \tilde{x}^2_\rho(\tilde{y}^2_\rho)_1)(x^2\cdot 
\v'\cdot\tilde{x}^3_\rho(\tilde{y}^3_\rho)_2)](x^3\cdot \v{''}\cdot \tilde{y}^3),
\end{eqnarray*}
where $\tilde{y}^1_\rho\ot \tilde{y}^2_\rho\ot \tilde{y}^3_\rho$ 
is a second copy of $\Phi^{-1}_\rho$.
It follows from the associativity of the multiplication on ${\cal A}$ in ${}_H{\cal M}_H$ (which is modulo  
the conjugation by $\Phi$), and from the $3$-cocycle condition on $\Phi_\rho$;\\
4) \equref{apdf2} is trivially satisfied.\\
Using \equref{multwpr} we can easily verify that the corresponding wreath product 
$\mfA\#_{\psi, \zeta, \sigma}{\cal A}$ is precisely the generalized quasi-Hopf smash product 
$\mfA\ov{\#}{\cal A}$.
\end{proof}

\begin{corollary}
Let $H$ be a quasi-bialgebra, let $(\mfA, \rho, \Phi_\rho)$ be a right $H$-comodule algebra 
and let $B$ be a right $H$-module algebra. Then $(\mfA, B)$ is a wreath in ${}_k{\cal M}$ and 
the corresponding wreath product is the right generalized smash product of 
$\mfA$ and $B$. 
\end{corollary}

\begin{proof}
This follows from the comments preceding \prref{wpqHa}. Observe that the wreath 
structure of $(\mfA, B)$ is precisely as in the statement of \prref{wpqHa}, and that 
the resulting wreath product is a $k$-algebra since the left $H$-action on $B$ is trivial.  
\end{proof}

In \cite[Example 3.3]{LackRoss} it is shown that Sweedler's crossed product of Hopf algebras 
\cite{sw2} is a particular example of a wreath product. For quasi-Hopf algebra we don't have yet 
such a construction. Nevertheless, it can be considered in the dual case \cite{balan}, and as we 
will next see it is a particular case of a wreath product as well. 

We end this subsection with two examples of cowreaths $(A, X)$ for which $X$ has to be considered  
in ${\cal T}_A^\#$. 

\begin{proposition}\prlabel{expnontrivcowr}
Let $H$ be a quasi-bialgebra, let $(\mfA, \rho, \Phi_\rho)$ be a right $H$-comodule algebra and 
let $C$ be an $H$-bimodule coalgebra, that is a coalgebra in the monoidal category of $H$-bimodules.  
Then $(\mfA, C)$ is a cowreath in ${}_H{\cal M}$ via the following structure:

$\bullet$ $\mfA$ is a left $H$-module via the counit of $H$; 

$\bullet$ $\psi: C\ot \mfA\ra \mfA\ot C$, $\psi(c\ot \mfa)=\mfa_{\le 0\ri}\ot c\cdot \mfa_{\le 1\ri}$; 

$\bullet$ $\d: C\ra \mfA\ot C\ot C$, $\d(c)=\tilde{X}^1_\r\ot \una\cdot \tilde{X}^2_\r\ot \unb\cdot \tilde{X}^3_\r$, 
where $\un{\Delta}(c)=\una\ot \unb$ is our Sweedler notation for the comultiplication on $C$; 

$\bullet$ $f: C\ra A$, $f(c)=\une(c)1$, where $\une$ is the counit of $C$.    

The category of right corepresentations of the 
induced $\mfA$-coring $\mfA\ot C$ in ${}_H{\cal M}$ is isomorphic to the category 
${}_H{\cal M}_\mfA^C$ of 
quasi-Hopf $(H, \mfA)$-bimodules over $C$.   
\end{proposition}

\begin{proof}
$\psi$ is left $H$-linear since $C$ is an $H$-bimodule.
(\ref{eq:e1}-\ref{eq:e2}) follow from the fact that $\r: \mfA\ra \mfA\ot H$ is an 
algebra morphism in ${}_k{\cal M}$. The map $\d$ is left $H$-linear because $\un{\Delta}$ is
left $H$-linear and the left $H$-action on $\mfA$ is trivial.

Write $\Phi_\r=\tilde{X}^1_\r\ot \tilde{X}^2_\rho\ot \tilde{X}^3_\r$. The first equation in 
\equref{apdf1} comes out as 
\begin{equation}\eqlabel{rightAlinspecquasi}
\tilde{X}^1_\rho\mfa_{{\le 0\ri}_{\le 0\ri}}\ot \una\cdot \tilde{X}^2_\rho\mfa_{{\le 0\ri}_{\le 1\ri}}\ot 
\unb\cdot \tilde{X}^3_\r\mfa_{\le 1\ri}=
\mfa_{\le 0\ri}\tilde{X}^1_\rho\ot \una\cdot \mfa_{{\le 1\ri}_1}\tilde{X}^2_\rho\ot \unb\cdot 
\mfa_{\le 1\ri}\tilde{X}^3_\rho,
\end{equation}
for all $\mfa\in \mfA$ and $c\in C$, and follows because of the coassociativity of $\rho$. 
 
The second equation in \equref{apdf1} takes the form 
\begin{eqnarray*}
&&\hspace*{-5mm}
\tilde{X}^1_\r\tilde{Y}^1_\r\ot \left(X^1\cdot c_{(\un{1}, \un{1})}\cdot (\tilde{X}^2_\rho)_1\tilde{Y}^2\ot 
(X^2\cdot c_{(\un{1}, \un{2})}\cdot (\tilde{X}^2_\r)_2\tilde{Y}^3_\r\ot 
X^3\cdot \unb\cdot \tilde{X}^3_\rho)\right)\\
&&\hspace*{5mm}=
\tilde{X}^1(\tilde{Y}^1_\rho)_{\le 0\ri}\ot\left( \una\cdot \tilde{X}^2_\r(\tilde{Y}^1_\rho)_{\le 1\ri}\ot (c_{(\un{2}, \un{1})}\cdot 
(\tilde{X}^3_\rho)_1\tilde{Y}^2_\rho\ot c_{(\un{2}, \un{2})}\cdot (\tilde{X}^3_\rho)_2\tilde{Y}^3_\rho)\right),
\end{eqnarray*}
where $\Phi_\r=\tilde{Y}^1_\r\ot \tilde{Y}^2_\r\ot \tilde{Y}^3_\r$ is 
a second copy of $\Phi_\r$, and follows by applying the coassociativity of $\un{\Delta}$ and the 
$3$-cocycle condition on $\Phi_\r$.\\
Finally, it follows from the fact that $\une$ is left $H$-linear that $f$ is a morphism in ${}_H{\cal M}$.  
\equref{apdf2} follows immediately from the normality of $\Phi_\r$ and the counit 
property of $\rho$, we leave the verification of these details to the reader.\\
We now prove the second assertion. The monoidal category ${}_H{\cal M}$
of left modules over a quasi-bialgebra $H$ has coequalizers: the coequalizers
of two parallel morphisms $f,g:\ M\to N$ in ${}_H{\cal M}$ 
is the pair $({\rm Coker}(f - g), q)$, where 
$q: N\ra {\rm Coker}(f - g)$ is the canonical surjection. Let $A$ be an algebra
in ${}_H{\cal M}$, and take $M\in ({}_H{\cal M})_A$ and 
$N\in {}_A({}_H{\cal M})$. The tensor product $M\ot_A N$ in ${}_H{\cal M}$
is the quotient of $M\ot N$ over the subobject of $M\ot N$ spanned 
by the elements of the form 
\[
m\triangleleft a\ot n - X^1\blacktriangleright m\ot (X^2\cdot a)\tr (X^3\blacktriangleright n)
\]     
with $m\in M$, $a\in A$ and $n\in N$. It is also clear that all objects of the category ${}_H{\cal M}$ are left and right  coflat.\\
Note that giving a left $\mfA$-module in ${}_H{\cal M}$ is equivalent to giving a left $H$-module 
$M\in {}_k{\cal M}$ which is also a left $\mfA$-module in ${}_k{\cal M}$, and such that 
$\mfa(hm)=h(\mfa m)$, for all 
$\mfa\in \mfA$, $h\in H$ and $m\in M$.
In a similar way, a right $\mfA$-module in ${}_H{\cal M}$ is a left $H$-module that is
also a 
right $\mfA$-module such that $h(m\mfa)=(hm)\mfa$, for all $\mfa\in \mfA$, $h\in H$ 
and $m\in M$.
It follows that a module with the structure of a left and right $\mfA$-module in ${}_H{\cal M}$
is an $\mfA$-bimodule in ${}_H{\cal M}$ if it is an $\mfA$-bimodule in 
${}_k{\cal M}$. In other words, $M$ is an $\mfA$-bimodule in ${}_H{\cal M}$ if and only if $M$ is an 
$(\mfA\ot H, \mfA)$-bimodule in ${}_k{\cal M}$, where $\mfA\ot H$ has the tensor product 
algebra structure in ${}_k{\cal M}$. It is also easy to see that the 
tensor product over $\mfA$ in ${}_H{\cal M}$ is nothing else then the usual tensor product over $\mfA$ in 
${}_k{\cal M}$ endowed with the $H$-module structure given by the comultiplication $\Delta$ of $H$. 
Summarizing our results, we have that ${}_\mfA({}_H{\cal M})_\mfA\cong {}_{\mfA\ot H}{\cal M}_\mfA$ are monoidal categories with 
tensor product taken over $\mfA$ in ${}_k{\cal M}$.\\
Applying \thref{genentcoring}, we obtain that the categories $({}_H{\cal M})^\mfC$ and 
${}_H{\cal M}(\psi, \d, f)_\mfA^C$ are isomorphic. 
The objects of ${}_H{\cal M}(\psi, \d, f)_\mfA^C$ are $(H, \mfA)$-bimodules $\mfM$ together with a left $H$-linear map 
$\r_\mfM^C: \mfM\ni m\mapsto \mfm_{[0]}\ot \mfm_{[1]}\in \mfM\ot C$ satisfying
(\ref{eq:rAlin}, \ref{eq:c1}) and $\une(\mfm_{[1]})\mfm_{[0]}=\mfm$.
More precisely, the left $H$-linearity means that 
$\r^C_\mfM(h\cdot \mfm)=
h_1\cdot \mfm_{[0]}\ot h_2\cdot \mfm_{[1]}$, for all $h\in H$ and $\mfm\in \mfM$. 
\equref{rAlin} says that    
\[
\rho^C_\mfM(m\cdot \mfa)=\mfm_{[0]}\cdot \mfa_{\le 0\ri}\ot \mfm_{[1]}\cdot \mfa_{\le 1\ri},
\]
for all $\mfm\in \mfM$ and $\mfa\in \mfA$.
Keeping the monoidal structure of ${}_H{\cal M}$ in mind, \equref{c1} reduces to 
\[
X^1\cdot \mfm_{[0]_{[0]}}\ot X^2\cdot \mfm_{[0]_{[1]}}\ot X^3\cdot \mfm_{[1]}=
\mfm_{[0]}\cdot\tilde{X}^1_\rho\ot \mfm_{[1]_{\un{1}}}\cdot \tilde{X}^2_\r\ot \mfm_{[1]_{\un{2}}}\cdot \tilde{X}^3_\r,
\]
for all $\mfm\in \mfM$. 
The morphisms in ${}_H{\cal M}(\psi, \d, f)_\mfA^C$ are the $(H, \mfA)$-bimodule morphisms that are right 
$C$-colinear, so ${}_H{\cal M}(\psi, \d, f)_\mfA^C$ is  the category of quasi-Hopf 
$(H, \mfA)$-bimodules over $C$, ${}_H{\cal M}_\mfA^C$.  
\end{proof}

Every right $H$-module coalgebra is an $H$-bimodule coalgebra,
with trivial left $H$-action. If we apply \prref{expnontrivcowr} to a right
$H$-module coalgebra, then we obtain the following result.

\begin{corollary}
Let $H$ be a quasi-bialgebra, let $(\mfA, \rho, \Phi_\rho)$ be a right $H$-comodule algebra and let $C$ be a right 
$H$-module coalgebra. Then with the same structure as in \prref{expnontrivcowr}, $(\mfA, C)$ 
is a cowreath in ${}_k{\cal M}$. The category of corepresentations over the induced $\mfA$-coring is 
isomorphic to the category of right Doi-Hopf modules ${\cal M}(H)_\mfA^C$ as introduced in \cite{bc}.   
\end{corollary}

\subsection{Dual quasi-Hopf algebras}\selabel{dqba}
It was shown in \cite[Example 3.3]{LackRoss} that the Sweedler crossed
product of Hopf algebras \cite{sw2} is a particular example of wreath product.
An obvious question is to extend this result to the context of quasi-bialgebras.
As far as we know, the Sweedler crossed product construction has not been
extended to quasi-bialgebras. Nevertheless, it has been introduced for
dual quasi-bialgebras by A. B\u alan in \cite{balan}. We will see that B\u alan's
construction is a particular case of a wreath product.\\
For the definition of a dual quasi-bialgebra, 
we refer to \cite{maj}. A dual quasi-bialgebra 
is a coassociative counital $k$-coalgebra $H$ equipped with a unital multiplication 
which is associative up to conjugation by a convolution invertible element 
$\phi: H\ot H\ot H\ra k$, called the reassociator. 
The definition is designed in such a way that ${\cal M}^H$, the category 
of corepresentations of $H$ is monoidal. 

\begin{proposition}\prlabel{dqHcrossprod}
Let $H$ be a dual quasi-Hopf algebra with reassociator $\phi$ and $A$ a $k$-algebra on which 
$H$ acts from the left, and consider it as a right $H$-comodule via the trivial 
coaction $A\ni a\mapsto a\ot 1\in A\ot H$. Consider a $k$-linear map $\tau: H\ot H\ra A$ and the 
following morphisms in ${\cal M}^H$:

$\bullet$ $\psi: H\ot A\ra A\ot  H$, $\psi(h\ot a)=h_1\cdot a\ot h_2$;

$\bullet$ $\zeta: H\ot H\ra A\ot H$, $\zeta(h\ot h')=\tau(h_1\ot h'_2)\ot h_2h'_2$; 

$\bullet$ $\sigma: k\ra A\ot H$, $\sigma(1)=1\ot 1$. 

Then $(A, H)$ is a wreath in ${\cal M}^H$ if and only if 
$(A, \tau)$ is an $H$-crossed system in the sense of \cite{balan}. Furthermore, the corresponding 
wreath product is the crossed product algebra $A\ov{\#}_\tau H$, an algebra in ${\cal M}^H$.   
\end{proposition}  

\begin{proof}
It can be seen easily that $\psi$, $\zeta$ and $\sigma$ are right $H$-colinear. 
The equalities in \equref{ta} hold if and only if 
\[
h\cdot (aa')=(h_1\cdot a)(h_2\cdot a')~~{\rm and}~~h\cdot 1_A=\va(h)1_A
\]
for all $h\in H$ and $a,~a'\in A$, that is, $H$ is measuring $A$. In a similar way, the first equality in \equref{apdf1} takes the form  
\[
[h_1\cdot (h'_1\cdot a)]\tau (h_2\ot h'_2)\ot h_3h'_3=\tau(h_1\ot h'_1)(h_2h'_2\cdot a)\ot h_3h'_3, 
\] 
and is clearly equivalent to 
\[
[h_1\cdot (h'_1\cdot a)]\tau (h_2\ot h'_2)=\tau(h_1\ot h'_1)(h_2h'_2\cdot a),
\] 
for all $h,~h'\in H$ and $a\in A$, the twisted module condition.
The second equality in \equref{apdf1} comes down to
\begin{eqnarray*}
&&\hspace*{-1.5cm}
\left(h_1\cdot \tau(h'_1\ot h^{''}_1)\right)\tau(h_2\ot h'_2h^{''}_2)\ot h_3(h'_3h^{''}_3)\\
&=&\tau(h_1\ot h'_1)\tau(h_2h'_2\ot h^{''}_1)\ot (h_3h'_3)h^{''}_2\phi^{-1}(h_4, h'_4, h^{''}_3).
\end{eqnarray*}
Using the quasi-associativity of $H$, we see that this condition is equivalent to
\[
\left(h_1\cdot \tau(h'_1\ot h^{''}_1)\right)\tau(h_2\ot h'_2h^{''}_2)=
\tau(h_1\ot h'_1)\tau(h_2h'_2\ot h^{''}_1)\phi^{-1}(h_3, h'_3, h^{''}_2),
\]
for all $h, h', h{''}\in H$. This is the cocycle condition in the dual quasi-bialgebra case.
Finally, 
it can be easily checked that \equref{apdf2} is equivalent to 
\[
\tau(1\ot h)=\tau(h\ot 1)=\va(h)1,
\]
for all $h\in H$, which means that $\tau$ is normal.\\
This shows that $(A, H)$ is a wreath in ${\cal M}^H$ 
if and only if $(A, \tau)$ is an $H$-crossed system. Moreover, the multiplication
on the wreath product is given by the formula
\[
(a\ot h)(a'\ot h')=a(h_1\cdot a')\tau(h_2\ot h'_1)\ot h_3h'_2, 
\]
and the unit of the wreath product is $1\ot 1$. 
This is precisely the crossed product algebra $A\ov{\#}_\tau H$ 
in ${\cal M}^H$. Note that $A\ov{\#}_\tau H$ is a right $H$-comodule via 
$a\ov{\#}_\tau h\mapsto a\ov{\#}_\tau h_1\ot h_2$.  
\end{proof}

\subsection{Bialgebroids}\selabel{bialgebroids}
Let $L$ be a $k$-algebra. Throughout this Section, $\mathbb{H}$ is a left $L$-bialgebroid 
with source map $s: L\ra \mathbb{H}$ and target map $t:L^{\rm op}\ra \mathbb{H}$. 
Recall from \cite{bb} that $\mathbb{H}$ considered  
as an $L$-bimodule via $l\cdot h\cdot l'=t(l')s(l)h=s(l)t(l')h$ admits an $L$-coring structure 
$(\mathbb{H}, \widetilde{\Delta}, \widetilde{\va})$ such that 
${\rm Im}(\widetilde{\Delta})\subseteq \mathbb{H}\times_L\mathbb{H}$, where  
\[
\mathbb{H}\times_L\mathbb{H}=\left\{
\sum\limits_ix_i\ot_Ly_i\in \mathbb{H}\ot_L\mathbb{H}\mid 
\sum\limits_ix_it(l)\ot_Ly_i=\sum\limits_ix_i\ot_Ly_is(l)~,~\forall~l\in L
\right\} 
\]
is the Takeuchi product \cite{take}, and
\[
\widetilde{\va}(1_\mathbb{H})=1_L~,~\widetilde{\va}(hh')=\widetilde{\va}(hs(\widetilde{\va}(h'))))=
\widetilde{\va}(ht(\widetilde{\va}(h'))),
\] 
for all $h,~h'\in \mathbb{H}$. Via the component-wise multiplication 
$\mathbb{H}\times_L\mathbb{H}$ is an $L$-ring  with unit $1_\mathbb{H}\ot_L1_\mathbb{H}$.
$\widetilde{\Delta}$ is an algebra map, by definition.\\
We use the notation $\widetilde{\Delta}(h)=h_1\ot_L h_2$. The axioms above
imply that
\begin{eqnarray}
&&\widetilde{\Delta}(s(l))=s(l)\ot_L 1_\mathbb{H}~,~\widetilde{\Delta}(t(l))=1_\mathbb{H}\ot_L t(l)~,\eqlabel{axb1}\\
&&s(\widetilde{\va}(h_1))h_2=h=t(\widetilde{\va}(h_2))h_1~{\rm and}~
h_1t(l)\ot_Lh_2=h_1\ot_Lh_2s(l)~,\eqlabel{axb2}
\end{eqnarray}
for all $l\in L$ and $h\in \mathbb{H}$.
Our main aim is to define a wreath $(\mathbb{H}, B)$ within the monoidal category
$({}_L{\cal M}_L, \ot_L, L)$ of $L$-bimodules 
such that the associated wreath product generalizes the 
Sweedler crossed product \cite{sw2} to bialgebroids. 
We work in a context that is different from the one onsidered in \cite{bb};
however, in the end we obtain the same algebra structure on 
the space $B\ot_L\mathbb{H}$ as in \cite{bb}.\\
Let $i: L\ra B$ be a $k$-algebra morphism; then $B$ is an $L$-ring, and 
$B$ is an $L$-bimodule by restriction of scalars. In a similar way, 
$s:\ L\ra \mathbb{H}$ is a $k$-algebra morphism, making $\mathbb{H}$ into
an $L$-bimodule. This new $L$-bimodule structure is given by the formula
$l \tr h\tl l'=s(l)h s(l')$. Note that it is different from the $L$-bimodule structure
given above.

\begin{definition}
Let $\mathbb{H}$ be a left $L$-bialgebroid, and let $B$ be an $L$-ring.
$\mathbb{H}$ measures $B$ if there exists a 
$k$-linear map $\cdot : \mathbb{H}\ot B\ra B$ satisfying the following conditions, 
for all $l\in L$, $h\in \mathbb{H}$ and $b, b'\in B$.
\begin{eqnarray}
&&~(a)~h\cdot (i(l)b)=(hs(l))\cdot b~,~h\cdot (bi(l))=(ht(l))\cdot b~,\nonumber\\
&&~(b)~h\cdot 1_B=i(\widetilde{\va}(h))~,~1_H\cdot b=b~,~\eqlabel{bialmeas}\\
&&~(c)~h\cdot (bb')=(h_1\cdot b)(h_2\cdot b').\nonumber
\end{eqnarray}
\end{definition}

We present an alternative characterization of measurings.

\begin{proposition}\prlabel{caractmwas}
With notation as above, $\mathbb{H}$ measures $B$ if and only if $1_H\cdot b=b$, for all $b\in B$, 
\equref{bialmeas}(c) holds and 
\begin{equation}\eqlabel{bialmeas2}
h\cdot i(l)=i(\widetilde{\va}(hs(l))=i(\widetilde{\va}(ht(l)),
\end{equation}
for all $l\in L$ and $h\in \mathbb{H}$.
\end{proposition}  

\begin{proof}
The direct implication is immediate, note only that in any left 
bialgebroid $\mathbb{H}$ we have that $\widetilde{\va}(hs(l))=\widetilde{\va}(ht(l))$, for all 
$l\in L$ and $h\in \mathbb{H}$, and so the equalities in \equref{bialmeas2} are not contradictory.
For the converse, observe first that $h\cdot 1_B=h\cdot i(1_L)=i(\widetilde{\va}(h))$, for all $h\in \mathbb{H}$. 
Now, for all $l\in L$, $h\in \mathbb{H}$ and $b\in B$ we have 
\begin{eqnarray*}
h\cdot (i(l)b)&\equal{\equref{bialmeas}(c)}&(h_1\cdot i(l))(h_2\cdot b)\equal{\equref{bialmeas2}}
i(\widetilde{\va}(h_1s(l)))(h_2\cdot b)\\
&=&i(\widetilde{\va}((hs(l))_1))((hs(l))_2\cdot b)
=((hs(l))_1\cdot 1_B)((hs(l))_2\cdot b)\\
&\equal{\equref{bialmeas}(c)}&(hs(l))\cdot b;\\
h\cdot (bi(l))&\equal{\equref{bialmeas}(c)}&(h_1\cdot b)(h_2\cdot i(l))\equal{\equref{bialmeas2}}
(h_1\cdot b)i(\widetilde{\va}(h_2t(l)))\\
&=&((ht(l))_1\cdot b)i(\widetilde{\va}((ht(l))_2))
=((ht(l))_1\cdot b)((ht(l))_2\cdot 1_B)\\
&\equal{\equref{bialmeas}(c)}&(ht(l))\cdot b, 
\end{eqnarray*}
as needed. This completes the proof.
\end{proof}

The definition of a measuring is designed in such a way that we have the following result.

\begin{lemma}\lelabel{psicrossbialg}
Let $B$ be an $L$-ring and $\mathbb{H}$ a left $L$-bialgebroid measuring $B$. Then 
\[
\psi: \mathbb{H}\ot_LB\ra B\ot_L\mathbb{H}~,~\psi(h\ot_Lb)=h_1\cdot b\ot_Lh_2
\]
is a well-defined morphism in ${}_L{\cal M}_L$ satisfying \equref{ta}. 
\end{lemma}

\begin{proof}
The map $\psi$ is well-defined because of the second equality in \equref{axb2} and the first condition 
in \equref{bialmeas}(a). $\psi$ is left $L$-linear because of the second equality in \equref{axb2}.\\
To see that $\psi$ is right $L$-linear we compute, $l\in L$, $h\in \mathbb{H}$, $b\in B$, 
\begin{eqnarray*}
\psi((h\ot_Lb)l)&=&\psi(h\ot_Lbi(l))=h_1\cdot (bi(l))\ot_L h_2=(h_1t(l))\cdot b\ot_L h_2\\
&=&h_1\cdot b\ot_L h_2s(l)=(h_1\cdot b\ot_L h_2)l=\psi(h\ot_L b)l.
\end{eqnarray*}
Finally, \equref{ta} follows from \equref{bialmeas}(c) and 
\begin{eqnarray*}
\psi(h\ot_L1_B)&=&h_1\cdot 1_B\ot_L h_2=i(\widetilde{\va}(h_1))\ot_Lh_2=1_B\ot_L \widetilde{\va}(h_1)\tr h_2\\
&=&1_B\ot_L s(\widetilde{\va}(h_1))h_2=1_B\ot_L h, 
\end{eqnarray*}
as needed. 
\end{proof}

For a left $L$-linear morphism $\tau: \mathbb{H}\ot_L\mathbb{H}\ra B\ot_L\mathbb{H}$,
we have another left $L$-linear morphism $\zeta: \mathbb{H}\ot_L\mathbb{H}\ra B\ot_L\mathbb{H}$,
defined by the formula
\begin{equation}\eqlabel{zetabialgcross}
\zeta(h\ot_Lh')=\tau(h_1\ot_Lh'_1)\ot_L h_2h'_2~,~\forall~h,~h'\in H.
\end{equation}
It is easy to show that $\zeta$ is well-defined and left $L$-linear. 
$\zeta$ is also right $L$-linear if $\tau$ satisfies the extra condition
\begin{equation}\eqlabel{bialgcrossax}
\tau(h\ot_Lh't(l))=\tau(h\ot_Lh's(l)),
\end{equation}  
for all $h, h'\in \mathbb{H}$ and $l\in L$.

\begin{proposition}\prlabel{scrossprodbialg}
Let $B$ be an $L$-coring, $\mathbb{H}$ a left $L$-bialgebroid measuring $B$ and 
$\tau: \mathbb{H}\ot_L\mathbb{H}\ra B$ a left $L$-linear map satisfying \equref{bialgcrossax}. 
Consider $\psi$ as in \leref{psicrossbialg}, $\zeta$ defined by \equref{zetabialgcross} 
and $\sigma: L\ra B\ot_L\mathbb{H}$ given by $\sigma(l)=i(l)\ot_L1_\mathbb{H}$, for all $l\in L$. 
If the equalities  
\begin{eqnarray*}
&&\tau(h_1\ot_Lh'_1)(h_2h'_2\cdot b)=[h_1\cdot (h'_1\cdot b)]\tau(h_2\ot_Lh'_2)~,\\
&&(h_1\cdot \tau(h'_1\ot_L h^{''}_1))\tau(h_2\ot_L h'_2h^{''}_2)=\tau(h_1\ot_Lh'_1)\tau(h_2h'_2\ot_L h^{''})~,\\
&&\tau(h\ot_L1_\mathbb{H})=i(\widetilde{\va}(h))=\tau(1_H\ot_L h)~,
\end{eqnarray*}	   
hold for all $h, h', h^{''}\in \mathbb{H}$ and $b\in B$ then $(B, \mathbb{H})$  is a 
wreath in ${}_L{\cal M}_L$.
Furthermore, if $\tau$ satisfies the condition
\begin{equation}\eqlabel{seconextrontau}
\tau(h_1\ot_Lh'_1)i(\widetilde{\va}(h_2h'_2))=\tau(h\ot_Lh'),
\end{equation}
for all $\forall~h,~h'\in \mathbb{H}$,
then the converse is also true. The corresponding 
wreath product is $B\ot_L\mathbb{H}$ with multiplication 
\[
(b\ot_L h)(b'\ot_L h')=b(h_1\cdot b')\tau(h_2\ot_L h'_1)\ot_L h_3h'_2
\]
and unit $1_B\ot_L1_\mathbb{H}$. It is an $L$-ring via $L\ni l\mapsto i(l)\ot_L1_H=1_B\ot_L s(l)\in B\ot_L\mathbb{H}$, 
and therefore a unital associative $k$-algebra.  
\end{proposition}

\begin{proof}
Similar to the one of \prref{dqHcrossprod}, see also the proof of \cite[Prop. 4.3]{bb}.
\end{proof}

The wreath product obtained in \prref{scrossprodbialg} is called
the crossed product of $B$ and $\mathbb{H}$, and it is denoted as
$B\#_\tau \mathbb{H}$. We will now show that it generalizes the smash product
algebra from \cite{kasz}.\\
Szlach\'anyi \cite{sz} has presented a reformulation of the definition of bialgebroid
in terms of monoidal categories. Let $\mathbb{H}$ be an $L\ot L^{\rm op}$-ring.
Left $L$-bialgebroid structures on $\mathbb{H}$ correspond bijectively to
monoidal structures on ${}_\mathbb{H}{\cal M}$, such that the restriction of scalars functor 
${}_\mathbb{H}{\cal M}\ra {}_L{\cal M}_L$ is strict monoidal.
The monoidal structure on ${}_\mathbb{H}{\cal M}$ is defined by $\ot_L$ with
$h\cdot (m\ot_Ln)=h_1\cdot m\ot_Lh_2\cdot n$,
for all $m\in M\in {}_\mathbb{H}{\cal M}$ and $n\in N\in {}_\mathbb{H}{\cal M}$,
and the unit is $L$ considered as a left $\mathbb{H}$-module via the action 
$h\btr l=\widetilde{\va}(hs(l))=\widetilde{\va}(ht(l))$,
for all $h\in \mathbb{H}$ and $l\in L$.

\begin{corollary}\colabel{smashbialg}
Let $\mathbb{H}$ be a left $L$-bialgebroid and $B$ and algebra in ${}_\mathbb{H}{\cal M}$, that is, 
a left $\mathbb{H}$-module algebroid. Then 
$
\xymatrix{
\mathbb{H}\ot B
\ar[r]^-{\rm can}&\mathbb{H}\ot_L B\ar[r]^-{\cdot}&B
}
$ 
defines a measuring of $\mathbb{H}$ on $B$. Furthermore, the map $\tau: \mathbb{H}\ot_L\mathbb{H}\ra B$ 
given by $\tau(h\ot_Lh')=(hh')\cdot 1_B$, for all $h, h'\in \mathbb{H}$, is well defined, left 
$L$-linear, obeys \equref{bialgcrossax} and is such that 
the induced map $\zeta$ from \equref{zetabialgcross} fulfills all the conditions 
in \prref{scrossprodbialg}. 

Consequently, $(B, \mathbb{H})$ is a wreath in ${}_L{\cal M}_L$ with the corresponding wreath product given by 
\[
(b\ot_Lh)(b\ot_Lh')=b(h_1\cdot b)\ot_Lh_2h',
\] 
for all $b, b'\in B$ and $h, h'\in \mathbb{H}$. 
\end{corollary} 

\begin{proof}
By the strict monoidality of the forgetful functor ${}_\mathbb{H}{\cal M}\ra {}_L{\cal M}_L$, any left $\mathbb{H}$-module 
algebroid $B$ has a canonical $L$-ring structure. Its unit is the map 
$i: L\ni l\mapsto s(l)\cdot 1_B=t(l)\cdot 1_B\in B$. Thus $B$ inherits an $L$-bimodule structure 
from the $\mathbb{H}$-action: $l\cdot b\cdot l'=s(l)t(l')\cdot b=t(l')s(l)\cdot b$, for all $l, l'\in L$ and 
$b\in B$. In addition, the fact that $i: L\ra B$ is left $\mathbb{H}$-linear means that
\[
h\cdot i(l)=i(h\btr l)=i(\widetilde{\va}(h)s(l))=i(\widetilde{\va}(h)t(l)),
\]
for all $h\in \mathbb{H}$ and $l\in L$.
Since $B$ is an algebra in ${}_\mathbb{H}{\cal M}$ it follows that $1_H\cdot b=b$, for all $b\in B$, 
and that \equref{bialmeas}(c) is satisfied. From \prref{caractmwas}, we obtain that $\mathbb{H}$ measures $B$.\\
Consider $\tau: \mathbb{H}\ot_L\mathbb{H}\ra B$, 
$\tau(h\ot_Lh')=hh'\cdot 1_B=i(\widetilde{\va}(hh'))$.
It is easy to see that $\tau$ is well-defined, and left $\mathbb{H}$-linear via $\tr$ defined by $s$.
\equref{seconextrontau} is satisfied since
\[
\tau(h_1\ot_Lh'_1)i(\widetilde{\va}(h_2h'_2))=(h_1h'_1\cdot 1_B)(h_2h'_2\cdot 1_B)=hh'\cdot 1_B=\tau(h\ot_Lh')~,~
\]
for all $h, h'\in \mathbb{H}$. Similar computations guarantee us that the three equalities in the statement 
of \prref{scrossprodbialg} hold, and $(B, \mathbb{H})$ is a wreath product in 
${}_L{\cal M}_L$.\\
We end the proof by noting that the resulting $L$-ring is the so called smash product of $B$ and 
$\mathbb{H}$ introduced in \cite[Def. 2.4]{kasz}.  
\end{proof}

\subsection{Weak bialgebras}\selabel{wba}
Particular examples of left bialgebroids are given by weak bialgebras, see 
\cite{en, bcm}.  Recall from \cite{bns} that a weak bialgebra $H$ 
is a $k$-algebra and a $k$-coalgebra such that the comultiplication $\Delta$ is 
multiplicative, the counit $\va$ respects the units, 
\[
(\Delta\ot\Id_H)(\Delta(1))=(\Delta(1)\ot 1)(1\ot\Delta(1))=(1\ot \Delta(1))(\Delta(1)\ot 1) 
\]
and $\va(hh')=\va(h1_1)\va(1_2h')=\va(h1_2)\va(1_1h')$, for all $h, h'\in H$. 
To a weak bialgebra $H$, we can associate four projections
$\varepsilon_t,\varepsilon_s,\ov{\varepsilon}_t,\ov{\varepsilon}_s:\ H\to H$,
given by the formulas
 $\varepsilon_t(h)=\va(1_1h)1_2$,
$\varepsilon_s(h)=\va(h1_2)1_1$, 
 $\ov{\varepsilon}_t(h)=\va(h1_1)1_2$, and
$\ov{\varepsilon}_s(h)=\va(1_2h)1_1$, 
for all $h\in H$. 
Let  $L={\rm Im}(\varepsilon_t)$ 
and define $s, t: L\ra H$ by 
\[
s(z)=z~,~t(z)=\ov{\varepsilon}_s(z).
\]
Then $H$ is a left $L$-bialgebroid with source and target morphisms $s$ and $t$, 
comultiplication 
$\widetilde{\Delta}: 
\xymatrix{
H\ar[r]^-{\Delta}&H\ot H\ar[r]^-{\rm can}&H\ot_LH
}
$ 
and counit $\widetilde{\va}=\varepsilon_t$, see for example \cite[Prop. 3.1]{bcm}. 
Specializing \prref{scrossprodbialg} to weak Hopf algebras 
we obtain the following result.

\begin{corollary}\colabel{crossprodwHa}
Let $H$ be a weak bialgebra, let $B$ be a $k$-algebra and let $i: L\ra B$ be 
a $k$-algebra map. Assume $B$ is a left $H$-module, the left $H$-action of $h\in H$
on $b\in B$ is denoted by $h\cdot b$. We assume that 
\[
h\cdot i(\varepsilon_t(g))=i(\varepsilon_t(hg))~,~1_H\cdot b=b~{\rm and}~~h\cdot (bb')=(h_1\cdot b)(h_2\cdot b')~,
\]
for all $h, g\in H$ and $b, b'\in B$. $H$ is an $L$-bimodule by restriction of scalars via $s$.
Suppose that we have a well defined left $L$-linear morphism 
$\tau: H\ot_LH\ra B$ such that,  
\begin{eqnarray*}
&&\tau(h\ot_Lh'_1)\va(h'_2g)=\va(h'_1g)\tau (h\ot_Lh'_2),\\
&&\tau(h_1\ot_Lh'_1)(h_2h'_2\cdot b)=[h_1\cdot (h'_1\cdot b)]\tau(h_2\ot_Lh'_2),\\
&&(h_1\cdot \tau(h'_1\ot_L h''_1))\tau(h_2\ot_L h'_2h''_2)=\tau(h_1\ot_Lh'_1)\tau(h_2h'_2\ot_L h''),\\
&&\tau(h\ot_L1)=i(\varepsilon_t(h))=\tau(1\ot_L h),
\end{eqnarray*}
for all $h, h', h'', g\in H$. 
With $\psi$, $\zeta$ and $\sigma$ defined as in \prref{scrossprodbialg},
$(B, H)$ is a wreath in the monoidal category of $L$-bimodules. 
\end{corollary}

\begin{proof}
It has to be shown that
$\tau(h\ot_Lh'1_1)\va(1_2g)=\tau (h\ot_Lh'1_2)\va(1_1g)$
is equivalent to the first of the four conditions imposed on $\tau$. This follows easily from the 
formulas $\varepsilon_s(h_1)\ot h_2=1_1\ot h1_2$,
$\va(\varepsilon_s(h)g)=\va(hg)$,
$h1_1\ot 1_2=h_1\ot \ov{\varepsilon}_t(h_2)$
and $\va(\ov{\varepsilon}_t(h)g)=\va(hg)$, see \cite{bns}.
\end{proof}

We can apply \coref{smashbialg} to the left $L$-bialgebroid associated to
a weak bialgebra $H$. An algebra in ${}_H\Mm$ is a left $H$-module algebra $A$,
as introduced in \cite{bohm}. The associated smash product is
the smash product that was introduced in \cite{nik}. 

\subsection{Doi-Koppinen data over right bialgebroids}\selabel{dkd}
Now we look at the dual situation. More precisely, we will construct a coring in a 
category of bimodules from a Doi-Koppinen datum over a bialgebroid. Since we have to deal 
with right actions and coactions we need to work over a right bialgebroid $\mathbb{H}$ 
over a $k$-algebra $R$. As in the left handed case, $s: R\ra \mathbb{H}$ and 
$t: R^{\rm op}\ra \mathbb{H}$ will be the source and target algebra morphisms, and endow 
$\mathbb{H}$ 
with the $R$-bimodule structure given by $r\cdot h\cdot r':=h\cdot t(r)s(r')=h\cdot s(r')t(r)$. 
Then, by definition, $\mathbb{H}$ is an $R$-coring such that the 
image of the comultiplication 
$\widetilde{\Delta}$ 
is included in 
\[
\mathbb{H}\times_R\mathbb{H}=\left\{\sum\limits_ix_i\ot_Ry_i\in \mathbb{H}\ot_R\mathbb{H}\mid
\sum\limits_is(r)x_i\ot_Ry_i=\sum\limits_ix_i\ot_Rt(r)y_i,~\forall~r\in R
\right\}.
\] 
$\mathbb{H}\times_R\mathbb{H}$ is an $R$-ring under the component-wise
multiplication, with unit $1_\mathbb{H}\ot_R1_\mathbb{H}$. It is required that
$\widetilde{\Delta}$ is an algebra morphism. The counit 
$\widetilde{\va}$ has to respect the unit and has to satisfy the condition
\[
\widetilde{\va}(t(\widetilde{\va}(h))h')=\widetilde{\va}(hh')=\widetilde{\va}(s(\widetilde{\va}(h))h'),
\]
for all $h, h'\in H$. A right corepresentation of a right R-bialgebroid $\mathbb{H}$
is a right $R$-module 
$\mfM$ together with a right $R$-module map $\r_\mfM: \mfM\ra \mfM\ot_R\mathbb{H}$ which is 
coassociative and counital. 
Although $\mfM$ is not a left $R$-module, $\mfM\ot_R \mathbb{H}$ is a left
$R$-module, with left $R$-action given by
$r\cdot (\mfm\ot_R h)=\mfm\ot_Rs(r)h$. By \cite[Prop. 1.1]{basz}, 
$\mfM$ has a unique left $R$-module structure, given by 
\[
r\cdot \mfm:=\mfm_{(0)}\cdot \widetilde{\va}(s(r)\mfm_{(1)}),~
{\rm where}~\rho_\mfM(\mfm)=\mfm_{(0)}\ot_R\mfm_{(1)}\in \mfM\ot_R\mathbb{H},
\]
making $\mfM$ into an $R$-bimodule, and such that $\rho_\mfM$ is an $R$-bimodule map 
and ${\rm Im}(\rho_\mfM)$ is included in the Takeuchi product 
\[
\mfM\times_R\mathbb{H}:=\left\{
\sum\limits_ix_i\ot_Ry_i\in \mfM\ot_R\mathbb{H}\mid 
\sum\limits_ir\cdot x_i\ot_Ry_i=\sum\limits_ix_i\ot_Rt(r)y_i,~\forall~r\in R
\right\}.
\]
This key result allows us to define a monoidal structure on ${\cal M}^\mathbb{H}$,
the category of right $\mathbb{H}$-core-presentations and 
right $\mathbb{H}$-colinear maps.
The tensor product of $\mfM, \mfN\in {\cal M}^\mathbb{H}$ is
$\mfM\ot_R\mfN$, together with the structure map
$\rho_{\mfM\ot_R\mfN}$ given by
$$\rho_{\mfM\ot_R\mfN}(\mfm\ot_R\mfn)=
\mfm_{(0)}\ot_R\mfn_{(0)}\ot_R\mfm_{(1)}\mfn_{(1)}.$$
The unit object is $R$ together with the structure map
$\rho_R$ given by $\rho_R(r)=1_R\ot_R s(r)\in R\ot_R\mathbb{H}$.
A right $\mathbb{H}$-comodule algebra is by definition an algebra in
$ {\cal M}^\mathbb{H}$.\\
${\cal M}_\mathbb{H}$, the category of right $\mathbb{H}$-representations,
is also a monoidal cateogry. The tensor product over $R$ of two right 
$\mathbb{H}$-modules is a right $\mathbb{H}$-module by restriction of scalars via 
$\widetilde{\Delta}$. The unit object is $R$, which is a right $\mathbb{H}$-module via 
\[
r\btl h=\widetilde{\va}(t(r)h)=\widetilde{\va}(s(r)h),
\] 
for all $r\in R$ and $h\in \mathbb{H}$. A coalgebra  in ${\cal M}_\mathbb{H}$
is called a right $\mathbb{H}$-module coalgebra. 
A right $\mathbb{H}$-module coalgebra $C$ is an $R$-coring; 
$C$ is as an $R$-bimodule via $r\cdot c\cdot r'=c\cdot s(r)t(r')=c\cdot t(r)s(r')$.
$C$ is a right $\mathbb{H}$-module, we denote
the action by $\cdot: C\ot_R\mathbb{H}\ra C$. Then 
\[
\Delta_C(c\cdot h)=c_1\cdot h_1\ot_Rc_2\cdot h_2~{\rm and}~\va_C(c\cdot h)=\va_C(c)\btl h,
\]       
for all $c\in C$ and $h\in H$.

\begin{definition} \cite{bcm}
A right Doi-Koppinen datum is a triple $(\mathbb{H}, A, C)$
where $\mathbb{H}$ is a right $R$-bialgebroid, $A$ is a right $\mathbb{H}$-comodule algebra
and $C$ is a right $\mathbb{H}$-module coalgebra.
A right $(\mathbb{H}, A, C)$-module is a vector space $\mfM$ with the following structure:
\begin{itemize}
\item[-] $\mfM$ is a right $A$-module, and therefore a right $R$-module by restriction
of scalars via $i:\ R\to A$, $i(r)=r\cdot 1_A=1_A\cdot r$;
\item[-] $\mfM$ is a right $C$-comodule, with structure map
$\rho^C_\mfM:\ M\to \mfM\ot_RC$, $\rho^C_\mfM(\mfm)= \mfm_{(0)}\ot_R\mfm_{(1)}$;
\item[-] for all $a\in A$ and $\mfm\in \mfM$, we have
$\rho^C_\mfM(\mfm\cdot a)=\mfm_{(0)}\cdot a_{\le 0\ri}\ot_R \mfm_{(1)}\cdot a_{\le 1\ri}$.
\end{itemize}
${\cal DK}(\mathbb{H})_A^C$ is the category with right $(\mathbb{H}, A, C)$-modules
as objects and right $A$-linear right $C$-colinear maps as morphisms.
\end{definition}

Now we will show that the (right-handed version of) \cite[Prop. 4.1]{bcm} is a
consequence of \thref{genentcoring}.

\begin{proposition}\prlabel{DKascomcoring}
For a right Doi-Koppinen datum $(\mathbb{H}, A, C)$, we have the following assertions.
\begin{itemize}
\item[(i)] $(A, C)$ is a cowreath in ${}_R{\cal M}_R$, via
$\psi: C\ot_RA\ra A\ot_RC$, $\psi(c\ot_Ra)=a_{\le 0\ri}\ot_Rc\cdot a_{\le 1\ri}$
and $\zeta: C\ra A\ot_R C\ot_RC$, $\zeta=i\ot_R\Delta_C$, where $i:\ R\to A$ is
as above.
\item[(ii)] The bifunctor $\ot_R$ defines a right ${}_R{\cal M}_R$-category structure on ${\cal M}_R$ and the 
category of right corepresentations in ${\cal M}_R$ over the $A$-coring $\mfC=A\ot C$ 
corresponding to the 
cowreath $(A, C)$ from (i) can be identified with the category of right corepresentations over the 
$A$-coring $\mfC$ viewed now in ${}_k{\cal M}$, and is isomorphic to the category ${\cal DK}(\mathbb{H})_A^C$. 
\end{itemize}
\end{proposition}

\begin{proof}
(i) We first show that $\psi$ is well-defined. For all 
$c\in C$, $r\in R$ and $a\in A$ we have that
\begin{eqnarray*}
&&\hspace*{-1cm}
\psi(c\ot_Rr\cdot a)=\psi(c\ot_R a_{\le 0\ri}\cdot \widetilde{\va}(s(r)a_{\le 1\ri}))
\equal{(*)}a_{\le 0\ri_{\le 0\ri}}\ot_R c\cdot a_{\le 0\ri_{\le 1\ri}}s(\widetilde{\va}(s(r)a_{\le 1\ri}))\\
&=&a_{\le 0\ri}\ot_T c\cdot a_{\le 1\ri_1}s(\widetilde{\va}(s(r)a_{\le 1\ri_2}))
=a_{\le 0\ri}\ot_Rc\cdot (s(r)a_{\le 1\ri})_1s(\widetilde{\va}((s(r)a_{\le 1\ri})_2))\\
&=&a_{\le 0\ri}\ot_Rc\cdot s(r)a_{\le 1\ri}
=a_{\le 0\ri}\ot_R(c\cdot r)\cdot a_{\le 1\ri}=\psi(c\cdot r\ot_R a),  
\end{eqnarray*}
At $(*)$, we used that $\rho_A$ is right $R$-linear.
The left $R$-linearity of $\psi$ follows from the fact that ${\rm Im}(\rho_A)\subseteq A\times_R\mathbb{H}$:
\[
\psi(r\cdot c\ot_R a)=\psi(c\cdot t(r)\ot_Ra)=a_{\le 0\ri}\ot_Rc\cdot t(r)a_{\le 1\ri}=
r\cdot a_{\le 0\ri}\ot_Rc\cdot a_{\le 1\ri}.
\] 
The right $R$-linearity of $\psi$ follows immediately from the
right $R$-linearity of $\rho_A$. All the other conditions that are needed to make
$(A, C, \psi)$ to be a cowreath in ${}_R{\cal M}_R$ 
with $(C, \psi)\in {\cal T}_A$ are straightforward, and are left to the reader.\\
(ii) Clearly $\ot_R: {\cal M}_R\times {}_R{\cal M}_R\ra {\cal M}_R$ yields a right 
${}_R{\cal M}_R$-category structure on ${\cal M}_R$. Furthermore, if $A$ is an algebra in 
${}_R{\cal M}_R$, 
then any right module $\mfM$ in ${\cal M}_R$ over $A$ has the right $R$-module structure 
inherited from the right $A$-action, since $(\mfm\cdot r)\cdot a=\mfm\cdot (r\cdot a)$, for all $\mfm\in \mfM$, 
$r\in R$ and $a\in A$, and therefore $\mfm\cdot r=\mfm\cdot (r\cdot 1_A)=\mfm\cdot i(r)$, for all 
$\mfm\in \mfM$ and $r\in R$. So $\mfM$ is a right $A$-module in ${}_k{\cal M}$, considered 
as a right $R$-module via $i$.\\
In a similar way, if $\mfN$ is an $A$-bimodule in ${}_R{\cal M}_R$ then the $R$-bimodule 
structure on $\mfN$ is induced by the $A$-module structure on $\mfN$, that is, $r\cdot \mfn\cdot r'=i(r)\cdot \mfn\cdot i(r')$, 
for all $r, r'\in R$, $\mfn\in \mfN$. This tells us that the tensor product over $A$ in ${}_R{\cal M}_R$ is 
precisely the tensor product over $A$ in ${}_k{\cal M}$. 
Thus, if $\mfC$ is an $A$-coring in ${}_R{\cal M}_R$ then it is actually an $A$-coring in ${}_k{\cal M}$,
viewed as an $R$-bimodule via $i$, and a right corepresentation in ${\cal M}_R$ over $\mfC$ is a usual 
corepresentation of the $A$-coring $\mfC$ in ${}_k{\cal M}$. In other words, $({\cal M}_R)^\mfC\cong {}_k{\cal M}^\mfC$.\\
Now consider $\mfC=A\ot C$, the $A$-coring in ${}_R{\cal M}_R$ determined by the 
cowreath $(A, C)$ in ${}_R{\cal M}_R$ described in (i). 
From the above comments and \thref{genentcoring} it follows that ${}_k{\cal M}^\mfC\cong 
({\cal M}_R)^\mfC\cong {\cal YD}(\mathbb{H})^C_A$.
\end{proof}

A weak bialgebra is a left bialgebroid. The base $k$-algebra $R$ is the image of the 
idempotent morphism $\varepsilon_s: H\ra H$ defined in \seref{wba}.
The source map $s: R\ra H$ is the inclusion, and $t:\ R\ra H$ is the restriction of
$\ov{\varepsilon}_t$. The comultiplication is defined as in the left-handed case,
and the counit is $\varepsilon_s$.\\
Applying  \prref{DKascomcoring} to the case where $\mathbb{H}$ is a weak bialgebra,
we obtain that the category of right weak Doi-Koppinen modules as defined in
\cite{bohm} is 
isomorphic to the category of corepresentations over a coring in ${}_k{\cal M}$. To this end 
we have to use the right handed version of \cite[Theorem 3.11]{bcm}.\\  
Our theory applies also to Doi-Koppinen data in braided monoidal categories. This will be 
explained in full detail in the forthcoming paper \cite{bc4}.

\section{Appendix}\selabel{Appendix}
\setcounter{equation}{0}
After an earlier version of this paper was circulating we were informed that \thref{4.2} holds in a more 
general setting, see \thref{EMAmain}, leading to a different proof based on $2$-categorical arguments.
We will now explore this idea; it will turn out that \thref{EMAmain} holds, but that we
need \thref{4.2} in the proof, so that it cannot be used to give an alternative proof
of \thref{4.2}.

The starting point for the generalization of \thref{4.2} is an explicit description 
of $EM({\cal K})$ in the case where the $2$-category 
${\cal K}$ admits the Eilenberg-Moore (EM for short) construction for monads.  

If ${\cal K}$ is a $2$-category then by ${\rm Mnd}({\cal K})$ we denote the 
$2$-category of monads, monad morphisms and monad transformations, see \seref{monint} for detail. 
An object $X$ of ${\cal K}$ gives rise to a monad 
$\un{X}=(X, X\stackrel{1_X}{\rightarrow}X, i_X, i_X)$, i.e., to an object $\un{X}$ in 
${\rm Mnd}({\cal K})$. Furthermore, any $1$-cell $X\stackrel{f}{\rightarrow}Y$ 
defines a $1$-cell 
$
\xymatrix{
\un{X}\ar[r]^-{(f, 1_f)}&\un{Y}
}
$ 
in ${\rm Mnd}({\cal K})$ and any $2$-cell $f\stackrel{\rho}{\Rightarrow}g$ in ${\cal K}$ becomes a 
$2$-cell $(f, 1_f)\stackrel{\rho}{\Rightarrow}(g, 1_g)$ in ${\rm Mnd}({\cal K})$. These 
correspondences produce a $2$-functor ${\rm Inc}_{\cal K}: {\cal K}\ra {\rm Mnd}({\cal K})$, 
called the inclusion $2$-functor of ${\cal K}$. Conversely, 
we have the so-called underlying $2$-functor ${\rm Und}_{\cal K}: {\rm Mnd}({\cal K})\ra 
{\cal K}$ that maps $(A, t)$ to $A$, $(f, \psi)$ to $f$, and $\rho$ to $\rho$. 
From \cite[Theorem 1]{rstreet} we know that the underlying $2$-functor is a left $2$-adjoint for 
the inclusion $2$-functor of ${\cal K}$ in ${\rm Mnd}({\cal K})$. For more detail on $2$-functors 
and $2$-adjunctions we invite the reader to consult \cite[Ch. 7]{Borceux}. 
 
\begin{definition}
A $2$-category ${\cal K}$ admits the EM construction for monads if the 
inclusion $2$-functor ${\rm Inc}_{\cal K}$ has a right $2$-adjoint.
\end{definition}

Let us explain this more explicitly. Let $F: {\rm Mnd}({\cal K})\ra {\cal K}$ be a 
right $2$-adjoint functor for ${\rm Inc}_{\cal K}$. If $\epsilon: {\rm Inc}_{\cal K}F\ra 1_{{\rm Mnd}({\cal K})}$ 
is the counit of the $2$-adjunction then for any monad $\mathbb{A}=(A, t, \mu_t, \eta_t)$ in ${\cal K}$ 
we have that $\epsilon_\mathbb{A}=(u^t, \chi^t): \underline{A^t}\ra \mathbb{A}$ is a monad 
morphism, where $A^t=F(\mathbb{A})$ is the so-called EM object of $\mathbb{A}$. 
Moreover, for any monad $\mathbb{A}$ in ${\cal K}$ and any 
object $X$ of ${\cal K}$ we have a category isomorphism, natural in both arguments,  
\begin{equation}\eqlabel{EMA}
\mathfrak{F}_{X, \mathbb{A}}: {\cal K}(X, A^t)\cong {\rm Mnd}({\cal K})(\un{X}, \mathbb{A})
\end{equation}
defined as follows. $\mathfrak{F}_{X, \mathbb{A}}$ sends a $1$-cell $X\stackrel{f}{\rightarrow}A^t$ in ${\cal K}$ 
in the monad morphism $(u^tf, \chi^t\odot 1_f): \un{X}\ra \mathbb{A}$, while a $2$-cell 
$
\xymatrix{
X\rtwocell^{f}_{f'}{\rho}&A^t 
}
$ 
in ${\cal K}$ is mapped by $\mathfrak{F}_{X, \mathbb{A}}$ in the $2$-cell 
$
\xymatrix{
\un{X}\rrtwocell^{(u^tf, \chi^t\odot 1_f)}_{(u^tf', \chi^t\odot 1_{f'})}{\hspace{6mm}1_{u^t}\odot\rho}&&\mathbb{A}
}
$ 
of ${\rm Mnd}({\cal K})$. 

To a monad $\mathbb{A}=(A, t, \mu_t, \eta_t)$ in ${\cal K}$ we can associate a 
monad morphism 
$
\xymatrix{
\un{A}\ar[r]^-{(t, \mu_t)}&\mathbb{A}
}
$. 
Consequently, there exists 
a unique $1$-cell $A\stackrel{v^t}{\rightarrow}A^t$ in ${\cal K}$ such that 
\begin{equation}\eqlabel{EMA1}
(u^tv^t, \chi^t\odot 1_{v^t})=(t, \mu_t).
\end{equation}

In a similar way, an easy computation shows that 
$(tu^t=u^tv^tu^t, \mu_t\odot 1_{u^t})\stackrel{\chi^t}{\Rightarrow}(u^t, \chi^t)$ 
is a $2$-cell in ${\rm Mnd}({\cal K})$, and therefore there exists a unique 
$2$-cell 
$
\xymatrix{
A^t\rtwocell^{v^tu^t}_{1_{A^t}}{\rho^t}&A^t 
} 
$
such that $1_{u^t}\odot \r^t=\chi^t$. By \cite[\S 1]{rstreet}   
$
\xymatrix{
A^t\ar@<1ex>[r]^-{u^t}& 
\ar@<1ex>[l]^-{v^t}A
}
$ 
defines an adjoint pair of $1$-cells, in the sense that the $2$-cells 
$\xymatrix{
1_A\ar@2{->}[r]^-{\eta_t}&u^tv^t=t
}
$ and 
$
\xymatrix{
v^tu^t\ar@2{->}[r]^-{\r^t}&
1_{A^t}
}
$  
satisfy the equalities 
\begin{equation}\eqlabel{EMA3}
(1_{u^t}\odot \r^t)(\eta_t\odot 1_{u^t})=1_{u^t}~~\mbox{and}~~
(\r^t\odot 1_{v^t})(1_{v^t}\odot \eta_t)=1_{v^t}.
\end{equation}
 
For any monad morphism 
$
\xymatrix{
\mathbb{A}=(A, t)\ar[r]^-{(f, \psi)}&\mathbb{B}=(B, s)
}
$ 
denote $F(f, \psi)$ by $\ov{f}$, a $1$-cell between $A^t$ and $B^s$. By the $2$-naturality  
of the counit we have a commutative diagram 
\[
\xymatrix{
\un{A^t}\ar[r]^-{(u^t, \chi^t)}\ar[d]_-{(\ov{f}, 1_{\ov{f}})}&\mathbb{A}\ar[d]^-{(f, \psi)}\\
\un{B^s}\ar[r]_-{(u^s, \chi^s)}&\mathbb{B}
}.
\]
Otherwise stated, $\ov{f}: A^t\ra B^s$ is a $1$-cell in ${\cal K}$ such that 
\begin{equation}\eqlabel{EMA4}
fu^t=u^s\ov{f}~~\mbox{and}~~
\chi^s\odot 1_{\ov{f}}=(1_f\odot \chi^t)(\psi\odot 1_{u^t}).
\end{equation}
Then by \cite[\S 2.2]{LackRoss} and \cite[Theorem 3.10]{powa} to 
give a $1$-cell in ${\rm Mnd}({\cal K})$ is equivalent to give 
a pair $(f, \ov{f})$ of $1$-cells in ${\cal K}$ such that the first equation 
in \equref{EMA4} holds. Note that the converse of this assertion has the 
following meaning. 

\begin{lemma}\lelabel{Ident1cellsEM}
If $\mathbb{A}, \mathbb{B}$ are monads in ${\cal K}$ 
and $(f, \ov{f})$ is a pair of $1$-cells in ${\cal K}$ obeying $fu^t=u^s\ov{f}$ 
then there exists $\psi: sf\Rightarrow ft$ a $2$-cell in ${\cal K}$ such that 
$(f, \psi): \mathbb{A}\ra \mathbb{B}$ is a monad morphism. Moreover, 
$\psi$ satisfies the second equality in \equref{EMA4}. 
\end{lemma}

\begin{proof}
The explicit definition of $\psi$ was given in the proof of \cite[Theorem 3.10]{powa} (as well as 
the fact that $(f, \psi)$ is a monad morphism and $F(f, \psi)=\ov{f}$). 
Namely, $\psi$ is defined by the following vertical composition of $2$-cells,
\begin{equation}\eqlabel{EMA5}
\xymatrix{
\psi: sf\ar@2{->}[r]^-{1_{sf}\odot \eta_t}&sft=sfu^tv^t=su^s\ov{f}v^t\ar@2{->}[rr]^-{\chi^s\odot 1_{\ov{f}v^t}}&&
u^s\ov{f}v^t=fu^tv^t=ft
}.
\end{equation} 
To see that $\psi$ satisfies the second equality in \equref{EMA4} observe first that  
\begin{eqnarray*}
(1_f\odot \chi^t)(\chi^s\odot 1_{\ov{f}v^tu^t})&=&
(1_{fu^t}\odot \r^t)(1_{u^s}\odot \r^s\odot 1_{\ov{f}v^tu^t})\\
&=&1_{u^s}\odot (1_{\ov{f}}\odot \r^t)(\r^s\odot 1_{\ov{f}v^tu^t})\\
&=&1_{u^s}\odot (\r^s\odot 1_{\ov{f}})(1_{v^su^s\ov{f}}\odot \r^t)\\
&=&(\chi^s\odot 1_{\ov{f}})(1_{u^sv^sf}\odot \chi^t)=(\chi^s\odot 1_{\ov{f}})(1_{sf}\odot \chi^t),
\end{eqnarray*}
so that
\begin{eqnarray*}
(1_f\odot \chi^t)(\psi\odot 1_{u^t})&=&(1_f\odot \chi^t)(\chi^s\odot 1_{\ov{f}v^tu^t})(1_{sf}\odot \eta_t\odot 1_{u^t})\\
&=&(\chi^s\odot 1_{\ov{f}})(1_{sf}\odot (1_{u^t}\odot \r^t)(\eta_t\odot 1_{u^t}))=\chi^s\odot 1_{\ov{f}},
\end{eqnarray*}
as stated. 
\end{proof} 

Thus the $1$-cells in ${\rm Mnd}({\cal K})$, and a fortiori the $1$-cells in $EM({\cal K})$, can be identified 
with commutative diagrams of the form
\begin{equation}\eqlabel{EMA6}
\xymatrix{
A^t\ar[r]^{\ov{f}}\ar[d]_-{u^t}&B^s\ar[d]^-{u^s}\\
A\ar[r]_-{f}&B
}.
\end{equation}
It was pointed out in \cite[\S 2.2]{LackRoss} that, via this identification, $2$-cells in $EM({\cal K})$ 
correspond to $2$-cells in ${\cal K}$. We discuss this in more detail in \prref{prEMA}. In the proof, we focus on the parts 
that are missing in \cite{LackRoss}.

\begin{proposition}\prlabel{prEMA}
Let ${\cal K}$ be a $2$-category admitting EM constructions for monads. Then $EM({\cal K})$ is isomorphic
to the following $2$-category:
\begin{itemize}
\item[$\bullet$] 0-cells are monads in ${\cal K}$;
\item[$\bullet$] a $1$-cell from $\mathbb{A}=(A, t)$ to $\mathbb{B}=(B, s)$ is a pair $(f, \ov{f})$ of $1$-cells in 
${\cal K}$ making the diagram \equref{EMA6} commutative;
\item[$\bullet$] a $2$-cell from $(f, \ov{f})$ to $(g, \ov{g})$ is a $2$-cell $\tau: \ov{f}\Rightarrow \ov{g}$ in ${\cal K}$. 
\end{itemize}
\end{proposition} 

\begin{proof}
Let $\rho: (f, \psi)\Rightarrow (g, \phi)$ be a $2$-cell in $EM({\cal K})$, that is, $\rho: f\Rightarrow gt$ 
is a $2$-cell in ${\cal K}$ such that \equref{2cellEM} is satisfied. We identify $(f, \psi)$ with $(f, \ov{f})$ 
and $(g, \phi)$ with $(g, \ov{g})$, as in \leref{Ident1cellsEM}, and we define 
\[
\xymatrix{
\l: u^s\ov{f}=fu^t\ar@2{->}[r]^-{\r\odot 1_{u^t}}&gtu^t\ar@2{->}[r]^-{1_g\odot \chi^t}&gu^t=u^s\ov{g}
}.
\]
We claim that $\l: (u^s\ov{f}, \chi^s\odot 1_{\ov{f}})\Rightarrow (u^s\ov{g}, \chi^s\odot 1_{\ov{g}})$ is a 
monad transformation, that is a $2$-cell in ${\rm Mnd}({\cal K})(\un{A^t}, \mathbb{B})$. Indeed, on one hand we have 
\begin{eqnarray*}
\l (\chi^s\odot 1_{\ov{f}})&=&(1_g\odot \chi^t)(\r\odot 1_{u^t})(\chi^s\odot 1_{\ov{f}})\\
&\equal{\equref{EMA4}}&(1_g\odot \chi^t)(\r\odot 1_{u^t})(1_f\odot \chi^t)(\psi\odot 1_{u^t})\\
&=&(1_g\odot \chi^t(1_t\odot \chi^t))((\rho\odot 1_t)\psi\odot 1_{u^t})\\
&=&(1_g\odot \chi^t)((1_g\odot \mu_t)(\rho\odot 1_t)\psi\odot 1_{u^t})\\
&\equal{\equref{2cellEM}}&(1_g\odot \chi^t)((1_g\odot \mu_t)(\phi\odot 1_t)(1_s\odot \r)\odot 1_{u^t}).
\end{eqnarray*}
On the other hand,
\begin{eqnarray*}
(\chi^s\odot 1_{\ov{g}})(1_s\odot \l)&\equal{\equref{EMA4}}&(1_g\odot \chi^t)(\phi\odot 1_{u^t})
(1_{sg}\odot \chi^t)(1_s\odot \r\odot 1_{u^t})\\
&=&(1_g\odot \chi^t(1_t\odot \chi^t))((\phi\odot 1_t)(1_s\odot \r)\odot 1_{u^t})\\
&=&(1_g\odot \chi^t)((1_g\odot \mu_t)(\phi\odot 1_t)(1_s\odot \r)\odot 1_{u^t}).
\end{eqnarray*}
Consequently $\l (\chi^s\odot 1_{\ov{f}})=(\chi^s\odot 1_{\ov{g}})(1_s\odot \l)$, proving 
that $\l$ is a monad morphism. It follows from \equref{EMA} that there exists a unique 
$2$-cell $\tau_\rho: \ov{f}\Rightarrow \ov{g}$ in ${\cal K}(A^t, B^s)$ such that 
\begin{equation}\eqlabel{EMA8}
1_{u^s}\odot \tau_\rho=(1_g\odot \chi^t)(\r\odot 1_{u^t}).
\end{equation}

Conversely, let $\tau: \ov{f}\Rightarrow \ov{g}$ be a $2$-cell in ${\cal K}$. We claim that   
$\rho_\tau: (f, \psi)\ra (g, \phi)$ defined as the composition 
\[
\xymatrix{
\rho_\tau: f\ar@2{->}[r]^-{1_f\odot \eta_t}&ft=fu^tv^t=u^s\ov{f}v^t\ar@2{->}[rr]^-{1_{u^s}\odot \tau\odot 1_{v^t}}&&
u^s\ov{g}v^t=gu^tv^t=gt
}
\]
is a $2$-cell in $EM({\cal K})$. To this end observe first that the diagram below is commutative 
\[
\xymatrix{
u^s\ov{f}v^tu^t=fu^tv^tu^t\ar@2{->}[r]^-{1_{fu^t}\odot \r^t}\ar@2{->}[d]_-{1_{u^s}\odot \tau\odot 1_{v^tu^t}}&
fu^t=u^s\ov{f}\ar@2{->}[d]^-{1_{u^s}\odot \tau}\\
u^s\ov{g}v^tu^t=gu^tv^tu^t\ar@2{->}[r]_-{1_{gu^t}\odot \r^t}&gu^t=u^s\ov{g}
}.
\]
Since $1_{u^t}\odot \r^t=\chi^t$ it follows that 
\begin{equation}\eqlabel{EMA7}
(1_{u^s}\odot \tau)(1_f\odot \chi^t)=(1_g\odot \chi^t)(1_{u^s}\odot \tau\odot 1_{v^tu^t}).
\end{equation} 
We use now \equref{EMA7} to compute   
\begin{eqnarray*}
&&\hspace*{-2cm}
(1_g\odot \mu_t)(1_{u^s}\odot \tau\odot 1_{v^tt})\equal{\equref{EMA1}}
(1_g\odot \chi^t)(1_{u^s}\odot \tau\odot 1_{v^tu^t})\odot 1_{v^t}\\
&\equal{\equref{EMA7}}&(1_{u^s}\odot \tau)(1_f\odot \chi^t)\odot 1_{v^t}
\equal{\equref{EMA1}}(1_{u^s}\odot \tau\odot 1_{v^t})(1_f\odot \mu_t).
\end{eqnarray*}
Define $\psi: sf\Rightarrow ft$, $\phi: sg\Rightarrow gt$ using \equref{EMA5}; then  
$$
(1_g\odot \mu_t)(\r_\tau\odot 1_t)\psi=(1_g\odot \mu_t)(1_{u^s}\odot \tau\odot 1_{v^tt})(1_f\odot \eta_t\odot 1_t)\psi
=(1_{u^s}\odot \tau\odot 1_{v^t})\psi.
$$
In a similar manner we compute that
\begin{eqnarray*}
&&\hspace*{-2cm}
(1_g\odot \mu_t)(\chi^s\odot 1_{\ov{g}v^tt})=(1_g\odot \chi^t)(\chi^s\odot 1_{\ov{g}v^tu^t})\odot 1_{v^t}\\
&=&(\chi^s\odot 1_{\ov{g}v^t})(1_{sg}\odot \chi^t\odot 1_{v^t})=(\chi^s\odot 1_{\ov{g}v^t})(1_{sg}\odot \mu_t),
\end{eqnarray*} 
and therefore 
\begin{eqnarray*}
&&\hspace*{-2cm}
(1_g\odot \mu_t)(\phi\odot 1_t)(1_s\odot \r_\tau)=(1_g\odot \mu_t)(\chi^s\odot 1_{\ov{g}v^tt})(1_{sg}\odot \eta_t\odot 1_t)
(1_{su^s}\odot \tau\odot 1_{v^t})(1_{sf}\odot \eta_t)\\
&=&((\chi^s\odot 1_{\ov{g}})(1_{su^s}\odot \tau)\odot 1_{v^t})(1_{sf}\odot \eta_t)\\
&=&(1_{u^s}\odot \tau\odot 1_{v^t})(\chi^s\odot 1_{\ov{f}v^t})(1_{sf}\odot \eta_t)=
(1_{u^s}\odot \tau\odot 1_{v^t})\psi.
\end{eqnarray*}
It remains to be shown that the correspondences $\r\mapsto \tau_\r$ and $\tau\mapsto \rho_\tau$ described above 
are inverses. Indeed,
\begin{eqnarray*}
\rho_{\tau_\r}&=&(1_{u^s}\odot \tau_\rho\odot 1_{v^t})(1_f\odot \eta_t)
\equal{\equref{EMA8}}(1_g\odot \chi^t\odot 1_{v^t})(\rho\odot 1_t)(1_f\odot \eta_t)\\
&=&(1_g\odot (\chi^t\odot 1_{v^t})(1_t\odot \eta_t))\rho\equal{\equref{EMA1}}
(1_g\odot \mu_t(1_t\odot \eta_t))\r=\r
\end{eqnarray*}
and  $\tau_{\r_\tau}=\tau$ since 
\begin{eqnarray*}
1_{u^s}\odot \tau_{\r_\tau}&\equal{\equref{EMA8}}&
(1_g\odot \chi^t)(\rho_\tau\odot 1_{u^t})=
(1_g\odot \chi^t)(1_{u^s}\odot \tau\odot 1_{v^tu^t})(1_f\odot \eta_t\odot 1_{u^t})\\
&=&(1_{u^s}\odot (1_{\ov{g}}\odot \r^t)(\tau\odot 1_{v^tu^t}))(1_f\odot \eta_t\odot 1_{u^t})\\
&=&(1_{u^s}\odot \tau)(1_f\odot (1_{u^t}\odot \r^t)(\eta_t\odot 1_{u^t}))
\equal{\equref{EMA3}}1_{u^s}\odot \tau.
\end{eqnarray*}
This completes the proof.
\end{proof}  

\prref{prEMA} allows us to provide an alternative description for 
wreaths and cowreaths in ${\cal K}$. For the proof of \coref{coEMA}, see \cite[p. 257]{LackRoss} 
and \prref{prEMA}. 

\begin{corollary}\colabel{coEMA}
Let ${\cal K}$ be a $2$-category admitting the EM constructions for monads. 

(i) Giving a wreath in ${\cal K}$ is equivalent to giving the following data:
\begin{itemize}
\item[$(i_1)$] A monad $\mathbb{A}=(A, t, \mu, \eta)$ in ${\cal K}$;
\item[$(i_2)$] A pair 
$(A\stackrel{s}{\rightarrow}A, A^t\stackrel{\ov{s}}{\rightarrow}A^t)$ of $1$-cells in ${\cal K}$ satisfying $u^t\ov{s}=su^t$;
\item[$(i_3)$] A monad structure 
$(A^t, A^t\stackrel{\ov{s}}{\rightarrow}A^t, \zeta: \ov{s}~\ov{s}\Rightarrow \ov{s}, \sigma: 1_{A^t}\Rightarrow \ov{s})$ 
on $\ov{s}$ in ${\cal K}$. 
\end{itemize}

(ii) Giving a cowreath in ${\cal K}$ is equivalent to giving the following data:
\begin{itemize}
\item[$(ii_1)$] A monad $\mathbb{A}=(A, t, \mu, \eta)$ in ${\cal K}$;
\item[$(ii_2)$] A pair 
$(A\stackrel{s}{\rightarrow}A, A^t\stackrel{\ov{s}}{\rightarrow}A^t)$ of $1$-cells in ${\cal K}$ satisfying $u^t\ov{s}=su^t$;
\item[$(ii_3)$] A comonad structure $(A^t, A^t\stackrel{\ov{s}}{\rightarrow}A^t, \d: \ov{s}\Rightarrow \ov{s}~\ov{s}, 
\va: \ov{s}\Rightarrow 1_{A^t})$ on $\ov{s}$ in ${\cal K}$. 
\end{itemize}
\end{corollary}    

We are now ready to state and prove the result announced at the beginning of this Appendix. 
From now on, we specialize our results above to the bicategory ${\rm Bim}(\Cc)$.   

Let $M$ be a left module over an algebra $A$ in $\Cc$; $M$ is called
right robust if for any 
$X\in {}_{A}\Cc$ and $Y\in \Cc$ the canonical morphism (see \cite{bc3, sch})
\[
\theta_{M, X, Y}: M\ot_{A}(X\ot Y)\ra (M\ot_{A}X)\ot Y
\]
is an isomorphism in $\Cc$. 

By ${}_A\Cc_A^!$ we denote the full subcategory of ${}_A\Cc_A$ whose objects are right robust. By 
\cite{bc3, sch} the category ${}_A\Cc_A^!$ is monoidal, the monidal structure being similar to that 
of ${}^!_A\Cc_A$. 

\begin{definition}
If $\Cc$ is a monoidal category with coequalizers then ${\rm Bim}(\Cc)$ is the bicategory that has as objects 
coflat algebras in ${\Cc}$, as $1$-cells right robust bimodules and as $2$-cells bimodule morphisms 
in $\Cc$, respectively. The vertical composition of $2$-cells in ${\rm Bim}(\Cc)$ is the morphisms composition in 
$\Cc$ and the horizontal one is given by the monoidal structure of $\Cc$.  
\end{definition}

In the sequel, we will regard ${\rm Bim}(\Cc)$ as a 2-category. Before giving a description of wreaths
and cowreaths, we need a desciption of monads in ${\rm Bim}(\Cc)$.

\begin{lemma}
Giving a monad in ${\rm Bim}(\Cc)$ is equivalent to giving a pair $(A, T)$ of coflat algebras in $\Cc$ 
together with an algebra morphism $i: A\ra T$ in $\Cc$ such that $T$ is right robust when it is 
considered as an $A$-bimodule via $i$.  
\end{lemma}  

\begin{proof}
A monad in ${\rm Bim}(\Cc)$ is defined by the following data: 
\begin{itemize}
\item[$\bullet$] A $0$-cell $A$ in ${\rm Bim}(\Cc)$, that is, a coflat algebra $A$ in $\Cc$;
\item[$\bullet$] A $1$-cell $A\stackrel{T}{\rightarrow}A$ in ${\rm Bim}(\Cc)$, that is, a 
right robust $A$-bimodule $T$ in $\Cc$;
\item[$\bullet$] $2$-cells $TT\stackrel{\mu_T}{\rightarrow}T$ and $1_A\stackrel{\eta_T}{\rightarrow}T$ satisfying 
coherence conditions. Explicitly, we have $A$-bimodule morphisms $\mu_T: T\ot_AT\ra T$ and 
$\eta_T: A\ra T$ such that $(T, \mu_T, \eta_T)$ is an algebra in ${}_A{\cal C}_A^!$.  
\end{itemize}
It is well known that such a data can provide an algebra morphism $i: A\ra T$ in $\Cc$. 
The algebra structure of $T$ in $\Cc$ is given my $\un{m}_T=\mu_Tq^A_{T, T}$ and $\un{\eta}_T=\eta_T\un{\eta}_A$, 
and so $\un{\eta}_T$ becomes an algebra morphism in $\Cc$. Conversely, $T$ is an $A$-bimodule in $\Cc$ 
via the restriction of scalars functor defined by $i$, and its algebra structure in $\Cc$ determines an algebra structure 
in ${}_{A}\Cc_A^!$. We leave the verification of the details to the reader.
\end{proof}

The category ${\rm Bim}(\Cc)$ admits the EM constructions for monads. If $i: A\ra T$ is a 
monad in ${\rm Bim}(\Cc)$ then $A^T=T$ as algebras, i.e., objects in ${\rm Bim}(\Cc)$. 
Furthermore, $T\stackrel{u^T}{\rightarrow}A$ is $u^T=T$ regarded as a $(T, A)$-bimodule in $\Cc$ 
via the multiplication of $T$ and $i$. $\chi^T: Tu^T=T\ot_AT\Rightarrow u^T=T$ is the $(T, A)$-bimodule 
morphism in $\Cc$ uniquely determined by $q^A_{T, T}\chi^T=\un{m}_T$. Finally, $A\stackrel{v^T}{\rightarrow}A^T$ is 
again $T$, viewed now as an $(A, T)$-bimodule in $\Cc$ via $i$ and $\un{m}_T$, while 
$\rho^T: v^Tu^T=T\ot_AT\Rightarrow T$ is again the unique morphism in $\Cc$ determined by $\rho^Tq^A_{T, T}=\un{m}_T$ 
but considered now as a $T$-bimodule morphism in the usual way.   

\begin{lemma}\lelabel{leEMA}
Let $i: A\ra T$ and $j: B\ra S$ be monads in ${\rm Bim}(\Cc)$. Give a monad morphism 
between $A\stackrel{i}{\rightarrow}T$ and $B\stackrel{j}{\rightarrow}S$ is equivalent to 
giving a pair $(X, \psi)$ consisting of a right robust $(A, B)$-bimodule $X$ and an $(A, B)$-bimodule 
morphism $\psi: X\ot_BS\ra T\ot_AX$ in $\Cc$ such that the following diagrams commute
\[
\xymatrix{
X\ot_BS\ot_BS\ar[r]^-{\widehat{\psi}}\ar[d]_-{\widetilde{\un{m}^B_S}}&T\ot_AX\ot_BS\ar[r]^-{\widetilde{\psi}}&T\ot_AT\ot_AX
\ar[d]^-{\widehat{\un{m}^A_T}}\\
X\ot_BS\ar[rr]^-{\psi}&&T\ot_AX
}~,~\hspace*{-2mm}
\xymatrix{
X\ar[r]^-{\Id_X\ot \un{\eta}_S}\ar[d]_-{\un{\eta}_T\ot \Id_X}&X\ot S\ar[r]^-{q^B_{X, S}}&X\ot_BS\ar[d]^-{\psi}\\
T\ot X\ar[rr]^-{q^A_{T, X}}&&T\ot_AX
}.
\]
Here $\un{m}^A_T$ and $\un{m}^B_S$ stand for the multiplication of the algebras 
$T$ and $S$ in ${}_A\Cc_A^!$ and ${}_B\Cc_B^!$, respectively.  

Giving a $2$-cell in $EM({\rm Bim}(\Cc))$ between $(X, \psi)$ and $(Y, \phi)$ is equivalent to 
giving an $(A, B)$-bilinear morphism $\ov{\tau}: X\ra T\ot_AY$ making the diagram below commutative, 
\[
\xymatrix{
X\ot_BS\ar[r]^-{\widehat{\ov{\tau}}}\ar[d]_-{\psi}&T\ot_AY\ot_BS\ar[r]^-{\widetilde{\phi}}&
T\ot_AT\ot_AY\ar[d]^-{\widehat{\un{m}^A_T}}\\
T\ot_AX\ar[r]^-{\widetilde{\ov{\tau}}}&T\ot_AT\ot_AY\ar[r]^-{\widehat{\un{m}^A_T}}&T\ot_AY
}.
\]
\end{lemma}

\begin{proof}
We know that ${\rm Bim}(\Cc)$ admits the $EM$ constructions for monads, hence giving a monad 
morphism from $A\stackrel{i}{\rightarrow}T$ to $B\stackrel{j}{\rightarrow}S$ is equivalent to 
giving a pair $(s: A\ra B, \ov{s}: A^T\ra B^S)$ of $1$-cells in ${\rm Bim}(\Cc)$ such that 
$u^S\ov{s}=su^T$. Clearly, this is equivalent to giving a pair $(X, \ov{X})$ with 
$X\in {}_A\Cc_B^!$ and $\ov{X}\in {}_T\Cc_S^!$ such that $\ov{X}\ot_SS=T\ot_AX$ in ${}_T\Cc_B$. 
This forces $\ov{X}=T\ot_AX$ in ${}_T\Cc_S$, and therefore we should have a right $S$-module structure on 
$T\ot_AX$ when $X$ is an $(A, B)$-bimodule in $\Cc$. We next show that this is equivalent to the existence 
of an $(A, B)$-bimodule morphism $\psi: X\ot_BS\ra T\ot_AX$ in $\Cc$ such that the first two 
diagrams in the statement are commutative. 

Let $\nu^S_{\ov{X}}: (T\ot_AX)\ot S\ra T\ot_AX$ be a right $S$-action on $\ov{X}=T\ot_AX$;
we then define $\psi: X\ot_BS\ra T\ot_AX$ as follows. First consider $\psi_0: X\ot S\ra T\ot_AX$ 
as the composition  
\[
\psi_0=\left(\xymatrix{
X\ot S\ar[rr]^-{\un{\eta}_T\ot\Id_{X\ot S}}&&T\ot X\ot S\ar[rr]^-{q^A_{T, X}\ot \Id_S}&&(T\ot_AX)\ot S\ar[r]^-{\nu^S_{\ov{X}}}&
T\ot_AX
}
\right).
\]
We can easily show that $\psi_0$ behaves well with respect to the universality property of the 
coequalizer
\begin{displaymath}
\xymatrix{
X\ot B\ot S\ar@<-1ex>[rr]_-{\Id_X\ot \un{m}_S(j\ot \Id_S)} 
\ar@<1ex>[rr]^-{{\nu_{X}^B\ot\Id_S}}&&
X\ot S\ar[r]^-{q^B_{X, S}}& X\ot_BS
},
\end{displaymath}  
and so there exists a unique morphism $\psi: X\ot_BS\ra T\ot_AX$ in $\Cc$ such that 
$\psi q^B_{X, S}=\psi_0$. We leave it to the reader to show that $\psi$ is left $A$-linear and right $B$-linear, and that 
it satisfies the two equations in the statement. 

Conversely, if we know $\psi$ then $T\ot_AX$ becomes a right $S$-module via the structure morphism
\[
\xymatrix{
(T\ot_AX)\ot S\ar[r]^-{\theta^{-1}_{T, A, X}}&T\ot_A(X\ot S)\ar[r]^-{\widetilde{q^B_{X, S}}}&T\ot_AX\ot_BS
\ar[r]^-{\widetilde{\psi}}&T\ot_AT\ot_AX\ar[r]^-{\widehat{\un{m}^A_T}}&T\ot_AX
}.
\]
Notice that these two correspondences are the counterparts of the ones defined 
in \leref{TaSchstr}. 

Finally it follows from \prref{prEMA} that giving a $2$-cell in $EM({\rm Bim}(\Cc))$ is equivalent to 
giving a $(T, S)$-bimodule morphism $\tau: T\ot_AX\ra T\ot_AY$ in $\Cc$. Also, it is immediate that 
$\tau$ is completely determined by an $(A, B)$-bilinear morphism $\ov{\tau}: X\ra T\ot_AY$. Since 
$\tau$ can be recovered from $\ov{\tau}$ as $\widehat{\un{m}_T^A}\widetilde{\ov{\tau}}$ 
it follows that $\psi$ is right $S$-linear if and only if the third diagram in the statement 
is commutative. This finishes the proof.   
\end{proof}

The above explicit description of $EM({\rm Bim}(\Cc))$ allows us to prove our final result \thref{EMAmain}.
Although \thref{4.2} follows from \thref{EMAmain}, this does not lead to a new proof, since
the proof of \thref{EMAmain} is based on \prref{5.1}, which is itself based on \thref{4.2}.

\begin{theorem}\thlabel{EMAmain}
Let $\Cc$ be a monoidal category with coequalizers. Then there exists a bijective correspondence 
between the (co)wreath structures in ${\rm Bim}(\Cc)$ and pairs $(A, \mathfrak{X})$ 
consisting of a coflat algebra $A$ in $\Cc$ and a (co)wreath $\mathfrak{X}$ in ${}_A\Cc_A^!$. 
\end{theorem}

\begin{proof}
We specialize \coref{coEMA} to ${\rm Bim}(\Cc)$ and use the descriptions in \leref{leEMA} 
to conclude a (co)wreath in ${\rm Bim}(\Cc)$ consists of a triple 
$(A\stackrel{i}{\rightarrow}T, X, \psi)$, where 
\begin{itemize}
\item[(a)] $A, T$ are coflat algebras in $\Cc$, $i$ is an algebra 
morphism, and $T$ considered as an $A$-bimodule via $i$ is right robust; 
\item[(b)] $X$ is a right robust $A$-bimodule;
\item[(c)] $\psi: X\ot_AT\ra T\ot_AX$ is an $A$-bimodule morphism in $\Cc$ such that $T\ot_AX$ 
becomes a $T$-bimodule in $\Cc$ when it is considered 
as a left $T$-module via the multiplication $\un{m}_T$ 
of $T$ and as a $T$-right module via $\psi$ and $\un{m}_T$; 
\item[(d)] $T\ot_AX$ with the $T$-bimodule structure described in (c) admits a $T$-(co)ring structure in $\Cc$. 
\end{itemize}
The conditions (a-d) can be restated as follows:
\begin{itemize}
\item[(a')] $A$ is a coflat algebra in $\Cc$ and $T\in {}_A\Cc_A^!$ is an algebra;
\item[(b')] $X$ is an object of ${}_A\Cc_A^!$; 
\item[(c')] $\psi: X\ot_AT\ra T\ot_AX$ is a morphism in ${}_A\Cc_A^!$ that endows $T\ot_AX$ with a 
$T$-bimodule structure in ${}_A\Cc_A^!$, providing that $T\ot_AX$ has the left $T$-module structure 
given by the multiplication $\un{m}^A_T$ of the algebra $T$ in  ${}_A\Cc_A^!$;  
\item[(d')] $T\ot_AX$ admits a $T$-(co)ring structure in ${}_A\Cc_A^!$.
\end{itemize}
We now apply \prref{5.1} to the monoidal category ${}_A\Cc_A^!$. Then we find that triples $(A\stackrel{i}{\rightarrow}T, X, \psi)$
satisfying (a'-d') are in bijective correspondence with 
pairs $(A, \mathfrak{X})$ consisting of a coflat algebra $A$ and a (co)wreath $\mathfrak{X}$ 
in ${}_A\Cc_A^!$. This finishes our proof.
\end{proof}

\end{document}